\documentclass[12pt,reqno,oneside]{amsart}
\usepackage{amscd,amssymb,amsmath,amsthm}
\usepackage[dvips]{graphicx}
\usepackage{lipsum}
\usepackage{mathrsfs}
\usepackage{setspace}
\usepackage{amsfonts,lscape}
\newtheorem{thm}{Theorem}[section]
\newtheorem{lem}[thm]{Lemma}
\newtheorem{prop}[thm]{Proposition}
\newtheorem{defn}[thm]{Definition}

\newtheorem{rem}[thm]{\textit{\textrm{Remark}}}
\newtheorem{notation}[thm]{Notation}
\newtheorem{ex}[thm]{\textit{\textrm{Example}}}
\newtheorem{note}{\textit{\textrm{Note}}}
\numberwithin{equation}{section}

\usepackage{hyperref}

\numberwithin{equation}{section}
\newcommand{\NI}{\noindent}
\newcommand{\ds}{\displaystyle}

\newcommand\HUGE{\@setfontsize\Huge{38}{47}}
\begin{document}
\title[Gram matrices  of a Class of Diagram algebras $\ldots$ ] {Gram Matrices and Stirling Numbers of a  Class of Diagram Algebras}

\author{N. Karimilla Bi and M.\,Parvathi$^\dag$}
\maketitle{\small{

\begin{center}
Ramanujan Institute for Advanced Study in Mathematics, \,\\
University  of  Madras,  \\
Chepauk, Chennai -600 005, Tamilnadu, India.\\

{\bf {$^\dag$ E-Mail: sparvathi@hotmail.com}}
\end{center}}

\begin{abstract}
In this paper, we introduce  Gram matrices for the signed partition algebras, the algebra of $\mathbb{Z}_2$-relations and the partition algebras. We prove that the Gram matrix is similar to a matrix which is a direct sum of block submatrices. In this connection, $(s_1, s_2, r_1, r_2, p_1, p_2)$-Stirling numbers of the second kind are introduced and their identities are established. As a consequence, the semisimplicity of a signed partition algebra is established.
\end{abstract}


\quad \qquad \qquad \qquad algebra of  $\mathbb{Z}_2$-relations.


\section{\textbf{INTRODUCTION}}
An extensive study of partition algebras is made by Martin
\cite{PH, PM1, PM2, PM3, PM4, PM5} and these algebras arose
naturally as Potts models in statistical mechanics and in the work of V. Jones \cite{J}.

A new class of algebras, called the signed partition algebras, are introduced in \cite{SP}
which are a generalization of partition algebras and signed Brauer
algebras \cite{PS}. The study of the structure of such finite-dimensional
algebras is important for it may be possible to find presumably new
examples of subfactors of a hyper finite $\Pi_1$-factor along the
lines of \cite{W}.

In this paper, we introduce a new class of matrices $G_{2s_1+s_2}^k$, $ \overrightarrow{G}_{2s_1+s_2}^k$ and $G_s^k$ of $A_k^{\mathbb{Z}_2}(x)$ (the algebra of $\mathbb{Z}_2$-relations),  $\overrightarrow{A}_k^{\mathbb{Z}_2}$ (signed partition algebras) and $A_k(x)$ (partition algebras) respectively which will be called as Gram matrices since by Theorem 3.8 in \cite{GL} the Gram matrices $G^{\lambda, \mu}_{2s_1+s_2}$ associated to the cell modules of $W[(s, (s_1, s_2)), ((\lambda_1, \lambda_2), \mu)]$ (for $\lambda = ([s_1], \Phi)$, $\mu = [s_2]$ if $s_1, s_2 \neq 0; \lambda = (\Phi, \Phi)$, $\mu = [s_2]$ if $s_1 = 0, s_2 \neq 0; \lambda = ([s_1], \Phi)$, $\mu = \Phi$ if $s_1 \neq 0 , s_2 = 0$; $\lambda = (\Phi, \Phi)$, $\mu = \Phi$ if $s_1 = s_2 = 0$, $0 \leq s_1 \leq k, 0 \leq s_2 \leq k$ and $0 \leq s_1+s_2 \leq k$) and $\overrightarrow{G}^{\lambda, \mu}_{2s_1+s_2}$ associated to the cell modules of $\overrightarrow{W}[(s, (s_1, s_2)), ((\lambda_1, \lambda_2), \mu)]$ (for $\lambda = ([s_1], \Phi)$, $\mu = [s_2]$ if $s_1, s_2 \neq 0; \lambda = (\Phi, \Phi)$, $\mu = [s_2]$ if $s_1 = 0, s_2 \neq 0; \lambda = ([s_1], \Phi)$, $\mu = \Phi$ if $s_1 \neq 0 , s_2 = 0$; $\lambda = (\Phi, \Phi)$, $\mu = \Phi$ if $s_1 = s_2 = 0$, $0 \leq s_1 \leq k, 0 \leq s_2 \leq k-1$ and $0 \leq s_1+s_2 \leq k - 1$)  defined in \cite{K1} coincides with the matrices  $G_{2s_1+s_2}^k$ and $\overrightarrow{G}_{2s_1+s_2}^k$ respectively. In this paper, we establish that $G^{k}_{2s_1+s_2}$ and $\overrightarrow{G}^{k}_{2s_1+s_2}$ are similar to matrices  $\widetilde{G}_{2s_1+s_2}^k$ and $\widetilde{\overrightarrow{G}}^k_{2s_1+s_2}$ respectively and each of which is a direct sum of block sub matrices $\widetilde{A}_{2r_1+r_2, 2r_1+r_2}$ and $\widetilde{\overrightarrow{A}}_{2r_1+r_2, 2r_1+r_2}$  of sizes $f^{2r_1+r_2}_{2s_1+s_2}$ and  $\overrightarrow{f}^{2r_1+r_2}_{2s_1+s_2}$ respectively. The diagonal entries of the matrices $\widetilde{A}_{2r_1+r_2, 2r_1+r_2}$ and $\widetilde{\overrightarrow{A}}_{2r_1+r_2, 2r_1+r_2}$ are the same and the diagonal element is a product of $r_1$ quadratic polynomials and $r_2$ linear polynomials which could help in determining the roots of the determinant of the Gram matrix. In this connection, $(s_1, s_2, r_1, r_2, p_1, p_2)$-Stirling numbers of the second kind for the algebra of $\mathbb{Z}_2$-relations and signed partition algebras are introduced and their identities are established. Similarly, we have also established that the Gram matrix $G^k_s$ of a partition algebra is similar to a matrix $\widetilde{G}^k_s$ which is a direct sum of block matrices $\widetilde{A}_{r, r}$ of size $f^r_s.$ The diagonal entries of the matrices $\widetilde{A}_{r, r}$ are the same and the diagonal element is a product of $r$ linear polynomials which could help in determining the roots of the determinant of the Gram matrix. Stirling numbers of second kind corresponding to the partition algebras are also introduced and their identities are established.

Using the cellularity structure defined in \cite{K1}, we show that the algebra of $\mathbb{Z}_2$-relations and signed partition algebras  are  semisimple over $\mathbb{K}(x)$ where $\mathbb{K}$ is  a field of
characteristic zero and $x$ is an indeterminate and it is also semisimple over a field of
characteristic zero except for a finite number of algebraic
elements and we also prove that the algebra of $\mathbb{Z}_2$-relations and the signed partition algebras are quasi-hereditary over a field of
characteristic zero. In particular, if $q$ is an integer $\leq k -2$ and $q$ is a  root of the polynomial $x^2-x-2r', 0 \leq r'\leq k-2$ then the algebras $A_k^{\mathbb{Z}_2}(q)$ and $\overrightarrow{A}_k^{\mathbb{Z}_2}(q)$ are not semisimple.

\section{Preliminaries}
\label{sect:prelims}

\subsection{Partition Algebras}

\quad We recall the definitions in \cite{HA} required in this paper. For $k \in \mathbb{N},$ let $\underline{k} = \{1,2, \cdots, k\}, \underline{k}' = \{1', 2', \cdots, k'\}.$ Let  $R_{\underline{k} \cup \underline{k}'}$ be the set of all partitions  of $\underline{k} \cup \underline{k}'$ or equivalence relation on $\underline{k} \cup \underline{k}'.$ Every equivalence class of $\underline{k} \cup \underline{k}'$ is called as connected component.

Any $d \in R_{\underline{k} \cup \underline{k}'}$ can be represented as a simple graph on two rows of $k$-vertices, one above the other with $k$ vertices in the top row, labeled $1,2,\cdots, k$ left to right and $k$ vertices in the bottom row labeled $1', 2', \cdots, k'$ left to right  with vertex $i$ and vertex $j$ connected by a path if $i$ and $j$ are in the same block of the set partition $d.$ The graph representing $d$ is called $k$-partition diagram and it is not unique. Two $k$-partition diagrams are said to be equivalent if they give rise to the same set partition of $2k$-vertices.

Any connected component $C$ of $d, d \in R_{\underline{k} \cup \underline{k}'}$  containing an element of $\{1, 2, \cdots, k\}$ and an element of $\{1', 2', \cdots, k'\}$ is  called a {\it through class}. Any connected component containing  elements only,  either from $\{1, 2, \cdots, k\}$ or $\{1', 2', \cdots, k'\}$ is called a {\it horizontal edge}.

\NI The number of through classes in $d$ is called {\it propagating number} and it is denoted by $\sharp^p(d).$

We shall define multiplication of two $k$-partition diagrams $d'$ and $d''$ as follows:

\begin{itemize}
  \item Place $d'$ above $d''$
  \item Identify the bottom dots of $d'$ with the top dots of $d''$
  \item $d' \circ d''$ is the resultant diagram obtained by using only the top row of $d'$ and bottom row of $d'',$ replace each connected component which lives entirely in the middle row by the variable $x$. i.e., $d' \circ d'' = x^l d'''$ where $l$ is the number of connected components that lie entirely in the middle row.
\end{itemize}

\NI This product is associative and is independent of the graph we choose to represent the $k$-partition diagram. Let $\mathbb{K}(x)$ be the field and $x$ be an indeterminate. The partition algebra $A_k(x)$ is defined to be the $\mathbb{K}(x)$-span of the $k$-partition diagrams, which is an associative algebra with identity $1$ where

\centerline{\includegraphics[height=2cm, width=7cm]{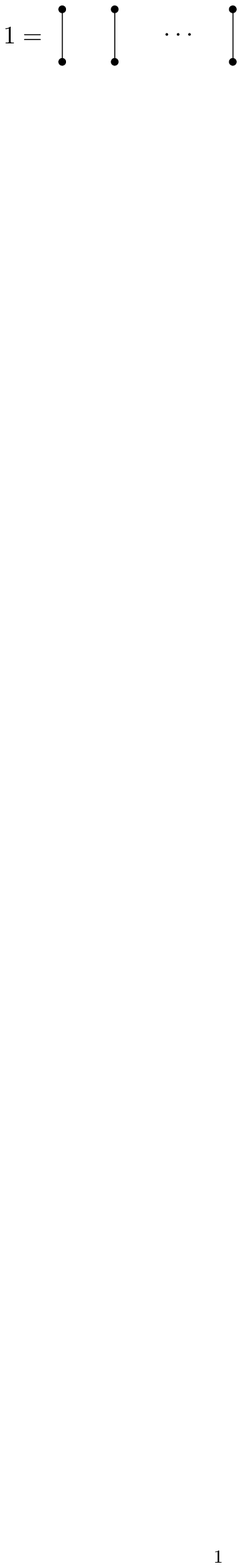}}

\NI By convention $A_0(x) = \mathbb{K}(x).$

\NI For $1 \leq i \leq k-1$ and $1 \leq j \leq k,$ the following are the generators of the partition algebras.
\begin{center}
\includegraphics[height=5cm, width=7cm]{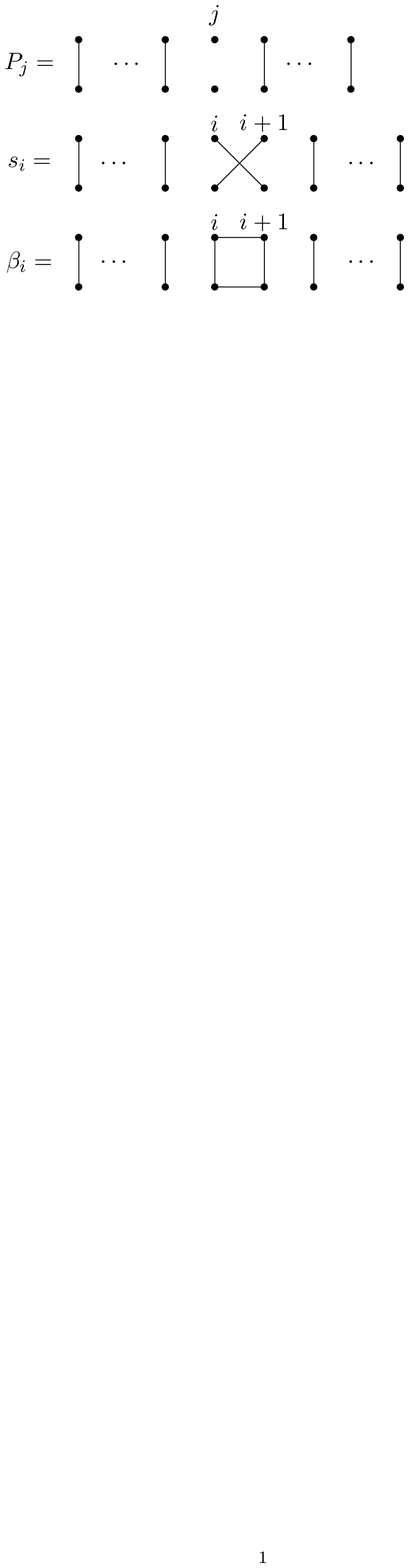}

\end{center}

\NI The above generators satisfy the relations given in Theorem 1.11 of \cite{HA}.

\subsection{The algebra of $\mathbb{Z}_2$-relations:}

\begin{defn}\label{D2.1} (\cite{VSS})
Let the group $\mathbb{Z}_2$ act on the set $X.$ Then the action of $\mathbb{Z}_2$ on $X$ can be extended to an action of $\mathbb{Z}_2$ on $R_X,$ where $R_X$ denote the set of all equivalence relations on $X$, given by

\centerline{$g.d = \{(gp, gq) \ | \ (p, q) \in d\}$}
\NI where $d \in R_X$ and $g \in \mathbb{Z}_2.$ (It is easy to see that the relation $g.d$ is again an equivalence relation).

An equivalence relation $d$ on $X$ is said to be a $\mathbb{Z}_2$-stable equivalence relation if $p \sim q$ in $d$ implies that $gp \sim gq$ in $d$ for all $g$ in $\mathbb{Z}_2.$ We denote $\underline{k}$ for the set $\{1, 2, \cdots, k\}.$ We shall only consider the case when $\mathbb{Z}_2$ acts freely on $X$; let $X = \underline{k} \times \mathbb{Z}_2$ and the action is defined by $g.(i, x) = (i, gx)$ for all $1 \leq i \leq k.$

Let $R_{\underline{k}}^{\mathbb{Z}_2}$ be the set of all $\mathbb{Z}_2$-stable equivalence relations on $\underline{k} \times \mathbb{Z}_2 .$
\end{defn}

\begin{notation}\label{N2.2}(\cite{VSS})
\begin{enumerate}
  \item[(i)]   $R_{\underline{k}}^{\mathbb{Z}_2}$ denotes the set of all $\mathbb{Z}_2$-stable equivalence relation on $ \underline{k} \times \mathbb{Z}_2.$

Each $d \in R_{\underline{k}}^{\mathbb{Z}_2}$ can be represented as a simple graph on row of $2k$ vertices.
\begin{itemize}
  \item The vertices $(1, e), (1, g), \cdots, (k, e), (k, g)$ is arranged from left to right in a single row.
  \item If $(i, g) \sim (j, g') \in R_{\underline{k}}^{\mathbb{Z}_2}$ then $((i, g), (j, g'))$ is an edge which is obtained by joining the vertices $(i, g)$ and $(j, g')$ by a line for $g, g' \in \mathbb{Z}_2.$
\end{itemize}
We say that the two graphs are equivalent if they give rise to the same set partition of the $2k$ vertices $\{(1, e), (1, g), \cdots, (k, e), (k, g)\}.$

\NI We may regard each element $d$ in $R_{\underline{k} \cup \underline{k}'}^{\mathbb{Z}_2}$ as a $2k$-partition diagram by arranging the $4k$ vertices $(i, g), i \in \underline{k} \cup \underline{k}', g \in \mathbb{Z}_2$ of $d$ in two rows in such a way that $(i, g)\left( (i', g)\right)$ is in the top(bottom) row of $d$ if $1 \leq i \leq k(1' \leq i' \leq k')$ for all $g \in \mathbb{Z}_2$  and if $(i, g) \sim (j, g')$ then $((i, g), (j, g'))$ is an edge which is obtained by joining the vertices $(i, g)$ and $(j, g')$ by a line where $g, g' \in \mathbb{Z}_2.$
\item[(ii)] $R_{\underline{k}}^{\mathbb{Z}_2}$ can be identified as a subset of $R_{\underline{2k}}$ by identifying $(r,e)$ with $2r-1, \ \forall \ 1 \leq r \leq k$ and $(r, g), g \neq e$ with $2r \ \forall \ 1 \leq r \leq k.$
\item[(iii)] The diagrams $d^+$ and $d^-$ are obtained from the diagram $d$ by restricting the vertex set to $\{(1, e), (1, g), \\  \cdots, (k, e), (k, g)\}$ and  $\{(1', e), (1', g), \cdots (k', e), (k', g)\}$ respectively. The diagrams $d^+$ and $d^-$ are also $\mathbb{Z}_2$-stable equivalence relation with $d^+ \in R_{\underline{k}}^{\mathbb{Z}_2}$ and $ d^- \in R_{\underline{k}'}^{\mathbb{Z}_2}.$
\end{enumerate}

\end{notation}

\begin{defn}\label{D2.3}(\cite{VSS})
Let $d \in R_{\underline{k} \cup \underline{k}'}^{\mathbb{Z}_2}.$ Then the equation

\centerline{$R^d = \{(i, j) \ | \ \text{ there exists } g, h \in \mathbb{Z}_2 \text{ such that }  ((i, g), (j, h)) \in d\}$}

\NI defines an equivalence relation on $\underline{k} \cup \underline{k}'.$
\end{defn}

\begin{rem}\label{R2.4} (\cite{VSS})
For every connected component $C$ of $R^{\mathbb{Z}_2}_{\underline{k} \cup \underline{k}'}$, $C^d$ will be a connected component in $R^d$ as in Definition \ref{D2.3}.

 For $d \in R_{\underline{k} \cup \underline{k}'}^{\mathbb{Z}_2}$ and for every $\mathbb{Z}_2$-stable equivalence class or a connected component $C$ in $d$ there exists a unique subgroup denoted by $H_C^d$ for $C^d \in R^d$ where
\begin{enumerate}
  \item[(i)]  $H_C^d = \{e\}$ if $(i, e) \not \sim (i, g) \ \forall \ i \in C^d, C$ is called an $\{e\}$-class or $\{e\}$-component and the $\{e\}$-component $C$ will always occur as a pair in $d$ which is denoted by  $C^e, C^g.$
  \item[(ii)] $H_C^d =\mathbb{Z}_2$ if $(i, e) \sim (i, g) \ \forall \ i \in C^d, C$ is called a $\mathbb{Z}_2$-class or $\mathbb{Z}_2$-component which is  denoted  by $C^{\mathbb{Z}_2}$ and the number of vertices in the $\mathbb{Z}_2$-component $C^{\mathbb{Z}_2}$ will always be even.
\end{enumerate}
\end{rem}

\begin{prop}\label{P2.5} (\cite{VSS})
The linear span of $R_{\underline{k} \cup \underline{k}'}^{\mathbb{Z}_2}$ is a subalgebra of $A_{2k}(x)$. We denote this subalgebra by $A_k^{\mathbb{Z}_2}(x),$ called the algebra of $\mathbb{Z}_2$-relations.
\end{prop}

\begin{defn}\label{D2.6}(\cite{VSS})
For $0 \leq 2s_1+s_2 \leq 2k,$ define $I^{2k}_{2s_1+s_2}$ as follows:

\centerline{ $I^{2k}_{2s_1+s_2} = \left\{ d \in R_{\underline{k} \cup \underline{k}'}^{\mathbb{Z}_2} \ | \ \sharp^p(d) = 2s_1+s_2 \right\}$}

\NI i.e., $d$ has $s_1$ number of pairs of $\{e\}$-through classes and $s_2$ number of $\mathbb{Z}_2$-through classes.

\NI For $0 \leq s \leq 2k$ define,  $I_s^{2k} = \underset{2s_1+s_2 \leq s}{\cup} I^{2k}_{2s_1+s_2}$ then it  is clear that
\begin{displaymath}
 R_{\underline{k} \cup \underline{k}'}^{\mathbb{Z}_2} =  \underset{0 \leq s \leq 2k}{\cup} I_s^{2k} =  \underset{\ds  0 \leq 2s_1+s_2 \leq 2k}{\cup} I^{2k}_{2s_1+s_2}.
 \end{displaymath}

\end{defn}

\subsection{Signed Partition Algebras:}
\begin{defn}\label{D2.7} (\cite{SP}, \textbf{Definition 3.1.1})
Let the signed partition algebra $\overrightarrow{A}_k^{\mathbb{Z}_2}(x)$ be the subalgebra of $A_{2k}(x)$ generated by $H_1, F'_i, F''_i, G_i, F_j$ for $1 \leq i \leq k-1$ and $1 \leq j \leq k$ where
\vspace{0.5cm}
\begin{center}
\includegraphics[height=7cm, width=15cm]{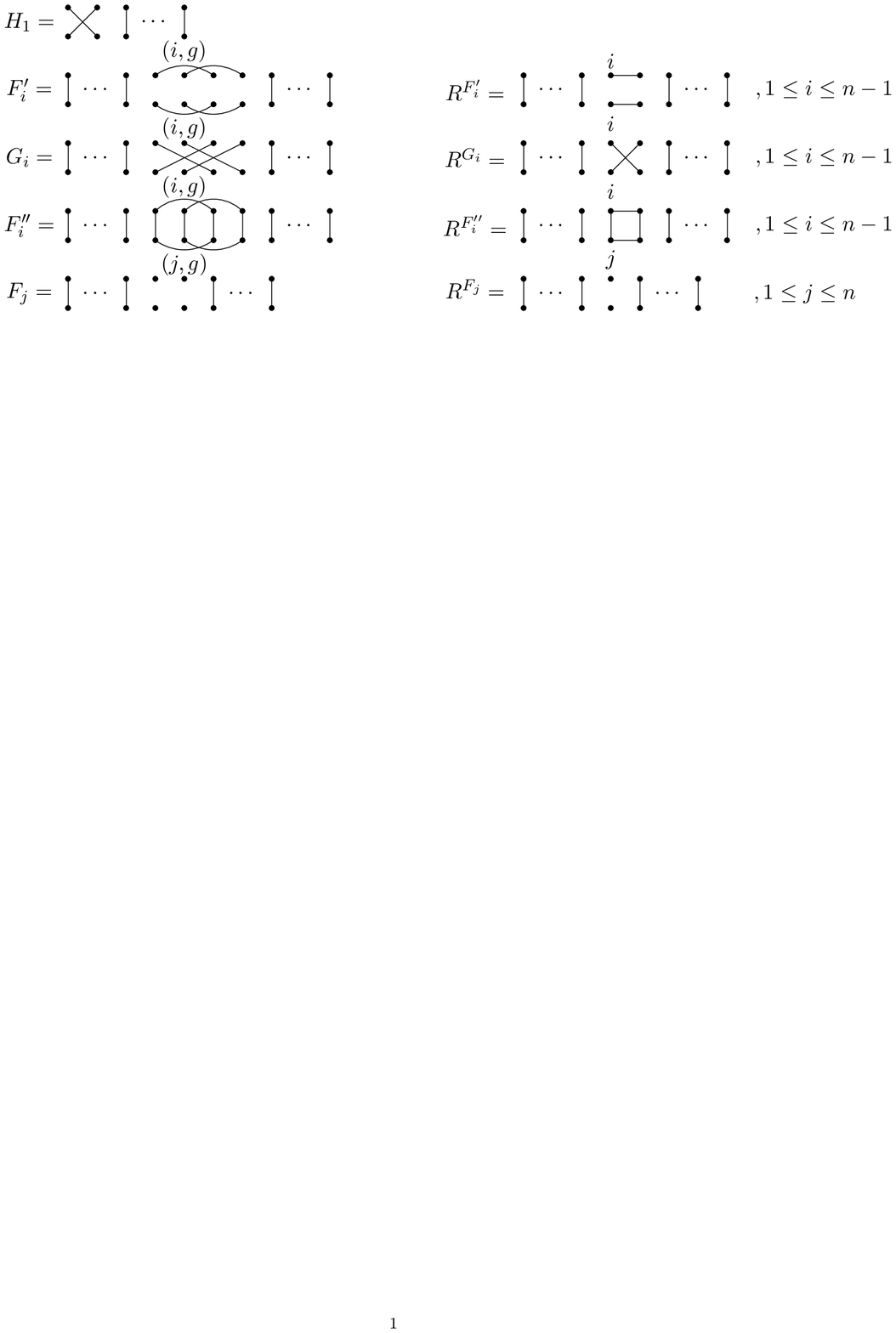}
\end{center}
\vspace{-0.5cm}
The subalgebra of the signed partition algebra generated by $F'_i, G_i, F''_i, F_j, 1 \leq i \leq k -1, 1 \leq j \leq k$ is isomorphic on to the partition algebra $A_{2k}(x^2).$

Also,  $ R^{G_i} = s_i, R^{F''_i} = \beta_i,  R^{F_j} = p_j, R^{F'_i} = p_i p_{i+1} \beta_i p_{i+1} p_i$ where $s_i, \beta_i, p_j$ are as in Page 3.
\end{defn}

We will obtain a basis for the signed partition algebra defined in Definition \ref{D2.7}.
\begin{defn}\label{D2.8}(\cite{SP}, \textbf{Definition 3.1.2})
Let $d \in R_{\underline{k} \cup \underline{k}'}^{\mathbb{Z}_2}$. For $0 \leq 2s_1+s_2 \leq 2k -1$ and $0 \leq s_1, s_2 \leq k - 1,$ define

$\overrightarrow{I}^{2k}_{2s_1+s_2} = \left\{ d \in I_{2s_1+s_2}^{2k} \left|
                                                                  \begin{array}{ll}
                                                                    (i) & s_1+s_2+r_1+r_2 \leq k-1 \text{ and } s_1+s_2+r'_1+r'_2 \leq k-1, \text{ or }  \\
                                                                    (ii) & s_1+s_2+r_1+r_2 \leq k \text{ and } s_1+s_2+r'_1+r'_2 \leq k-1 \text{ then } r_1 \neq 0, \text{ or } \\
                                                                    (iii) & s_1+s_2+r_1+r_2 \leq k-1 \text{ and } s_1+s_2+r'_1+r'_2 \leq k \text{ then } r'_1 \neq 0, \text{ or } \\
                                                                   (iv) & s_1+s_2+r_1+r_2 \leq k \text{ and } s_1+s_2+r'_1+r'_2 \leq k \text{ then } r_1 \neq 0 \text{ and } r'_1 \neq 0 .
                                                                  \end{array}
                                                                \right.
 \right\}$,
\NI  where
\begin{enumerate}
  \item[(a)] $s_1 = \natural \{(C^e, C^g) : C^d$ is a through class of $R^d$ and $H_C^d = \{e\} \},$
  \item[(b)] $s_2 = \natural \{C^{\mathbb{Z}_2} : C^d$ is a through class of $R^d$ and $H_C^d = \mathbb{Z}_2 \},$
  \item[(c)] $r_1 \left( r'_1\right)$ is the number of horizontal edges $C^d$ in the top(bottom) row of $R^d$ such that $H_C^d = \{e\}$
  \item[(d)] $r_2 \left( r'_2\right)$ is the number of horizontal edges $C^d$ in the top(bottom) row of $R^d$ such that $H_C^d = \mathbb{Z}_2$
  \item[(e)] $\sharp^p \left( R^d \right) = s_1 + s_2.$
\end{enumerate}

Also, $\overrightarrow{I}_{2k}^{2k} = I^{2k}_{2k}.$

\NI For $0 \leq s \leq 2k$, put $\overrightarrow{I}_s^{2k} = \underset{2s_1+s_2 \leq s}{\cup} \overrightarrow{I}^{2k}_{2s_1+s_2}.$
\end{defn}

\begin{prop}\label{P2.9}
\mbox{ }
\begin{enumerate}
  \item The \textit{ linear span of $\overrightarrow{I}_s^{2k}, 0 \leq s \leq 2k$ is  the signed partition algebra $\overrightarrow{A}_k^{\mathbb{Z}_2}.$}
  \item The linear span of $I_s^{2k}$ is an ideal of $\overrightarrow{A}_k^{\mathbb{Z}_2}.$
\end{enumerate}
\end{prop}

\begin{rem}\label{R2.10}
The algebra generated by $\{R^{F'_i}, R^{G_i}, R^{F''_i}, R^{F_j}\}_{\substack{1 \leq i \leq k-1 \\ \hspace{-0.4cm}1 \leq j \leq k}}$ is isomorphic to the partition algebra $A_k(x).$

\NI Also, let $I_s^k$ be the set of all $k$-partition diagrams $R^d$ in $A_k(x)$ such that $\sharp^p \left(R^d \right) \leq s$ where $d \in I^{2k}_{2s_1+0} \subseteq A_{2k}(x^2).$
\end{rem}

\begin{defn}\textbf{(\cite{K1}, Definition 4.2)} \label{D2.11}

Define,
\begin{enumerate}
  \item[(i)] $M^k[(s,(s_1, s_2))] = \Big\{(d, P) \ | \ d \in R_{\underline{k}}^{\mathbb{Z}_2}, P \in R_{\underline{s_1+s_2}}^{\mathbb{Z}_2} \text{ and } d \setminus P \in R_{\underline{k-s_1-s_2}}^{\mathbb{Z}_2}, |d| \geq 2s_1 + s_2, P \text{ is a }$

       \NI $\text{ $\mathbb{Z}_2$-stable subset of } d \text{ with } |P| = s \text{ where } s = 2s_1 + s_2, P = \underset{i=1}{\overset{s_1}{\cup}} (P_i^e \cup P_i^g) \underset{j=1}{\overset{s_2}{\cup}}P_{j}^{\mathbb{Z}_2}$ such that

       \NI  $H^d_{P_i^{\{e\}}} = \{e\}, 1 \leq i \leq s_1,$  $H^d_{P_j^{\mathbb{Z}_2}} = \mathbb{Z}_2, 1 \leq j \leq s_2 \Big\}.$

  \item[(ii)] $\overrightarrow{M}^k[(s,(s_1, s_2))] = \Big\{(d, P) \in M^k[(s, (s_1, s_2))] \ \Big| \  s_1+s_2+r_1+r_2 \leq k - 1 \text{ and if } s_1+s_2+r_1+r_2 = k \text{ then } s_1 = k \text{ or } r_1 \neq 0 \text{ where } 2r_1 \text{ is the number of } \{e\}-\text{ connected components in } d \setminus P \text{ and }r_2 \text{ is the number of } \mathbb{Z}_2-\text{connected components} \text{ in } d \setminus P  \Big\}.$
\end{enumerate}
\end{defn}

\NI  We shall now introduce an ordering for the connected components in
$P.$

\NI Suppose that  $P = \underset{1 \leq i \leq s_1}{\cup}(P_i^e \cup P_i^g) \cup \underset{1 \leq j \leq s_2}{\cup}P_{j}^{\mathbb{Z}_2}$
then $R^P = \underset{1 \leq i \leq s_1}{\cup} R^{P_i^{\{e\}}} \cup  \underset{1 \leq j \leq s_2}{\cup}R^{P_{j}^{\mathbb{Z}_2}}.$

Let $a_{11}, \cdots, a_{1s_1}$ be the minimal vertices of the
connected components $R^{P_1^{\{e\}}}, \cdots,
R^{P_{s_1}^{\{e\}}}$ in $R^P$  and $b_{11}, \cdots, b_{1s_2}$ be
the minimal vertices of the connected components
$R^{P_1^{\mathbb{Z}_2}}, \cdots, R^{P_{s_2}^{\mathbb{Z}_2}}$ in
$R^P$ then

\centerline{ $P_i^e < P_j^e$ \text{ and } $P_i^g < P_j^g$ \text{ if
and only if } $R^{P_i^{\{e\}}} < R^{P_j^{\{e\}}}$ \text{ if and
only if } $a_{1i} < a_{1j} \in R^{P}$ \text{ and }}

 \qquad \qquad \qquad \quad \quad $P_l^{\mathbb{Z}_2} < P_f^{\mathbb{Z}_2}$ \text{ if
and only if } $R^{P_l^{\mathbb{Z}_2}} < R^{P_f^{\mathbb{Z}_2}}$
\text{ if and only if } $b_{1l} < b_{1f} \in R^P.$

\NI Since $\overrightarrow{M}^k[(s, (s_1, s_2))] \subseteq M^k[(s, (s_1, s_2))],$ the above ordering can be used for the connected components $P$ when $(d, P) \in \overrightarrow{M}^k[(s, (s_1, s_2))].$

\begin{lem}\textbf{[\cite{K1}, Lemma 4.3]}\label{L2.12}
\mbox{ }
\NI Let $M^k[(s, (s_1, s_2))]$ and $\overrightarrow{M}^k[(s, (s_1, s_2))]$ be as in Definition \ref{D2.11}.
\begin{enumerate}
  \item[(i)] Each $d \in I^{2k}_{2s_1+s_2}$ can be associated with a pair of
elements $(d^+, P),  (d^{-}, Q) \in M^k[(s, (s_1, s_2))]$ and  an element
$((f, \sigma_1), \sigma_2) \in \left(\mathbb{Z}_2 \wr \mathfrak{S}_{s_1}\right)
\times \mathfrak{S}_{s_2}$ where $(d^+, P), (d^-, Q) \in M^k[(s, (s_1, s_2))]$ and  $((f, \sigma_1), \sigma_2) \in (\mathbb{Z}_2 \wr \mathfrak{S}_{s_1}) \times \mathfrak{S}_{s_2}.$
  \item[(ii)]  Each $d \in \overrightarrow{I}^{2k}_{2s_1+s_2}$ can  be associated with a pair of
elements $(d^+, P),  (d^{-}, Q) \in \overrightarrow{M}^k[(s, (s_1, s_2))]$ and  an element
$((f, \sigma_1), \sigma_2) \in \left(\mathbb{Z}_2 \wr \mathfrak{S}_{s_1}\right)
\times \mathfrak{S}_{s_2}$ where $(d^+, P), (d^-, Q)\in \overrightarrow{M}^k[(s, (s_1, s_2))]$ and  $((f, \sigma_1), \sigma_2)\in (\mathbb{Z}_2 \wr \mathfrak{S}_{s_1}) \times \mathfrak{S}_{s_2}.$
\end{enumerate}

\end{lem}

\begin{defn}\textbf{(\cite{K1}, Definition 4.6)}\label{D2.13}
\begin{enumerate}
  \item[(i)] Define a map $\phi^s_{s_1, s_2}: M^k[(s,(s_1, s_2))] \times M^k[(s,(s_1, s_2))]
\rightarrow R [\left(\mathbb{Z}_2 \wr \mathfrak{S}_{s_1}\right)
\times \mathfrak{S}_{s_2}]$ as follows:

\centerline{$\phi^s_{s_1, s_2}\left((d', P), (d'', Q)\right) = x^{l(P \vee Q)} ((f, \sigma_1),
    \sigma_2)$ and }

  \item[(ii)] Define a map $\overrightarrow{\phi}^s_{s_1, s_2}: \overrightarrow{M}^k[(s,(s_1, s_2))] \times \overrightarrow{M}^k[(s,(s_1, s_2))]
\rightarrow R [\left(\mathbb{Z}_2 \wr \mathfrak{S}_{s_1}\right)
\times \mathfrak{S}_{s_2}]$ as follows:

\centerline{$\overrightarrow{\phi}^s_{s_1, s_2}\left((d', P), (d'', Q)\right) = x^{l(P \vee Q)} ((f, \sigma_1),
    \sigma_2)$  }
\end{enumerate}
\NI {\bf Case (i):} if
\begin{enumerate}
    \item[(a)] No two connected components of $Q$ in $d''$ have non-empty intersection with a common  connected component of $d'$ in $d'. d''$, or vice versa.
       \item[(b)] No connected component of $Q$ has non-empty intersection only with the connected components   excluding the connected components of $P$ in $d' . d''.$ Similarly, no connected component in $P$ has non-empty intersection only with  a connected component excluding the connected \\ components of $Q$ in $d'. d''.$
        \end{enumerate}
        where $l(P \vee Q)$ denotes the number of connected
    components in $d' . d''$ excluding the union of all the
    connected components of $P$ and $Q$ and $d' . d'' \in R^{\mathbb{Z}_2}_{\underline{k} \cup \underline{k}'}$ is the smallest $d$ in $R^{\mathbb{Z}_2}_{\underline{k} \cup \underline{k}'}$ such that $d' \cup d'' \subset d.$

    The permutation $\left((f, \sigma_1), \sigma_2 \right)$ is obtained as
    follows: If there is a unique connected  component in $d'. d''$
    containing $P_i^e$ and $Q_j^{g'}$ then, define $\sigma_1(i) = j$ and

    \centerline{$f(i) =  \left\{%
\begin{array}{ll}
    \overline{1}, & \hbox{if $g' = g$;} \\
    \overline{0}, & \hbox{if $g' = e$.} \\
\end{array}%
\right.   $}

Also, if there is a unique connected component in $d' . d''$
containing $P_l^{\mathbb{Z}_2}$ and $Q_f^{\mathbb{Z}_2}$  then, define $\sigma_2(l) =  f).$

\NI {\bf Case (ii):} Otherwise, $\phi^s_{s_1, s_2}\left((d', P), (d'', Q)\right) = 0$ and $\overrightarrow{\phi}^s_{s_1, s_2}\left((d', P), (d'', Q)\right) = 0.$
\end{defn}

\begin{defn}\label{D2.14}

Let $(d, P) \in M^k[(s, (s_1, s_2))]$ such that $|d \setminus P| = 2r_1+r_2$ where $M^k[(s, (s_1, s_2))]$ be as in Definition \ref{D2.11}.

Let $\{P_{1i}^g, g \in \mathbb{Z}_2\}_{1 \leq i \leq s_1} \cup \{P_{2j}^{\mathbb{Z}_2}\}_{1 \leq j \leq s_2}$ be the connected components in $P$ and $ \{P_{3l}^{g'}, g' \in \mathbb{Z}_2\}_{1 \leq l \leq r_1} \cup \{P_{4m}^{\mathbb{Z}_2}\}_{1 \leq m \leq r_2}$ be the connected components in $d \setminus P.$

Define a map $\phi: M^k[(s, (s_1, s_2))] \rightarrow P(k)$ as $\phi((d, P)) = (\alpha_1, \alpha_2, \alpha_3, \alpha_4)$ where $\alpha_1 \vdash k_1, \alpha_2 \vdash k_2, \alpha_3 \vdash k_3, \alpha_4 \vdash k_4$ with $k_1 + k_2 + k_3 + k_4 = k, \alpha_1 = (\alpha_{11}, \alpha_{12}, \cdots,  \alpha_{1s_1}), \alpha_2 = (\alpha_{21}, \alpha_{22}, \cdots, \alpha_{2s_2}), \alpha_3 = (\alpha_{31}, \alpha_{32}, \cdots, \alpha_{3r_1})$ and $\alpha_{4} = (\alpha_{41}, \alpha_{42}, \cdots, \alpha_{4r_2})$ such that $|P_{1i}| = \alpha_{1i}, |P_{2j}| = \alpha_{2j}, |P_{3l}| = \alpha_{3l}, |P_{4m}| = \alpha_{4m}$ respectively for all $ 1 \leq i \leq s_1, 1 \leq j \leq s_2, 1 \leq l \leq r_1 $ and $1 \leq m \leq r_2.$
\end{defn}

\begin{ex}\label{E2.15}
The following example illustrates the use of $2s_1+s_2$  instead of $s=2s_1+s_2$ to denote the number of through classes for the diagrams in algebra of $\mathbb{Z}_2$-relations  and signed partition algebras.

\NI For $s_1 = 0$ and $s_2 = 2, $

\begin{center}
\includegraphics[width=17cm]{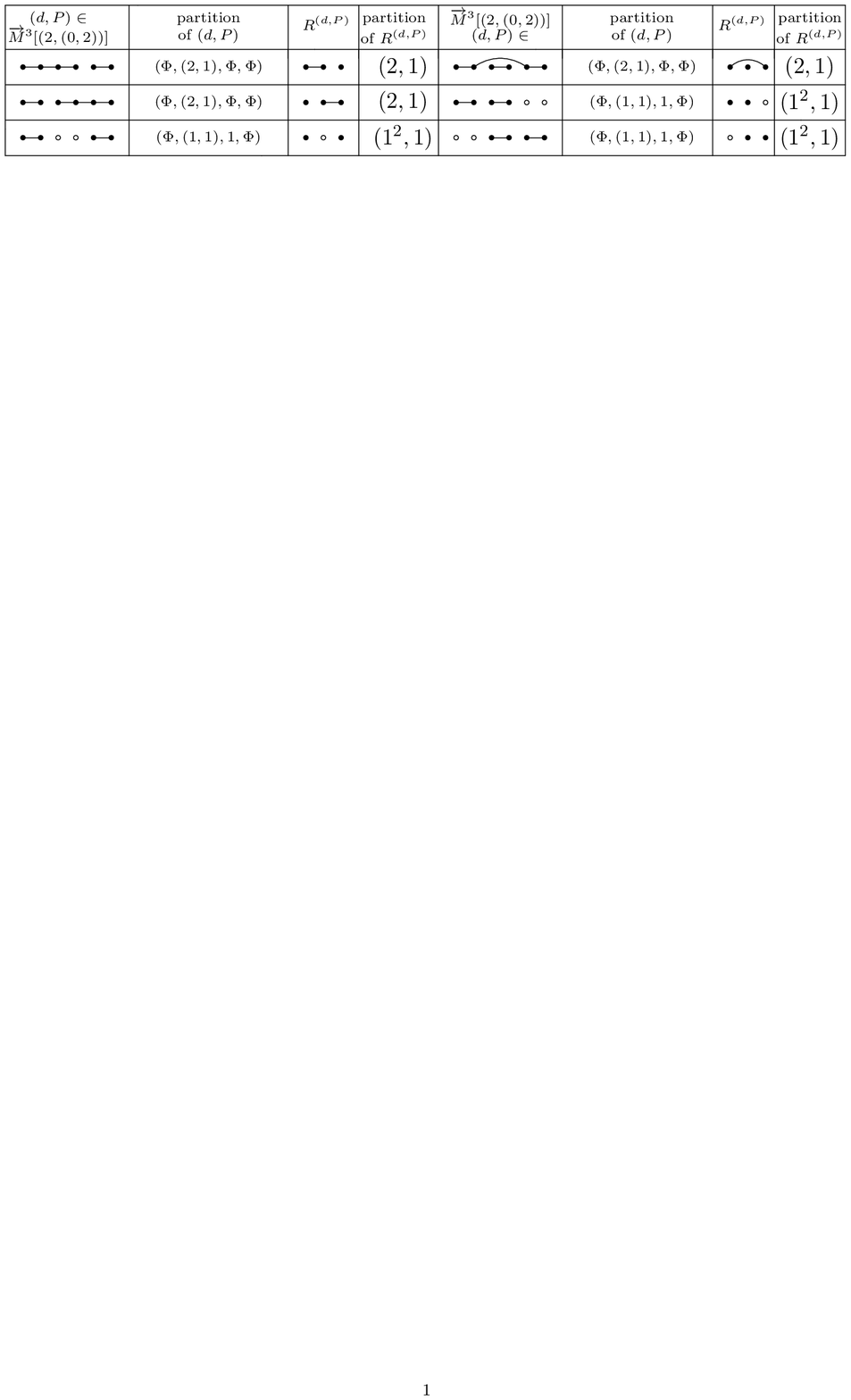}
\end{center}
\vspace{-0.5cm}

\NI For $s_1 = 1$ and $s_2 = 0,$
\vspace{-1cm}
\begin{center}
\includegraphics{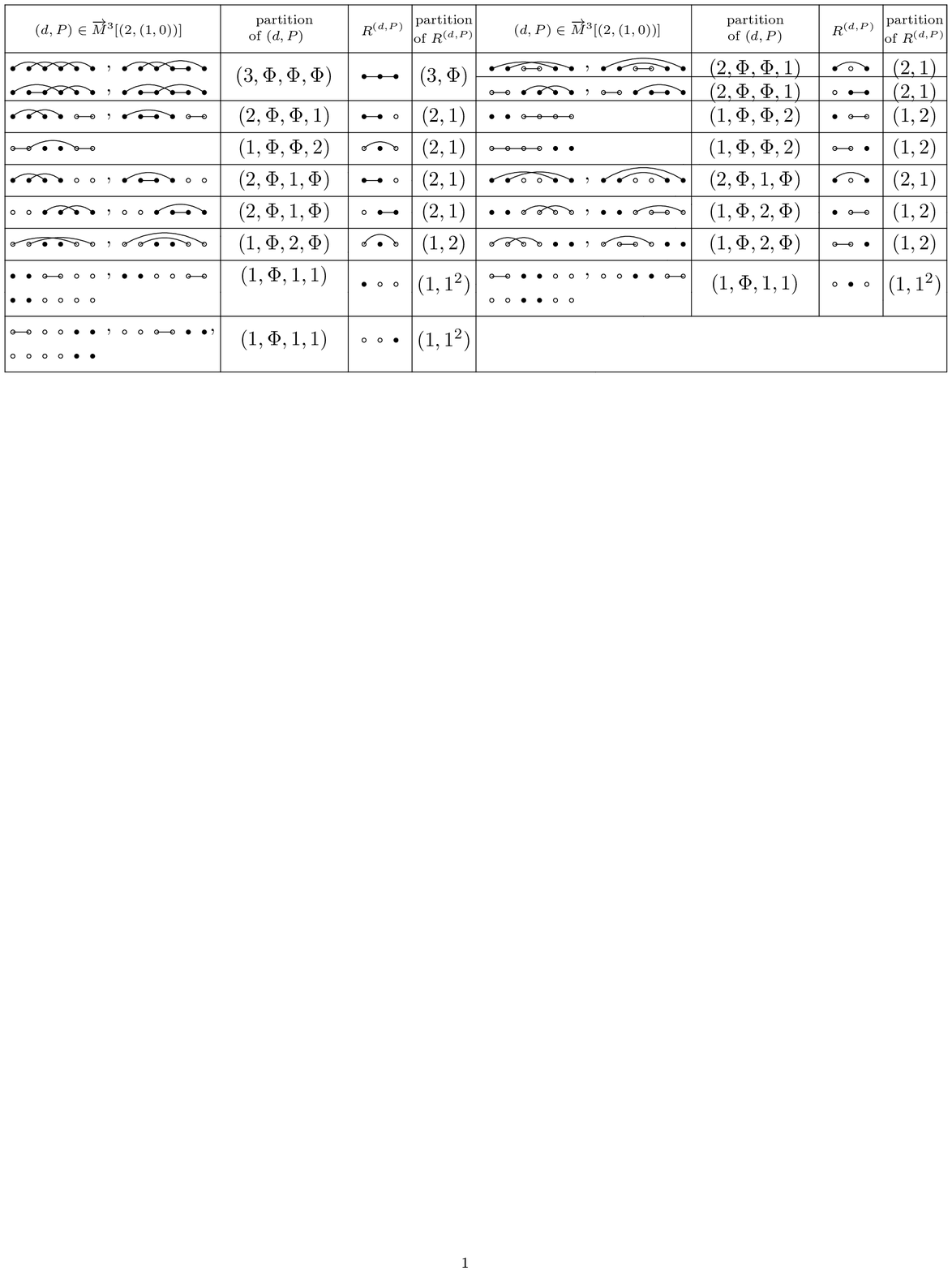}
\end{center}

\NI In the above diagrams, connected components with thick dots(hollow dots) belongs to $P(d \setminus P).$

In partition algebra, for any $d$ whose top row is $(d, P)$ and the bottom row is $(d', P')$ with $|P| = s$ then the number of possible ways to permute the through classes in $d$ will be $s!$ ways.In case of signed partition algebras, for  $(d, P), (d', P') \in M^k[(s, (s_1, s_2))]$ with $|P| = |P'| = 2s_1+s_2 = s,$ then the number of diagram $d$'s whose top row is $(d, P)$ and bottom row is $(d', P')$ will be $2^{s_1} \ s_1! \ s_2!.$ Since $\{e\}$-connected components$(\mathbb{Z}_2$-connected components) in $P$ can be joined only to $\{e\}$-connected components$(\mathbb{Z}_2$-connected components) in $P'.$

Moreover, By Definition \ref{D2.8} we know that $\overrightarrow{I}_s^{2k} = \underset{2s_1+s_2 \leq s}{\cup} \overrightarrow{I}^{2k}_{2s_1+s_2}.$

Let $\overrightarrow{L}_s^{2k}$ be the linear span of $\overrightarrow{I}_s^{2k}$ for every $0 \leq s \leq 2k$ then $\overrightarrow{L}_s^{2k}$ is an ideal of $\overrightarrow{I}_s^{2k}$  and the quotient
$\overrightarrow{L}_s^{2k} / \overrightarrow{L}_{s-1}^{2k} = $ linear span of $\{d \ | \ \sharp^p(d) = s\}.$

For example,

$\overrightarrow{I}_2^6 = \overrightarrow{I}^6_{2 \times 1 +0} \cup \overrightarrow{I}^6_{2 \times 0 +2} \cup \overrightarrow{I}^6_{2 \times 0+1 } \cup \overrightarrow{I}^6_{2 \times 0 +0}$ and $\overrightarrow{I}^6_1 = \overrightarrow{I}^6_{2 \times 0+1 } \cup \overrightarrow{I}^6_{2 \times 0 +0}$ then the quotient ring $\overrightarrow{L}_2^6 / \overrightarrow{L}_1^6$ splits into a direct sum of four ideals $A_1, A_2, A_3, A_4$ where

$\begin{array}{lll}
   A_1 & \text{ is the linear span of } & \left\{d \left(\ds \frac{((0, id), id) + ((0, id), \sigma_2)}{2}\right) \ \Big|  \ d = \widetilde{U}^{(d, P)}_{(d, P)} \right\}_{\widetilde{U}^{(d, P)}_{(d, P)} \in J^6_{2 \times 0 +2}}, \\
   A_2 & \text{ is the linear span of } & \left\{d \left(\ds \frac{((0, id), id) - ((0, id), \sigma_2)}{2}\right) \ \Big|  \ d = \widetilde{U}^{(d, P)}_{(d, P)}\right\}_{\widetilde{U}^{(d, P)}_{(d, P)} \in J^6_{2 \times 0 +2}}, \\
    B_1 & \text{ is the linear span of } & \left\{d \left(\ds \frac{((0, id), id) + ((0,  \sigma_1), id)}{2}\right) \ \Big|  \ d = \widetilde{U}^{(d, P)}_{(d, P)} \right\}_{\widetilde{U}^{(d, P)}_{(d, P)} \in J^6_{2 \times 1 +0}},\\
   B_2 & \text{ is the linear span of } &  \left\{d \left(\ds \frac{((0, id), id) - ((0, \sigma_1), id)}{2}\right) \ \Big|  \ d = \widetilde{U}^{(d, P)}_{(d, P)}\right\}_{\widetilde{U}^{(d, P)}_{(d, P)} \in J^6_{2 \times 1 +0}}.
 \end{array}
$

\NI where $\sigma_1^2 = Id, \sigma_2^2 = Id$ and  $0(i) = 0 $ for every $i.$
 \end{ex}

\section{\textbf{Gram Matrices  and $(s_1, s_2, r_1, r_2, p_1, p_2)$-Stirling Numbers}}

In this section, we introduce a new class of matrices $G_{2s_1+s_2}^k, \overrightarrow{G}_{2s_1+s_2}^k$ and $ G_s^k$ of the algebra of $\mathbb{Z}_2$-relations, signed partition algebras  and  partition algebras respectively which will be called as Gram matrices since by Theorem 3.8 in  \cite{GL}  the Gram matrices $G^{\lambda, \mu}_{2s_1+s_2}$ associated to the cell modules of $W[(s, (s_1, s_2)), ((\lambda_1, \lambda_2), \mu)]$ (for $\lambda = ([s_1], \Phi)$, $\mu = [s_2]$ if $s_1, s_2 \neq 0; \lambda = (\Phi, \Phi)$, $\mu = [s_2]$ if $s_1 = 0, s_2 \neq 0; \lambda = ([s_1], \Phi)$, $\mu = \Phi$ if $s_1 \neq 0 , s_2 = 0$; $\lambda = (\Phi, \Phi)$, $\mu = \Phi$ if $s_1 = s_2 = 0$, $0 \leq s_1 \leq k, 0 \leq s_2 \leq k$ and $0 \leq s_1+s_2 \leq k$) and $\overrightarrow{G}^{\lambda, \mu}_{2s_1+s_2}$ $\overrightarrow{W}[(s, (s_1, s_2)), ((\lambda_1, \lambda_2), \mu)]$  (for $\lambda = ([s_1], \Phi)$, $\mu = [s_2]$ if $s_1, s_2 \neq 0; \lambda = (\Phi, \Phi)$, $\mu = [s_2]$ if $s_1 = 0, s_2 \neq 0; \lambda = ([s_1], \Phi)$, $\mu = \Phi$ if $s_1 \neq 0 , s_2 = 0$; $\lambda = (\Phi, \Phi)$, $\mu = \Phi$ if $s_1 = s_2 = 0$, $0 \leq s_1 \leq k, 0 \leq s_2 \leq k-1$ and $0 \leq s_1+s_2 \leq k - 1$)  defined in Definition 6.3 of \cite{K1} coincides with the  matrices  $G_{2s_1+s_2}^k$ and $\overrightarrow{G}_{2s_1+s_2}^k$ respectively. Also, the Gram matrices  $G_{2s_1+s_2}^k$ and $\overrightarrow{G}_{2s_1+s_2}^k$ are similar to the matrices $\widetilde{G}^k_{2s_1+s_2}$ and $\widetilde{\overrightarrow{G}}^k_{2s_1+s_2}$ which is a direct sum of block sub matrices $\widetilde{A}_{2r_1+r_2, 2r_1+r_2}$ and $\widetilde{\overrightarrow{A}}_{2r_1+r_2, 2r_1+r_2}$  of sizes $f^{2r_1+r_2}_{2s_1+s_2}$ and  $\overrightarrow{f}^{2r_1+r_2}_{2s_1+s_2}$ respectively. The diagonal entries of the matrices $\widetilde{A}_{2r_1+r_2, 2r_1+r_2}$ and $\widetilde{\overrightarrow{A}}_{2r_1+r_2, 2r_1+r_2}$ are the same and the diagonal element is a product of $r_1$ quadratic polynomials and $r_2$ linear polynomials which could help in determining the roots of the determinant of the Gram matrix. In this connection, $(s_1, s_2, r_1, r_2, p_1, p_2)$-Stirling numbers of the second kind for the algebra of $\mathbb{Z}_2$-relations and signed partition algebras are introduced and their identities are established. Similarly, we have also established that the Gram matrix $G^k_s$ of a partition algebra is similar to a matrix $\widetilde{G}^k_s$ which is a direct sum of block matrices $\widetilde{A}_{r, r}$ of size $f^r_s.$ The diagonal entries of the matrices $\widetilde{A}_{r, r}$ are the same and the diagonal element is a product of $r$ linear polynomials which could help in determining the roots of the determinant of the Gram matrix. Stirling numbers of second kind corresponding to the partition algebras are also introduced and their identities are established.

We begin by calculating the size of the Gram matrices before explaining the entries of the Gram matrices.

\begin{defn} \label{D3.1}
Put
\begin{enumerate}
  \item[(a)] $\Omega_{s_1, s_2}^{r_1, r_2} = \Big\{ \left[ \alpha_1 \right]^1 \left[ \alpha_2 \right]^2 \left[ \alpha_3 \right]^3 \left[  \alpha_4 \right]^4 \Big| \alpha_1 \vdash k_1, \alpha_2 \vdash k_2, \alpha_3 \vdash k_3, \alpha_4 \vdash k_4  \text{ with } \alpha_1 \in \mathbb{P}(k_1, s_1),  \alpha_2 \in \mathbb{P}(k_2, s_2), \alpha_3 \in \mathbb{P}(k_3, r_1), \alpha_4 \in \mathbb{P}(k_4, r_2) \text{ such that } k_1 + k_2 + k_3 + k_4 = k \Big\}$

\NI where $\alpha_1 = \left( \alpha_{11}, \alpha_{12}, \cdots, \alpha_{1 s_1} \right),  \alpha_2 = \left(  \alpha_{21},  \alpha_{22}, \cdots,  \alpha_{2 s_2} \right), \alpha_3 = \left( \alpha_{31}, \alpha_{32}, \cdots, \alpha_{3 r_1} \right)$  and $ \alpha_4 = \left(  \alpha_{41},  \alpha_{42}, \cdots,  \alpha_{4 r_2}\right).$
  \item[(b)] $\overrightarrow{\Omega}_{s_1, s_2}^{r_1, r_2} = \Big\{ \left[ \alpha_1 \right]^1 \left[ \alpha_2 \right]^2 \left[ \alpha_3 \right]^3 \left[  \alpha_4 \right]^4 \in \Omega^{r_1, r_2}_{s_1, s_2} \ \Big| \ s_1 + s_2 + r_1 + r_2 \leq k-1 \text{ and if } s_1+s_2+r_1+r_2 = k \text{ then } r_1 \neq 0 \text{ or } s_1 = k\Big\}$

  \item[(c)] $\Omega^{r}_{s} = \{ [\alpha_1]^1 [\alpha_2]^2 \ | \ \alpha_1 \in \mathbb{P}(k_1, s), \alpha_2 \in \mathbb{P}(k_2, r) \text{ such that } k_1 + k_2 = k\}.$
\end{enumerate}

\end{defn}

\begin{defn}\label{D3.2}

Let $\alpha = [\alpha_1]^1 [\alpha_2]^2 [\alpha_3]^3 [\alpha_4]^4 \in \Omega^{r_1, r_2}_{s_1, s_2}.$

\NI We shall draw a graph corresponding to the partition $\alpha = [\alpha_1]^1 [\alpha_2]^2 [\alpha_3]^3 [\alpha_4]^4$ on the vertices $(i, e), (i, g)$ for all $1 \leq i \leq k$ and $1' \leq i \leq k'$ arranged in two rows of each having $k$-vertices labeled from left to right. The edges are drawn as follows:

\begin{enumerate}
  \item[(a)] Draw an edge  connecting the vertices $\left( \left( \underset{n=1}{\overset{i-1}{\sum}} |\alpha_{1n}|\right)+1, e\right)$, $\left( \left( \underset{n=1}{\overset{i-1}{\sum}} |\alpha_{1n}|\right)+2, e\right), $

     \NI  $\cdots, \left( \left( \underset{n=1}{\overset{i}{\sum}} |\alpha_{1n}|\right), e\right),\left(\left( \left( \underset{n=1}{\overset{i-1}{\sum}} |\alpha_{1n}|\right)+1\right)', e\right), \left( \left(\left( \underset{n=1}{\overset{i-1}{\sum}} |\alpha_{1n}|\right)+2\right)', e\right), \cdots,$

      \NI $\left( \left( \underset{n=1}{\overset{i}{\sum}} |\alpha_{1n}|\right)', e\right)$ and denote it by $P_{1i}^e$ for $1 \leq i \leq s_1.$ Since the diagram has to be a $\mathbb{Z}_2$-stable diagram there should be a copy of the connected component which is obtained by connecting the vertices  $\left( \left( \underset{n=1}{\overset{i-1}{\sum}} |\alpha_{1n}|\right)+1, g\right), \left( \left( \underset{n=1}{\overset{i-1}{\sum}} |\alpha_{1n}|\right)+2, g\right), \cdots,\left( \left( \underset{n=1}{\overset{i}{\sum}} |\alpha_{1n}|\right), g\right)$,

       \NI $\left(\left( \left( \underset{n=1}{\overset{i-1}{\sum}} |\alpha_{1n}|\right)+1\right)', g\right), \left( \left(\left( \underset{n=1}{\overset{i-1}{\sum}} |\alpha_{1n}|\right)+2\right)', g\right), \cdots, \left( \left( \underset{n=1}{\overset{i}{\sum}} |\alpha_{1n}|\right)', e\right)$ and denote it by $P^g_{1i}$ for $ 1\leq i \leq s_1.$ The connected components $P_{1i}^e$ and $P_{1i}^g$ for $1 \leq i \leq s_1$ are called $\{e\}$-through classes.\\

  \item[(b)] Draw an edge  connecting the vertices $\left( \left( \underset{i=1}{\overset{s_1}{\sum}} |\alpha_{1i}| + \underset{m=1}{\overset{j-1}{\sum}}|\alpha_{2m}|\right)+1, e\right)$,

       \NI $\left( \left( \underset{i=1}{\overset{s_1}{\sum}} |\alpha_{1i}| + \underset{m=1}{\overset{j-1}{\sum}}|\alpha_{2m}|\right)+1, g\right), \cdots, \left( \left( \underset{i=1}{\overset{s_1}{\sum}} |\alpha_{1i}| + \underset{m=1}{\overset{j}{\sum}}|\alpha_{2m}|\right), e\right), \left( \left( \underset{i=1}{\overset{s_1}{\sum}} |\alpha_{1i}| + \underset{m=1}{\overset{j}{\sum}}|\alpha_{2m}|\right), g\right), $

       \NI $\left(\left( \left( \underset{i=1}{\overset{s_1}{\sum}} |\alpha_{1i}| + \underset{m=1}{\overset{j-1}{\sum}}|\alpha_{2m}|\right)+1\right)', e\right), \left(\left( \left( \underset{i=1}{\overset{s_1}{\sum}} |\alpha_{1i}| + \underset{m=1}{\overset{j-1}{\sum}}|\alpha_{2m}|\right)+1\right)', g\right), \cdots,$

        \NI $\left( \left( \underset{i=1}{\overset{s_1}{\sum}} |\alpha_{1i}|+ \underset{m=1}{\overset{j}{\sum}}|\alpha_{2m}|\right)', e\right), \left( \left( \underset{i=1}{\overset{s_1}{\sum}} |\alpha_{1i}|+ \underset{m=1}{\overset{j}{\sum}}|\alpha_{2m}|\right)', g\right)$ and denote it by $P_{2j}^{\mathbb{Z}_2}$ for $1 \leq j \leq s_2.$

       \NI The connected components $P_{2j}^{\mathbb{Z}_2}$  for $1 \leq j \leq s_2$ are called $\mathbb{Z}_2$-through classes.

  \item[(c)] Draw edges connecting the vertices $\left( \left( \underset{i=1}{\overset{s_1}{\sum}} |\alpha_{1i}| + \underset{j=1}{\overset{s_2}{\sum}}|\alpha_{2j}| + \right.  \left.\underset{f=1}{\overset{l-1}{\sum}}|\alpha_{3f}|\right)+1, e\right), \cdots,$

        \NI $\left( \left( \underset{i=1}{\overset{s_1}{\sum}} |\alpha_{1i}| + \underset{j=1}{\overset{s_2}{\sum}}|\alpha_{2j}| + \underset{f=1}{\overset{l}{\sum}}|\alpha_{3f}|\right), e\right)$ in the top row and $\left(\left( \left( \underset{i=1}{\overset{s_1}{\sum}} |\alpha_{1i}| + \underset{j=1}{\overset{s_2}{\sum}}|\alpha_{2j}| + \underset{f=1}{\overset{l-1}{\sum}}|\alpha_{3f}|\right)+1\right)', e\right),$

      \NI $ \cdots,  \left( \left( \underset{i=1}{\overset{s_1}{\sum}} |\alpha_{1i}| + \underset{j=1}{\overset{s_2}{\sum}}|\alpha_{2j}| + \underset{f=1}{\overset{l}{\sum}}|\alpha_{3f}|\right)', e\right)$ in the bottom row and denote it by $P_l^e$ and $P_l^{'e}$ respectively. Since the diagram
       has to be $\mathbb{Z}_2$-stable diagram there will be copy of the above connected components obtained by connecting the vertices  $\left( \left( \underset{i=1}{\overset{s_1}{\sum}} |\alpha_{1i}| + \underset{j=1}{\overset{s_2}{\sum}}|\alpha_{2j}| + \underset{f=1}{\overset{l-1}{\sum}}|\alpha_{3f}|\right)+1, g\right), \cdots,$

         \NI $\left( \left( \underset{i=1}{\overset{s_1}{\sum}} |\alpha_{1i}| + \underset{j=1}{\overset{s_2}{\sum}}|\alpha_{2j}| + \underset{f=1}{\overset{l}{\sum}}|\alpha_{3f}|\right), g\right)$ in the top row $\left(\left( \left( \underset{i=1}{\overset{s_1}{\sum}} |\alpha_{1i}| + \underset{j=1}{\overset{s_2}{\sum}}|\alpha_{2j}| + \underset{f=1}{\overset{l-1}{\sum}}|\alpha_{3f}|\right)+1\right)', g\right),  $

      \NI $\cdots, \left( \left( \underset{i=1}{\overset{s_1}{\sum}} |\alpha_{1i}| + \underset{j=1}{\overset{s_2}{\sum}}|\alpha_{2j}| + \underset{f=1}{\overset{l}{\sum}}|\alpha_{3f}|\right)', g\right)$ and denote it by $P_l^g$ and $P^{'g}_l$ respectively.

      \NI The connected components $P_l^e, P^{'e}_l, P_l^g$ and $P^{'g}_l$ for $1 \leq l \leq r_1$ are called $\{e\}$-horizontal edges.

  \item[(d)] Draw edges connecting the vertices $\left( \left( \underset{i=1}{\overset{s_1}{\sum}} |\alpha_{1i}| + \underset{j=1}{\overset{s_2}{\sum}}|\alpha_{2j}| + \underset{l=1}{\overset{r_1}{\sum}}|\alpha_{3l}| + \underset{t=1}{\overset{m-1}{\sum}}|\alpha_{4t}| \right)+1, e\right),$

       \NI $\left( \left( \underset{i=1}{\overset{s_1}{\sum}} |\alpha_{1i}| + \underset{j=1}{\overset{s_2}{\sum}}|\alpha_{2j}| + \underset{l=1}{\overset{r_1}{\sum}}|\alpha_{3l}| + \underset{t=1}{\overset{m-1}{\sum}}|\alpha_{4t}| \right)+1, g\right), \cdots, $

      \NI $\left( \left( \underset{i=1}{\overset{s_1}{\sum}} |\alpha_{1i}| + \underset{j=1}{\overset{s_2}{\sum}}|\alpha_{2j}| + \underset{l=1}{\overset{r_1}{\sum}}|\alpha_{3l}| + \underset{t=1}{\overset{m}{\sum}}|\alpha_{4t}| \right), e\right), \left( \left( \underset{i=1}{\overset{s_1}{\sum}} |\alpha_{1i}| + \underset{j=1}{\overset{s_2}{\sum}}|\alpha_{2j}| + \underset{l=1}{\overset{r_1}{\sum}}|\alpha_{3l}| + \underset{t=1}{\overset{m}{\sum}}|\alpha_{4t}| \right), g\right)$

       \NI in the top row and $\left( \left(\left( \underset{i=1}{\overset{s_1}{\sum}} |\alpha_{1i}| + \underset{j=1}{\overset{s_2}{\sum}}|\alpha_{2j}| + \underset{l=1}{\overset{r_1}{\sum}}|\alpha_{3l}| + \underset{t=1}{\overset{m-1}{\sum}}|\alpha_{4t}| \right)+1\right)', e\right),$

      \NI $ \left(\left( \left( \underset{i=1}{\overset{s_1}{\sum}} |\alpha_{1i}| + \underset{j=1}{\overset{s_2}{\sum}}|\alpha_{2j}| + \underset{l=1}{\overset{r_1}{\sum}}|\alpha_{3l}| + \underset{t=1}{\overset{m-1}{\sum}}|\alpha_{4t}| \right)+1\right)', g\right), \cdots,$

        \NI $\left( \left( \underset{i=1}{\overset{s_1}{\sum}} |\alpha_{1i}| + \underset{j=1}{\overset{s_2}{\sum}}|\alpha_{2j}| + \underset{l=1}{\overset{r_1}{\sum}}|\alpha_{3l}| + \underset{t=1}{\overset{m}{\sum}}|\alpha_{4t}| \right)', e\right), \left( \left( \underset{i=1}{\overset{s_1}{\sum}} |\alpha_{1i}| + \underset{j=1}{\overset{s_2}{\sum}}|\alpha_{2j}| + \underset{l=1}{\overset{r_1}{\sum}}|\alpha_{3l}| + \underset{t=1}{\overset{m}{\sum}}|\alpha_{4t}| \right)', g\right)$ in the bottom row and it is denoted by $P_m^{\mathbb{Z}_2}$ and $P^{' \mathbb{Z}_2}_m$ for $1 \leq m \leq r_2.$

      \NI The connected components $P_m^{\mathbb{Z}_2}, P_m^{' \mathbb{Z}_2}$ for $1 \leq m \leq r_2$ are called $\mathbb{Z}_2$-horizontal edges.

  \NI The diagram obtained above is called standard diagram and it is denoted by $U^{\alpha}$ where

  \NI $\alpha = [\alpha_1]^1 [\alpha_2]^2 [\alpha_3]^3 [\alpha_4]^4 \in \Omega^{r_1, r_2}_{s_1, s_2}.$
\end{enumerate}
\end{defn}

\begin{ex}\label{E3.3}
The following are some examples of standard diagrams of $U^{\alpha}$ type in signed partition algebras $\overrightarrow{A}_5^{\mathbb{Z}_2}$  with their corresponding partitions.
\begin{center}
\includegraphics[height=4.5cm, width=14cm]{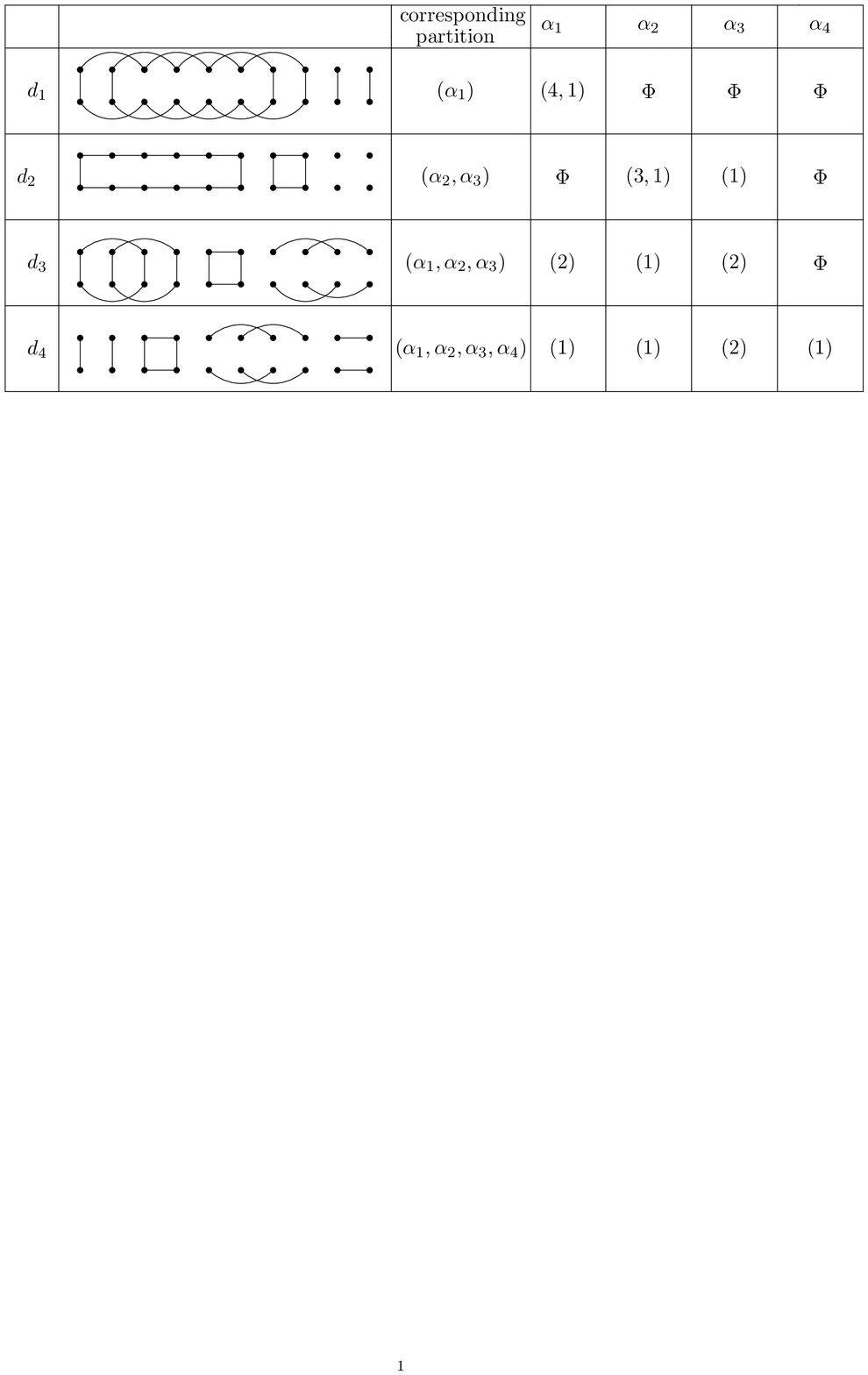}
\end{center}
\vspace{-0.5cm}

\end{ex}

\begin{rem} \label{R3.4}
Let $d \in I^{2k}_{2s_1+s_2}.$

By Lemma \ref{L2.12}, for any $d \in I^{2k}_{2s_1+s_2}$ we can associate a pair $(d^+, P), (d^-, Q) \in M^k[(s, (s_1, s_2))]$ and an element $((f, \sigma_1), \sigma_2) \in \left(\mathbb{Z}_2 \wr \mathfrak{S}_{s_1} \right) \times \mathfrak{S}_{s_2}$ and vice versa and it is denoted by $U^{(d^+, P)}_{(d^-, Q)}((f, \sigma_1), \sigma_2).$

If $d^+ = d^-, P = Q$ and $((f, \sigma_1), \sigma_2) = ((0, id), id) \in \left(\mathbb{Z}_2 \wr \mathfrak{S}_{s_1} \right) \times \mathfrak{S}_{s_2}$ then without loss of generality we can write such $d$ as $\widetilde{U}^{(d, P)}_{(d, P)}.$
\end{rem}

\begin{defn}\label{D3.5}
Let $\alpha = \left[ \alpha_1 \right]^1 \left[ \alpha_2 \right]^2 \left[ \alpha_3 \right]^3 \left[  \alpha_4 \right]^4 \in \Omega_{ s_1, s_2}^{r_1, r_2}.$

Define,

\centerline{$St^c \left( U^{\alpha} \right) = \left\{ \sigma \in \mathbb{Z}_2 \wr \mathfrak{S}_{k} \ \Big| \ \sigma U^{\alpha} \sigma^{-1} = U^{\alpha} \right\}$}

\NI where $U^{\alpha}$ is the standard diagram corresponding to the partition $\alpha$ as in Definition \ref{D3.2}.
\end{defn}

\begin{note}\label{N1}
\mbox{ }
\begin{enumerate}
  \item[(i)] Let $U^{\overrightarrow{\alpha}}$ denote the standard diagram in signed partition algebra  corresponding to the partition $\overrightarrow{\alpha} \in \overrightarrow{\Omega}^{r_1, r_2}_{s_1, s_2}$ and  $R^{U^{\alpha}}$ denote the standard diagram in  partition algebra corresponding to the partition $R^{\alpha} \in \Omega^r_s$ which can be defined as in Definition \ref{D3.2}, $St^c\left( U^{\overrightarrow{\alpha}}\right)$ and $St^c\left( R^{U^{\alpha}}\right)$ can also be defined as in Definition \ref{D3.5} for  the signed partition algebras $\overrightarrow{A}_k^{\mathbb{Z}_2}(x)$ and the partition algebras $A_k(x).$
  \item[(ii)] All other diagrams  $U^{(d, P)}_{(d, P)}, \overrightarrow{U}^{(d, P)}_{(d, P)}, $ and $R^{U^{(d, P)}_{(d, P)}}$ whose underlying partition is same as the underlying partition of $U^{\alpha}, U^{\overrightarrow{\alpha}}$ and $R^{U^{\alpha}}$ respectively  can be obtained as follows:

$U^{(d, P)}_{(d, P)} = \tau \ U^{\alpha} \ \tau^{-1} \ , \  \overrightarrow{U}^{(d, P)}_{(d, P)} = \overrightarrow{\tau} \ U^{\overrightarrow{\alpha}} \ \overrightarrow{\tau}^{-1}$ and $R^{U^{(d, P)}_{(d, P)}} = \rho \ R^{U^{\alpha}} \ \rho$

 \NI where $\tau,  \in \mathbb{Z}_2 \wr \mathfrak{S}_{k}$ and $\rho \in \mathfrak{S}_k$ are the coset representatives of $St^c \left( U^{\alpha} \right), St^c \left( U^{\overrightarrow{\alpha}} \right)   $ and $St^c\left( R^{U^{\alpha}}\right)$  respectively. Also,  $U^{\alpha}, U^{\overrightarrow{\alpha}}$ and $R^{U^{\alpha}}$ are the standard diagrams as  in Definition \ref{D3.2}.
 \end{enumerate}
\end{note}

\begin{notation}\label{N3.6}
\mbox{ }

\begin{enumerate}
            \item[(a)]  For $ 0 \leq r_1, r_2 \leq k-s_1 - s_2 $ and $ 0 \leq s_1, s_2 \leq k,$

            \NI put

             $\begin{array}{lcl}
                       J^{2k}_{2s_1+s_2} & = & \underset{ \ds 0 \leq r_1 + r_2 \leq k-s_1-s_2 }{\cup} \ \ \mathbb{J}_{2s_1+ s_2}^{2r_1 + r_2} \text{ and } \\
                       \mathbb{J}_{2s_1 + s_2}^{2r_1+ r_2} & = & \underset{ \ds \alpha = [\alpha_1]^1[\alpha_2]^2[\alpha_3]^3[\alpha_4]^4 \in \Omega^{r_1, r_2}_{s_1, s_2}}{\cup} \ \  \mathbb{J}_{2s_1+ s_2}^{2r_1 + r_2, \alpha}
                        \end{array}$

                \NI where $\mathbb{J}_{2s_1 + s_2}^{2r_1+ r_2, \alpha} = \Big\{ d \in I^{2k}_{2s_1 + s_2} \ \Big| \  d = \widetilde{U}^{(d, P)}_{(d, P)}  \text{ with } d^+ = (d, P), d^- = (d, P), \eta_e \left(\widetilde{U}^{(d, P)}_{(d, P)}\right) = s_1,$

                \NI $ \eta_{\small{\mathbb{Z}_2}} \left(\widetilde{U}^{(d, P)}_{(d, P)}\right) = s_2, \widetilde{U}^{(d, P)}_{(d, P)} \text{ has } r_1 \text{ number of pairs of }  \{e\}-\text{of horizontal edges, } r_2 \text{ number of }    $

                 \NI $\mathbb{Z}_2-\text{horizontal}, \text{edges }(d, P) \in  M^k[(s, (s_1, s_2))] \text{ as in Definition 2.10},   \| P\| = 2s_1 + s_2 \text{ and } \alpha \\ \text{is the underlying partition of } (d, P) \text{ as in Definition } \ref{D2.13}  \Big\}.$


 Also,

$\begin{array}{lll}
\left| \mathbb{J}^{2r_1+r_2, \alpha}_{2s_1+s_2}\right| & = & \text{index of } St^c(U^{\alpha}) = f^{2r_1+r_2, \alpha}_{2s_1+s_2}\\
 \left| \mathbb{J}_{2s_1 + s_2}^{2 r_1 + r_2} \right| & = & \underset{\ds \alpha = \left[\alpha_1 \right]^1 \left[  \alpha_2 \right]^2 \left[ \alpha_3 \right]^3 \left[ \alpha_4 \right]^4 \in \Omega_{s_1, s_2}^{r_1, r_2}}{\sum}  \text{ index of } St^c \left(U^{\alpha} \right) = f^{2r_1+r_2}_{2s_1+s_2} \\
  \left|  J^{2k}_{2s_1 + s_2} \right| & = & \underset{\ds 0 \leq r_1 + r_2 \leq k-s_1-s_2 }{\sum} \left| \mathbb{J}_{2s_1 + s_2}^{2r_1+r_2} \right|.
\end{array}$

\NI $\left|  J^{2k}_{2s_1 + s_2} \right|$ will define the size of the Gram matrix in the algebra of $\mathbb{Z}_2$-relation and it is denoted by $f_{2s_1+s_2}.$

\item[(b)] For $ 0 \leq r_1 \leq k-s_1 - s_2, 0 \leq r_2 \leq k -s_1-s_2-1, $ $ 0 \leq s_1 \leq k, 0 \leq s_2 \leq k - 1,$ and $ 0 \leq s_1+s_2+r_1+r_2 \leq k - 1$
     \begin{enumerate}
              \item[(i)] if $r_1 \neq 0$ then $\overrightarrow{\mathbb{J}}^{2r_1+r_2, \alpha}_{2s_1+s_2} = \mathbb{J}^{2r_1+r_2, \alpha}_{2s_1+s_2}$
              \item[(ii)] if $r_1 = 0$ then $\overrightarrow{\mathbb{J}}^{2r_1+r_2, \alpha}_{2s_1+s_2} = \{ d \in \mathbb{J}^{2r_1+r_2, \alpha}_{2s_1+s_2} \ | \ \text{ either } s_1 = k \text{ or } s_1+s_2 + r_2 \leq k-1\}$
            \end{enumerate}

  $\begin{array}{lcl}
    \overrightarrow{\mathbb{J}}_{2s_1 + s_2}^{2r_1+ r_2} & = & \underset{ \ds \alpha = [\alpha_1]^1[\alpha_2]^2[\alpha_3]^3[\alpha_4]^4 \in \overrightarrow{\Omega}^{r_1, r_2}_{s_1, s_2}}{\cup} \ \  \overrightarrow{\mathbb{J}}_{2s_1+ s_2}^{2r_1 + r_2, \alpha} \text{ and } \\ \overrightarrow{J}^{2k}_{2s_1+s_2} & = & \underset{\substack{\ds 0 \leq r_1 \leq k-s_1-s_2 \\ \ds \hspace{0.6cm} 0 \leq r_2 \leq k-s_1-s_2-1 \\ \ds \hspace{0.75cm} 0 \leq r_1+r_2 \leq k-s_1-s_2}}{\cup} \ \  \overrightarrow{\mathbb{J}}_{2s_1+ s_2}^{2r_1 + r_2}
                        \end{array}$

 Also,

$\begin{array}{lll}
\left| \overrightarrow{\mathbb{J}}^{2r_1+r_2, \alpha}_{2s_1+s_2}\right| & = & \text{index of } St^c(U^{\alpha}) = \overrightarrow{f}^{2r_1+r_2, \alpha}_{2s_1+s_2}\\
 \left| \overrightarrow{\mathbb{J}}_{2s_1 + s_2}^{2 r_1 + r_2} \right| & = & \underset{\ds \alpha = \left[\alpha_1 \right]^1 \left[  \alpha_2 \right]^2 \left[ \alpha_3 \right]^3 \left[ \alpha_4 \right]^4 \in \overrightarrow{\Omega}_{s_1, s_2}^{r_1, r_2}}{\sum}  \text{ index of } St^c \left(U^{\overrightarrow{\alpha}} \right) = \overrightarrow{f}^{2r_1+r_2}_{2s_1+s_2} \\
  \left|  \overrightarrow{J}^{2k}_{2s_1 + s_2} \right| & = & \underset{\substack{\ds 0 \leq r_1 \leq k-s_1-s_2 \\ \ds \hspace{0.6cm}0 \leq r_2 \leq k - s_1-s_2 -1 \\ \ds \hspace{0.75cm}0 \leq r_1+r_2 \leq k-s_1-s_2}}{\sum} \left| \overrightarrow{\mathbb{J}}_{2s_1 + s_2}^{2r_1+r_2} \right|.
\end{array}$

 \NI $\left|  \overrightarrow{J}^{2k}_{2s_1 + s_2} \right|$ will define the size of the Gram matrix in signed partition algebras and it is denoted by $\overrightarrow{f}_{2s_1+s_2}.$

\item[(c)] For $ 0 \leq r \leq k-s,  0 \leq s \leq k$  put $ J^{k}_{s} = \underset{\ds 0 \leq r \leq k-s }{\cup} \ \  \mathbb{J}_{s}^{r}$ and $\mathbb{J}^r_s = \underset{\alpha = [\alpha_1]^1[\alpha_2]^2 \in \Omega^{r}_s}{\cup} \mathbb{J}^{r, \alpha}_s$ where

$\mathbb{J}_{s}^{r, \alpha} = \Big\{ R^d \in I^{k}_{s} \ \Big| \ R^d = U^{(R^d)^+}_{(R^d)^-}, \left( R^d\right)^+ \text{ and }  \left(R^d\right)^{-} \text{ are the same},  \sharp^p(U^{(R^d)^+}_{(R^d)^-}) = s,  U^{(R^{d})^+}_{(R^d)^-} \text{ has }$

 \NI \hspace{3cm} $r \text{ number of horizontal edges and } \alpha \text{ is the underlying partition of } R^d \Big\}.$

\NI For the sake of simplicity we write,  $U^{(R^d)^+}_{(R^d)^-} = U^{R^d}_{R^d}.$

\NI Also, $
\left| \mathbb{J}^{r, \alpha}_{s}\right| = \text{index of }St^c \left(U^{R^{\alpha}} \right) = f^{r, \alpha}_{s},  \left| \mathbb{J}_{s}^{r} \right|  =  \underset{\ds R^{\alpha} = \left[\alpha_1 \right]^1 \left[  \alpha_2 \right]^2 \in \Omega_{s}^{r}}{\sum}  \text{ index of } St^c \left(U^{R^{\alpha}} \right)  = f^r_s \text{ and }$

\NI $  \left|  J^{k}_{s} \right|  =  \underset{\ds 0 \leq r \leq k-s}{\sum} \left| \mathbb{J}_{s}^{r} \right|.
$

          \end{enumerate}
\NI $\left|J^k_s \right|$ will define the size of the Gram matrix in the partition algebra and it is denoted by $f_s.$
\end{notation}

\begin{defn} \label{D3.7}
\mbox{ }
\begin{enumerate}
  \item[(a)] The diagrams in $J^{2k}_{2s_1 + s_2}$  are indexed as follows:

\centerline{$\ds \left\{\left( \widetilde{U}^{(d, P)}_{(d, P)}\right)_{i, \alpha}^{r_1, r_2} \ \Big| \ 1 \leq i \leq f^{2r_1+r_2, \alpha}_{2s_1 + s_2}, \alpha \in \Omega^{r_1, r_2}_{s_1, s_2}\right\}_{\substack{\ds 0 \leq r_1, r_2  \leq k-s_1-s_2\\ \ds  0 \leq r_1+r_2 \leq k-s_1-s_2  }} .$}

$(i, \alpha, r_1, r_2) < (j, \beta, r'_1, r'_2),$

\begin{itemize}
  \item[(i)] if $2r_1+r_2 < 2r'_1+r'_2$
  \item[(ii)] if $2r_1+r_2 = 2r'_1+r'_2 $ and $r_1+r_2 < r'_1+r'_2$
  \item[(iii)] if  $2r_1+r_2 = 2r'_1+r'_2, r_1+r_2 = r'_1+r'_2$ and $\alpha <  \beta$(lexicographical ordering)
   \item[(iv)]  if  $2r_1+r_2 = 2r'_1+r'_2, r_1+r_2 = r'_1+r'_2$ and $\alpha =  \beta$ then it can be indexed arbitrarily.
\end{itemize}
where $r_1$ is the number of pairs of $\{e\}$-horizontal edges in $\left( \widetilde{U}^{(d, P)}_{(d, P)}\right)_{i, \alpha}^{r_1, r_2}  , r'_1$ is the number of pairs of $\{e\}$-horizontal edges in $\left( \widetilde{U}^{(d, P)}_{(d, P)}\right)_{j, \beta}^{r'_1, r'_2} , r_2$ is the number of $\mathbb{Z}_2$-horizontal edges in $\left( \widetilde{U}^{(d, P)}_{(d, P)}\right)_{i, \alpha}^{r_1, r_2} , r'_2$is the number of $\mathbb{Z}_2$-horizontal edges in $\left( \widetilde{U}^{(d, P)}_{(d, P)}\right)_{j, \beta}^{r'_1, r'_2}$, $\alpha[\beta]$ is the partition corresponding to the diagram $\left( \widetilde{U}^{(d, P)}_{(d, P)}\right)_{i, \alpha}^{r_1, r_2} \left(\left( \widetilde{U}^{(d, P)}_{(d, P)}\right)_{j, \beta}^{r'_1, r'_2}\right)$ and $\alpha, \beta \in \Omega^{r_1, r_2}_{s_1, s_2}.$
\item[(b)] Since $\overrightarrow{J}^{2k}_{2s_1+s_2} \subset J^{2k}_{2s_1+s_2},$ we shall use the index defined above in (i) to index the diagrams of $\overrightarrow{J}^{2k}_{2s_1+s_2}.$
  \item[(c)] The diagrams in $J^{k}_{s}$ are indexed as follows:

\centerline{$\left\{ \left(U^{R^{d}}_{R^{d}} \right)_{i, \alpha}^r \ \Big| \ 1 \leq i \leq f^{r, \alpha}_{s} \text{ and } \alpha \in \Omega^r_s\right\}_{\ds 0 \leq r \leq k - s}$}

$(i, r, \alpha) < (j, r', \beta),$
\begin{enumerate}
  \item[(1)] if $r < r'$,
  \item[(2)] if $r = r'$ and $\alpha < \beta$ (lexicographic ordering)
  \item[(3)] if $r=r', \alpha = \beta,$ then it can be indexed arbitrarily
\end{enumerate}

\NI where $r(r')$ is the number of horizontal edges in $\left(U^{R^{d}}_{R^{d}} \right)_{i, \alpha}^r \left( \left(U^{R^{d}}_{R^{d}} \right)_{j, \beta}^{r'}\right)$, $\alpha(\beta)$ is the partition corresponding to the diagram $\left(U^{R^{d}}_{R^{d}} \right)_{i, \alpha}^r \left( \left(U^{R^{d}}_{R^{d}} \right)_{j, \beta}^{r'}\right)$ and $\alpha, \beta \in \Omega^{r}_{s}.$
\end{enumerate}
Now, $(d, P) \mapsto U^{(d, P)}_{(d, P)}$ gives a bijection of $M^k[(s, (s_1, s_2))]$ and $J^{2k}_{2s_1+s_2}.$
\end{defn}

\begin{note} \label{N2}
For the sake of simplicity, we shall write $\left(\widetilde{U}^{(d, P)}_{(d, P)}\right)_{i, \alpha}^{r_1, r_2}$ as $d_{i, \alpha}^{r_1, r_2}$ and $\left(U^{R^{d}}_{R^{d}} \right)_{i, \alpha}^r$ as $R^{d_{i, \alpha}^r}.$
\end{note}

We shall now explain the entries of the Gram matrices.

\begin{defn} \label{D3.8}
\mbox{ }
\begin{enumerate}
  \item[(a)] For $ 0 \leq s_1+s_2 \leq k,$ define  $G_{2s_1+s_2}^k$ (\textit{Gram matrices of the algebra of $\mathbb{Z}_2$-relations}) as follows:

\centerline{$G_{2s_1 + s_2}^k = \left( A_{2r_1 + r_2, 2r'_1+r'_2}\right)_{\substack{\ds 0 \leq r_1+r_2 \leq k-s_1-s_2 \\  \ds 0 \leq  r'_1 + r'_2 \leq k-s_1-s_2 }} $}
\NI where $\ds A_{2r_1 + r_2, 2r'_1 + r'_2}$ denotes the block matrix whose entries are $a_{(i, \alpha, r_1, r_2), (j, \beta, r'_1, r'_2)}$ with
\begin{center}
$\begin{array}{lllll}
  a_{(i, \alpha, r_1, r_2),  (j, \beta, r'_1, r'_2)} & = & x^{l(P_i \vee P_j)}  &  if &  \sharp^p \left( d^{r_1, r_2}_{i, \alpha} . d^{r'_1, r'_2}_{j, \beta}\right) = 2s_1 + s_2 \\
       & = & 0 & \text{Otherwise } i.e., &  \sharp^p \left( d^{r_1, r_2}_{i, \alpha} . d^{r'_1, r'_2}_{j, \beta} \right) < 2s_1 + s_2, \\
 \end{array}
$
\end{center}
\NI where  $1 \leq i \leq \Big| \mathbb{J}^{2r_1+r_2, \alpha}_{2s_1+s_2}\Big|, 1 \leq j \leq \Big| \mathbb{J}^{2r'_1+r'_2, \beta}_{2s_1+s_2}\Big|,  l(P_i \vee P_j) = l \left( d^{r_1, r_2}_{i, \alpha} . d^{r'_1, r'_2}_{j, \beta} \right),  l(P_i \vee P_j)$ denotes the number of connected components in $d^{r_1, r_2}_{i, \alpha} . d^{r'_1, r'_2}_{j, \beta}$ excluding the union of all the connected components of $P_i$ and $P_j$ or equivalently, $ l \left( d^{r_1, r_2}_{i, \alpha} . d^{r'_1, r'_2}_{j, \beta} \right)$ is the number of loops which lie in the middle row when $d^{r_1, r_2}_{i, \alpha}$ is multiplied with $d_{j, \beta}^{r'_1, r'_2}$, $d_{i, \alpha}^{r_1, r_2} \in \mathbb{J}^{2r_1 + r_2, \alpha}_{2s_1 +s_2}$ and $d_{j, \beta}^{r'_1, r'_2} \in \mathbb{J}^{2r'_1+ r'_2, \beta}_{2s_1 +s_2}$ respectively.

For example,

\begin{center}
\hspace{-3cm}\includegraphics[height=5cm, width=17cm]{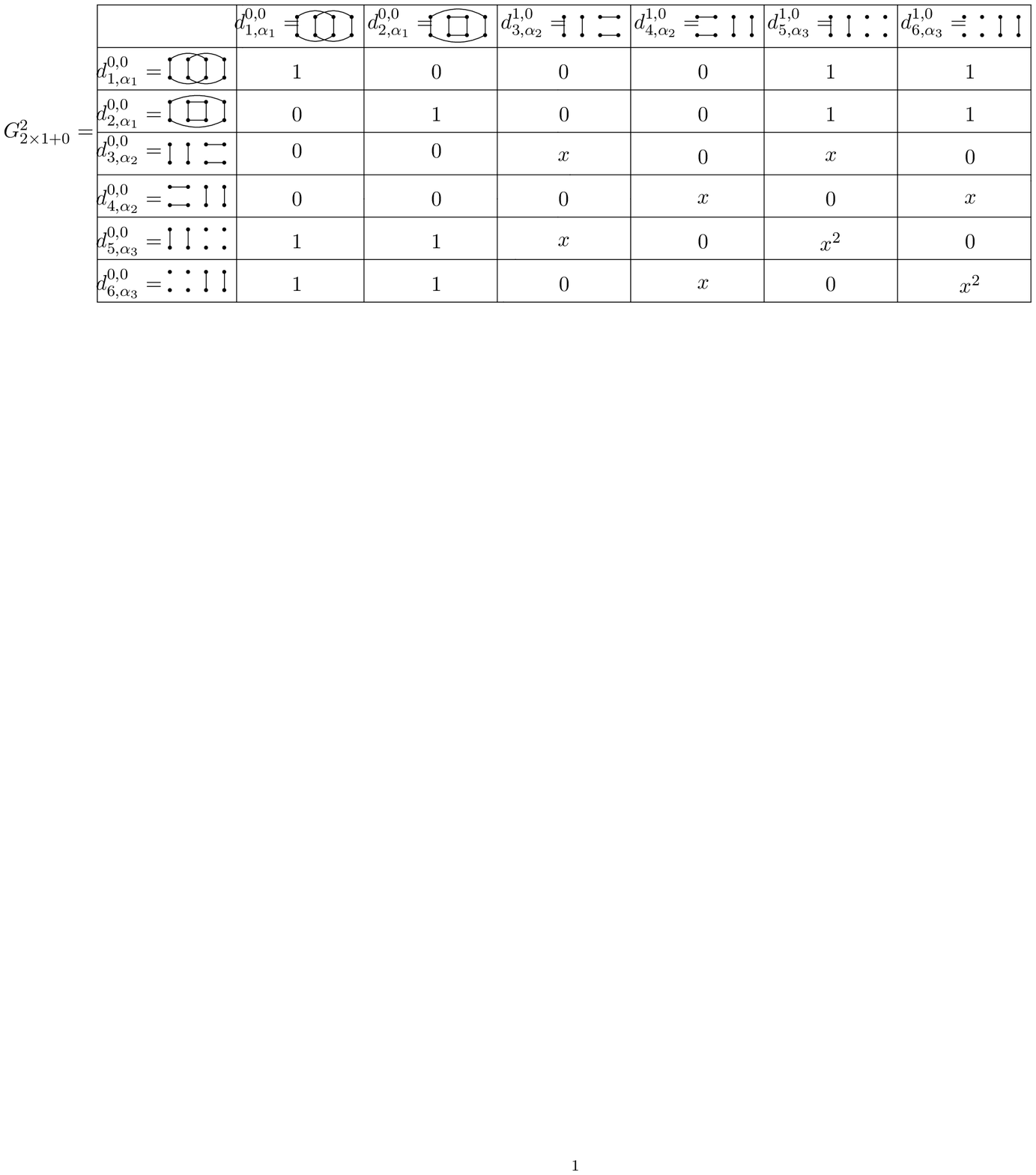}
\end{center}
\vspace{-1cm}
\NI where $\alpha_1 = (2, \Phi, \Phi, \Phi), \alpha_2= (1, \Phi, \Phi, 1)$ and $\alpha_3 = (1, \Phi, 1, \Phi).$
  \item[(b)] For $0 \leq s_1 \leq k, 0 \leq s_2 \leq k-1, 0 \leq s_1+s_2 \leq k-1,$ define $\overrightarrow{G}_{2s_1+s_2}^k$  (\textit{Gram matrices of signed partition algebra}) as follows:

\centerline{$\overrightarrow{G}^k_{2s_1 + s_2} = \left( \overrightarrow{A}_{2r_1 + r_2, 2r'_1+r'_2}\right)_{\substack{\hspace{-2.4cm}\ds 0 \leq r_1+r_2, r'_1 + r'_2 \leq k-1-s_1-s_2    \\ \ds 0 \leq r_1, r'_1 \leq k-s_1-s_2, 0 \leq r_2, r'_2 \leq k-s_1-s_2-1}}$} where $\overrightarrow{A}_{2r_1 + r_2, 2r'_1 + r'_2}$ denotes the block matrix whose entries are $a_{(i, \alpha, r_1, r_2), (j, \beta, r'_1, r'_2)}$ with
\begin{center}
$\begin{array}{lllll}
  a_{(i, \alpha, r_1, r_2), (j, \beta, r'_1, r'_2)} & = & x^{l(P_i \vee P_j)}  &  if &  \sharp^p \left( d_{i, \alpha}^{r_1, r_2} . d_{j, \beta}^{r'_1, r'_2} \right) = 2s_1 + s_2 \\
       & = & 0 & \text{Otherwise } i.e., &  \sharp^p \left( d_{i, \alpha}^{r_1, r_2} . d_{j, \beta}^{r'_1, r'_2} \right) < 2s_1 + s_2, \\
 \end{array}
$
\end{center}
\NI  $1 \leq i \leq \Big| \overrightarrow{\mathbb{J}}^{2r_1+r_2, \alpha}_{2s_1+s_2}\Big|, 1 \leq j \leq \Big| \overrightarrow{\mathbb{J}}^{2r'_1+r'_2, \beta}_{2s_1+s_2}\Big|,  l(P_i \vee P_j) = l \left( d_{i, \alpha}^{r_1, r_2} . d_{j, \beta}^{r'_1, r'_2} \right),  l(P_i \vee P_j)$ denotes the number of connected components in $d_{i, \alpha}^{r_1, r_2} . d_{j, \beta}^{r'_1, r'_2}$ excluding the union of all the connected components of $P_i$ and $P_j$ or equivalently, $ l \left( d_{i, \alpha}^{r_1, r_2} . d_{j, \beta}^{r'_1, r'_2} \right)$ is the number of loops which lie in the middle row when $d_{i, \alpha}^{r_1, r_2}$ is multiplied with $d_{j, \beta}^{r'_1, r'_2}$, $d_{i, \alpha}^{r_1, r_2} \in \overrightarrow{\mathbb{J}}^{2r_1 + r_2, \alpha}_{2s_1 +s_2}$ and $d_{j, \beta}^{r'_1, r'_2} \in \overrightarrow{\mathbb{J}}^{2r'_1+ r'_2, \beta}_{2s_1 +s_2}$ respectively.

 For example,
\vspace{-0.5cm}
 \begin{center}
 \includegraphics{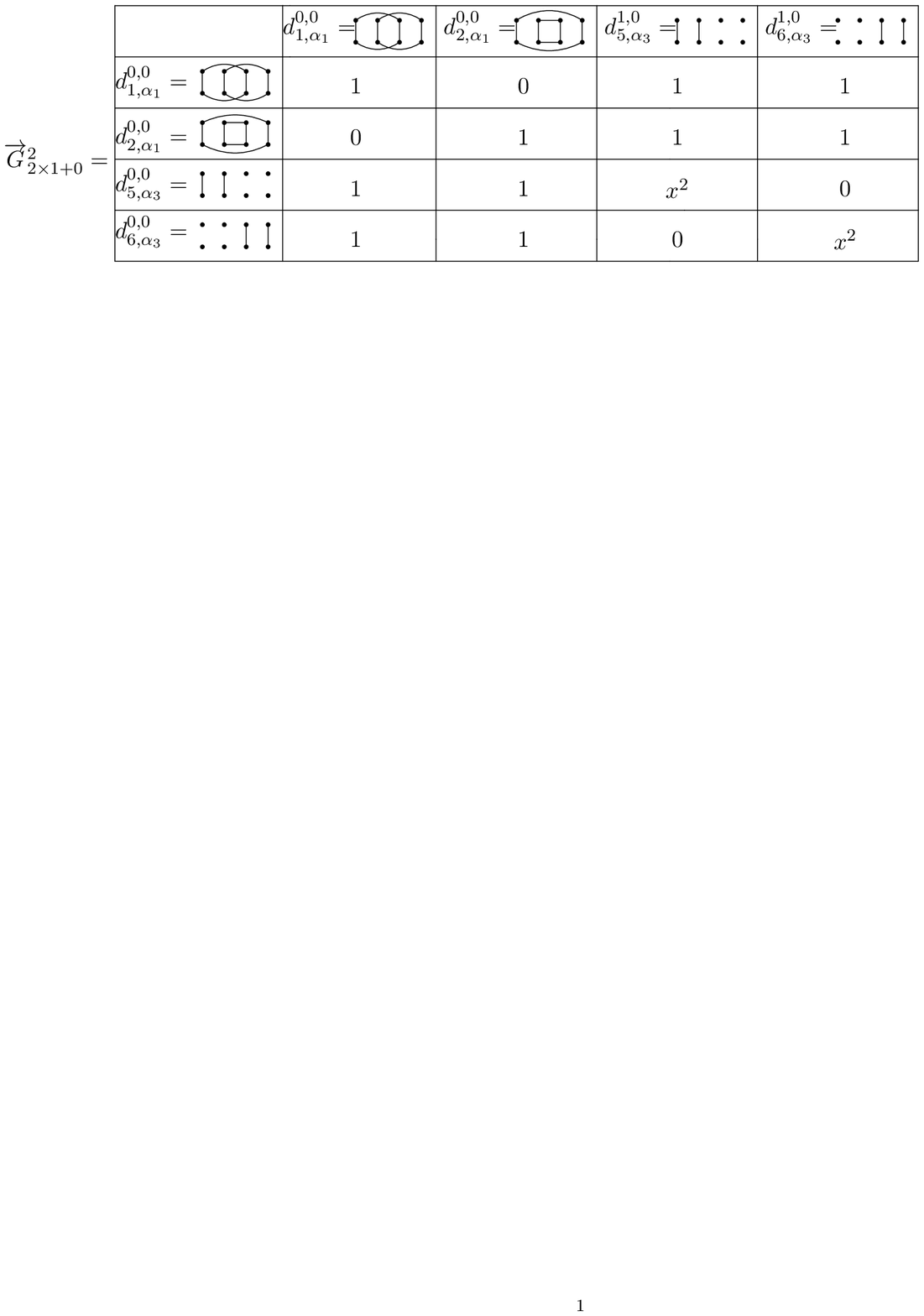}
 \end{center}
 \vspace{-1cm}
 \NI where $\alpha_1 = (2, \Phi, \Phi, \Phi), \alpha_2= (1, \Phi, \Phi, 1)$ and $\alpha_3 = (1, \Phi, 1, \Phi).$
  \item[(c)] For $0 \leq s \leq k,$ define $G_s^k$  (\textit{ Gram matrices of partition algebra}) as follows:

\centerline{$G_{s}^k = \left( A_{r, r'}\right)_{0 \leq r, r' \leq k-s} $}
\NI where $A_{r, r'}$ denotes the block matrix whose entries are $a_{(i, \alpha, r), (j, \beta, r')}$ with
\begin{center}
$\begin{array}{lllll}
  a_{(i, \alpha, r), (j, \beta, r')} & = & x^{l\left(R^{d_i} R^{d_j}\right)}  &  if &  \sharp^p \left( R^{d^r_{i, \alpha}}. R^{d^{r'}_{j, \beta}} \right) = s\\
       & = & 0 & Otherwise & i.e., \sharp^p \left( R^{d^r_{i, \alpha}}. R^{d^{r'}_{j, \beta}}\right) < s, \\
 \end{array}
$
\end{center}
\NI where  $1 \leq i \leq \Big| \mathbb{J}^{r, \alpha}_{s}\Big|, 1 \leq j \leq \Big| \mathbb{J}^{r', \beta}_{s}\Big|, l(R^{d_i} R^{d_j}) = l ( R^{d^r_{i, \alpha}}. R^{d^{r'}_{j, \beta}}), l(R^{d^r_{i, \alpha}}  R^{d^{r'}_{j, \beta}})$ denotes the number of connected components which lie in the middle row while multiplying  $R^{d^r_{i, \alpha}}$ with $ R^{d^{r'}_{j, \beta}},$ $R^{d^r_{i, \alpha}} \in \mathbb{J}^{r, \alpha}_{s}$ and $ R^{d^{r'}_{j, \beta}}\in \mathbb{J}^{r', \beta}_{s}$ respectively.
For example,

\begin{center}
\includegraphics{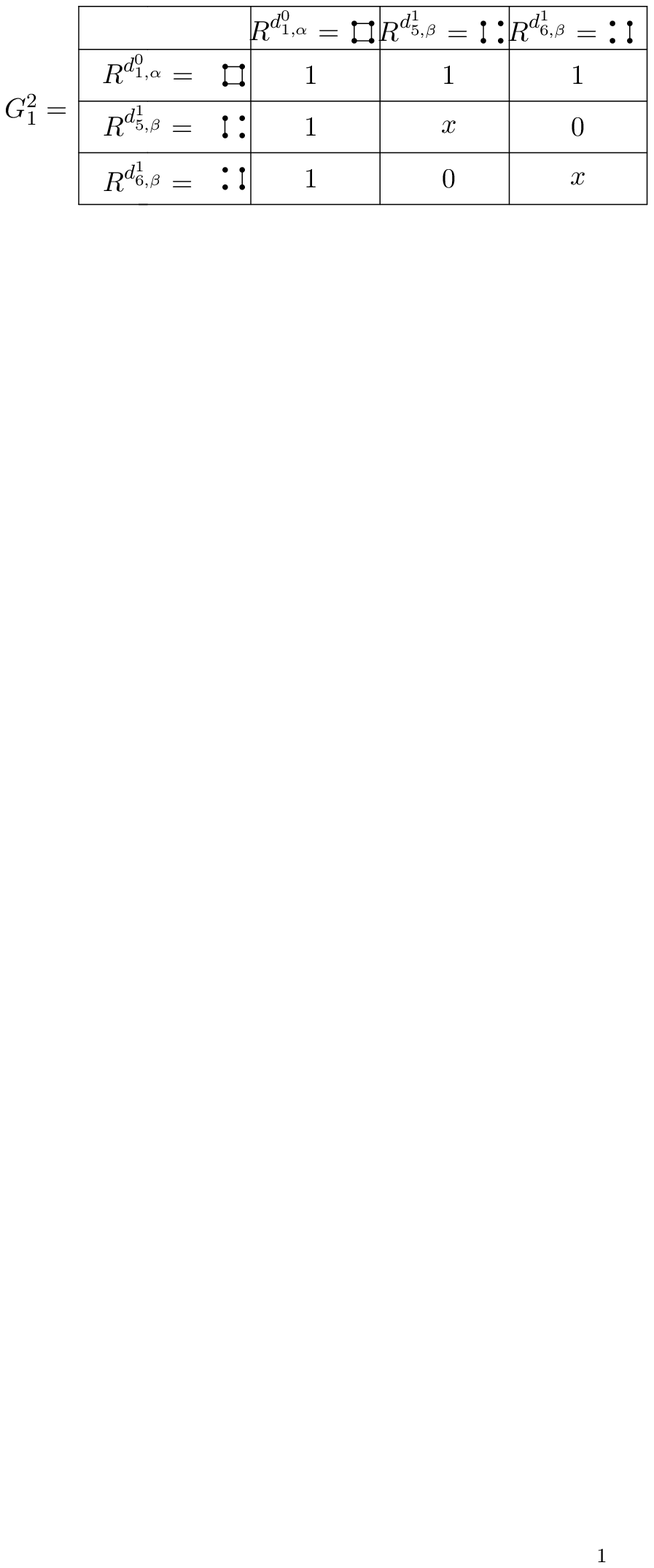}
\end{center}
\vspace{-1cm}
\end{enumerate}

\end{defn}

\NI We establish the non-singularity of the Gram matrices over the field $\mathbb{K}(x)$ where $x$ is an indeterminate.

\begin{lem} \label{L3.9}
\mbox{ }
\begin{enumerate}
  \item[(i)]\begin{enumerate}
              \item[(a)] For the algebra of $\mathbb{Z}_2$-relations, $l \left(d^{r_1, r_2}_{i, \alpha} . d^{r'_1, r'_2}_{j, \beta} \right)  <  l \left(d^{r_1, r_2}_{i, \alpha}. d^{r_1, r_2}_{i, \alpha} \right), \ \ \ \forall (j, \beta, r'_1, r'_2) < (i, \alpha, r_1, r_2),$ where  $l\left(d^{r_1, r_2}_{i, \alpha} . d^{r'_1, r'_2}_{j, \beta}\right)$ is the number of loops which lie in the middle row when $d^{r_1, r_2}_{i, \alpha}$ is multiplied with $d^{r'_1, r'_2}_{j, \beta}$ where $d^{r_1, r_2}_{i, \alpha}, d^{r'_1, r'_2}_{j, \beta} \in J^{2k}_{2s_1+s_2}$ and $J^{2k}_{2s_1+s_2}$ is as in Notation \ref{N3.6}(a).
              \item[(b)]For the signed partition algebras,  $l \left(d^{r_1, r_2}_{i, \alpha} . d^{r'_1, r'_2}_{j, \beta} \right)  <  l \left(d^{r_1, r_2}_{i, \alpha}. d^{r_1, r_2}_{i, \alpha} \right), \ \ \ \forall (j, \beta, r'_1, r'_2) < (i, \alpha, r_1, r_2),$ where  $l\left(d^{r_1, r_2}_{i, \alpha} . d^{r'_1, r'_2}_{j, \beta}\right)$ is the number of loops which lie in the middle row when $d^{r_1, r_2}_{i, \alpha}$ is multiplied with $d^{r'_1, r'_2}_{j, \beta}$ where $d^{r_1, r_2}_{i, \alpha}, d^{r'_1, r'_2}_{j, \beta} \in \overrightarrow{J}^{2k}_{2s_1+s_2}$ and $\overrightarrow{J}^{2k}_{2s_1+s_2}$ is as in Notation \ref{N3.6}(b).
              \item[(c)]For the partition algebras,  $l(R^{d^r_{i, \alpha}} . R^{d_{j, \beta}^{r'}})  <  l(R^{d^r_{i, \alpha}} . R^{d_{i, \alpha}^{r}}), \ \ \ \forall (j,  \beta, r') < (i,  \alpha, r),$ where  $l\left(R^{d^r_{i, \alpha}} . R^{d_{j, \beta}^{r'}}\right)$ is the number of loops which lie in the middle row when $R^{d^r_{i, \alpha}} $ is multiplied with $ R^{d_{j, \beta}^{r'}}$ where $R^{d^r_{i, \alpha}}, R^{d_{j, \beta}^{r'}} \in J^k_s$ and $J^k_s$ is as in Notation \ref{N3.6}(c).
            \end{enumerate}
      \item[(ii)]  $\det \ G_{2s_1+s_2}^k, \det \ \overrightarrow{G}_{2s_1+s_2}^k$ and $\det \ G_s^k$ are  non-zero polynomials with leading coefficient 1.

\end{enumerate}
\end{lem}
\begin{proof}

\NI \textbf{Proof of (i)(a):}

\NI A loop consists of at least one horizontal edge from the bottom row of $d_{i, \alpha}^{r_1, r_2}$ and one from the top row of $d_{j, \beta}^{r'_1, r'_2},$ hence the number of loops in the middle component of the product $d_{i, \alpha}^{r_1, r_2} . d_{j, \beta}^{r'_1, r'_2}$ is always less than the minimum of number of loops in $\left( d_{i, \alpha}^{r_1, r_2} . d_{i, \alpha}^{r_1, r_2}\right)$ and $\left(d_{j, \beta}^{r'_1, r'_2} . d_{j, \beta}^{r'_1, r'_2} \right).$

\NI Thus, $l(d_{i, \alpha}^{r_1, r_2} . d_{j, \beta}^{r'_1, r'_2}) \leq l(d_{i, \alpha}^{r_1, r_2} . d_{i, \alpha}^{r_1, r_2}) \text{ and } l(d_{i, \alpha}^{r_1, r_2} . d_{j, \beta}^{r'_1, r'_2}) \leq l(d_{j, \beta}^{r'_1, r'_2} . d_{j, \beta}^{r'_1, r'_2}), \ \ \ \forall \ i, j.$

\NI For if $(j, \beta, r'_1, r'_2) < (i, \alpha, r_1, r_2) $

\NI \textbf{Case (i):}  when $2r'_1+r'_2 < 2r_1 + r_2$ where $r_1(r'_1)$ is the number of pairs of $\{e\}$ horizontal edges  and $r_2(r'_2)$ is the number of $\mathbb{Z}_2$-horizontal edges in $d^{r_1, r_2}_{i, \alpha}\left( d^{r'_1, r'_2}_{j, \beta}\right)$ respectively, then

\centerline{$l(d_{i, \alpha}^{r_1, r_2} . d_{j, \beta}^{r'_1, r'_2}) \leq l(d_{j, \beta}^{r'_1, r'_2} . d_{j, \beta}^{r'_1, r'_2}) < l(d_{i, \alpha}^{r_1, r_2} . d_{i, \alpha}^{r_1, r_2}).$}

 \NI \textbf{Case (ii):} when $2r'_1+r'_2 = 2r_1 + r_2$ and $r'_1+r'_2 < r_1+r_2$ where $r_1(r'_1)$ is the number of pairs of $\{e\}$ horizontal edges  and $r_2(r'_2)$ is the number of $\mathbb{Z}_2$-horizontal edges in $d^{r_1, r_2}_{i, \alpha}\left( d^{r'_1, r'_2}_{j, \beta}\right)$ respectively, which implies that

 \NI {\bf Subcase (i):}Suppose that $r'_2 < r_2,$ i.e., atleast two $\mathbb{Z}_2$-horizontal edges of $ d^{r'_1, r'_2}_{j, \beta}$ is connected to a $\mathbb{Z}_2$-horizontal edge of $d^{r_1, r_2}_{i, \alpha}$  to make a loop or one $\mathbb{Z}_2$-horizontal edge of $d^{r_1, r_2}_{i, \alpha}$ is connected to a $\mathbb{Z}_2$-through class of $ d^{r'_1, r'_2}_{j, \beta}$ in the product $d^{r_1, r_2}_{i, \alpha} . d^{r'_1, r'_2}_{j, \beta}.$

  \NI {\bf Subcase (ii):}Suppose that $r'_1 < r_1,$ i.e., atleast two $\{e\}$ horizontal edges of $ d^{r'_1, r'_2}_{j, \beta}$ is connected to a $\{e\}$ or $\mathbb{Z}_2$-horizontal edge of $d^{r_1, r_2}_{i, \alpha}$ to make a loop or one $\{e\}$-horizontal edge of $d^{r_1, r_2}_{i, \alpha}$ is connected to a $\{e\}$ or $\mathbb{Z}_2$-through class of $ d^{r'_1, r'_2}_{j, \beta}$ in the product $d^{r_1, r_2}_{i, \alpha} . d^{r'_1, r'_2}_{j, \beta}.$

   \NI Therefore the number of loops is strictly less than $2r'_1 + r'_2,$ and thus

    \centerline{$l\left(d^{r_1, r_2}_{i, \alpha} . d^{r'_1, r'_2}_{j, \beta} \right) \lvertneqq l\left( d^{r'_1, r'_2}_{j, \beta} d^{r'_1, r'_2}_{j, \beta} \right) = l\left(d^{r_1, r_2}_{i, \alpha} . d^{r_1, r_2}_{i, \alpha} \right)$}

   \NI \textbf{Case (iii):} when $2r'_1+r'_2 = 2r_1+r_2, r'_1+r'_2 = r_1+r_2$ and $\alpha < \beta$ where $r_1(r'_1)$ is the number of pairs of $\{e\}$ horizontal edges  and $r_2(r'_2)$ is the number of $\mathbb{Z}_2$-horizontal edges in $d^{r_1, r_2}_{i, \alpha} \left( d^{r'_1, r'_2}_{j, \beta}\right)$ respectively and $\alpha(\beta)$ is the underlying partition of $d^{r_1, r_2}_{i, \alpha} \left( d^{r'_1, r'_2}_{j, \beta}\right)$, which implies that

    $l\left(d^{r_1, r_2}_{i, \alpha} . d^{r_1, r_2}_{i, \alpha} \right) = l\left(d^{r'_1, r'_2}_{j, \beta} . d^{r'_1, r'_2}_{j, \beta} \right) = 2r_1+r_2 = 2r'_1+r'_2$ and $r_1+r_2 = r'_1+r'_2.$

\NI   Every $\{e\}$-through class of $U^{(d_i, P_i)}_{(d_i, P_i)}$ is uniquely connected to a $\{e\}$-through class of $ d^{r'_1, r'_2}_{j, \beta}$ and vice versa and if $l\left(d^{r_1, r_2}_{i, \alpha} . d^{r'_1, r'_2}_{j, \beta} \right) = l\left(d^{r_1, r_2}_{i, \alpha} . d^{r_1, r_2}_{i, \alpha}  \right) = l\left( d^{r'_1, r'_2}_{j, \beta} . d^{r'_1, r'_2}_{j, \beta} \right)$ then every $\{e\}(\mathbb{Z}_2)$-horizontal edge of $d^{r_1, r_2}_{i, \alpha} $ is connected uniquely to a $\{e\}(\mathbb{Z}_2)$-horizontal edge of $ d^{r'_1, r'_2}_{j, \beta}$ and vice versa which implies that $d^{r_1, r_2}_{i, \alpha} = d^{r'_1, r'_2}_{j, \beta}.$

   \NI Thus, if $d^{r_1, r_2}_{i, \alpha} \neq d^{r'_1, r'_2}_{j, \beta}$ and $2r_1+r_2 = 2r'_1+r'_2$ and $r_1+r_2 = r'_1+r'_2$ then $l\left(d^{r_1, r_2}_{i, \alpha} . d^{r'_1, r'_2}_{j, \beta} \right) < l\left(d^{r_1, r_2}_{i, \alpha} . d^{r_1, r_2}_{i, \alpha}\right) = l\left(d^{r'_1, r'_2}_{j, \beta} . d^{r'_1, r'_2}_{j, \beta} \right).$

\NI   Proof of (i)(b) and (i)(c) are similar to the proof of (i)(a).

\NI \textbf{Proof of (ii):} It follows from (i) of Lemma \ref{L3.9}, that the degree of the monomial $\left\{ \prod a_{i \sigma(i)}\right\}_{\sigma \in \mathfrak{S}_{f_{2s_1+s_2}}}, $ is strictly less than the degree of the monomial $\underset{i=1}{\overset{f_{2s_1+s_2}}{\prod}}  a_{ii}.$

\NI Thus, the determinant of the Gram matrix $G_{2s_1+s_2}^k$ of the algebra of $\mathbb{Z}_2$-relations is a non-zero monic polynomial  with integer coefficients and the roots are all algebraic integers.

\NI Similarly, we can prove for the determinant of the Gram matrices $\overrightarrow{G}_{2s_1+s_2}^k$ and $G_s^k$ of signed partition algebras and partition algebras respectively.

 \end{proof}

\begin{lem}\label{L3.10}

 The Gram matrices $G_{2s_1+s_2}^k, \overrightarrow{G}_{2s_1 + s_2}^k$ and $G_s^k$ are symmetric.
 \end{lem}

 \begin{proof}
 The proof follows from the Definition \ref{D3.8}, since the top and bottom rows of the diagrams in $J^{2k}_{2s_1+s_2}, \overrightarrow{J}^{2k}_{2s_1+s_2}, J^{k}_{s}$ have the same number of horizontal edges.
 \end{proof}

  \begin{rem}\label{R3.11}
Every partition diagram can be represented as a set partition and in set partition we can talk about subsets.
Thus a connected component of the diagram $ d^{r'_1, r'_2}_{j, \beta}$ is contained in a connected component of $d^{r_1, r_2}_{i, \alpha}$ if the corresponding set partition of $d^{r'_1, r'_2}_{j, \beta}$ is contained in the set partition of $d^{r_1, r_2}_{i, \alpha}.$
\end{rem}

We introduce a finer version of coarser diagrams in order to prove the main result of this paper.

\begin{defn}\label{D3.12}
\mbox{ }
\begin{enumerate}
\item[(a)] Let $d^{r_1, r_2}_{i, \alpha}, d^{r'_1, r'_2}_{j,\beta} \in J^{2k}_{2s_1+s_2}.$

 Define a relation on $J^{2k}_{2s_1+s_2}$ as follows: $d^{r_1, r_2}_{i, \alpha} < d^{r'_1, r'_2}_{j,\beta},$
\begin{enumerate}
   \item[(i)] if each $\{e\}$- through class of $d^{r_1, r_2}_{i, \alpha}$ is contained in a $\{e\}$-through class of $d^{r'_1, r'_2}_{j,\beta}$,
   \item[(ii)]  every $\mathbb{Z}_2$-through class  of $d^{r_1, r_2}_{i, \alpha}$ is contained in a $\mathbb{Z}_2$-through class of $d^{r'_1, r'_2}_{j,\beta}$
   \item[(iii)] every $\{e\}$-horizontal edge of $d^{r_1, r_2}_{i, \alpha}$ is contained in a ($\{e\}$ or $\mathbb{Z}_2$) horizontal edge or ($\{e\}$ or $\mathbb{Z}_2$)-through class of $ d^{r'_1, r'_2}_{j,\beta}$ and
   \item[(iv)] every $\mathbb{Z}_2$-horizontal edge of $d^{r_1, r_2}_{i, \alpha}$ is contained in a $\mathbb{Z}_2$-horizontal edge or $\mathbb{Z}_2$-through class of $ d^{r'_1, r'_2}_{j,\beta}.$
 \end{enumerate}

\NI We say that $d^{r'_1, r'_2}_{j,\beta}$ is a coarser diagram of $d^{r_1, r_2}_{i, \alpha}$ and $(j, \beta, r'_1, r'_2) < (i, \alpha, r_1, r_2).$
\item[(b)] since $\overrightarrow{J}^{2k}_{2s_1+s_2} \subset J^{2k}_{2s_1+s_2}$ the relation defined on $J^{2k}_{2s_1+s_2}$ in (a) holds for the diagrams in $\overrightarrow{J}^{2k}_{2s_1+s_2}.$
\item[(c)]   Define a relation on $J^k_s$ as follows: $R^{d^r_{i, \alpha}} < R^{d^{r'}_{j, \beta}}$,
\begin{enumerate}
  \item[(i)$'$] if each through class of $R^{d^{r}_{i, \alpha}}$ is contained in a through class of $R^{d^{r'}_{j, \beta}},$
  \item[(ii)$'$] if each horizontal edge of $R^{d^r_{i, \alpha}}$ is contained in a horizontal edge or through class of $R^{d^{r'}_{j, \beta}},$
\end{enumerate}
\NI We say that $R^{d^{r'}_{j, \beta}}$ is a coarser diagram of $R^{d^r_{i, \alpha}}$ then $(j, \beta, r') < (i, \alpha, r).$ The relation $ < $ holds for the diagrams in $\widetilde{J}^{k}_{s}.$

\end{enumerate}

\end{defn}

\begin{ex}\label{E3.14}
This example illustrates Definition \ref{D3.12}.
\begin{center}
\includegraphics{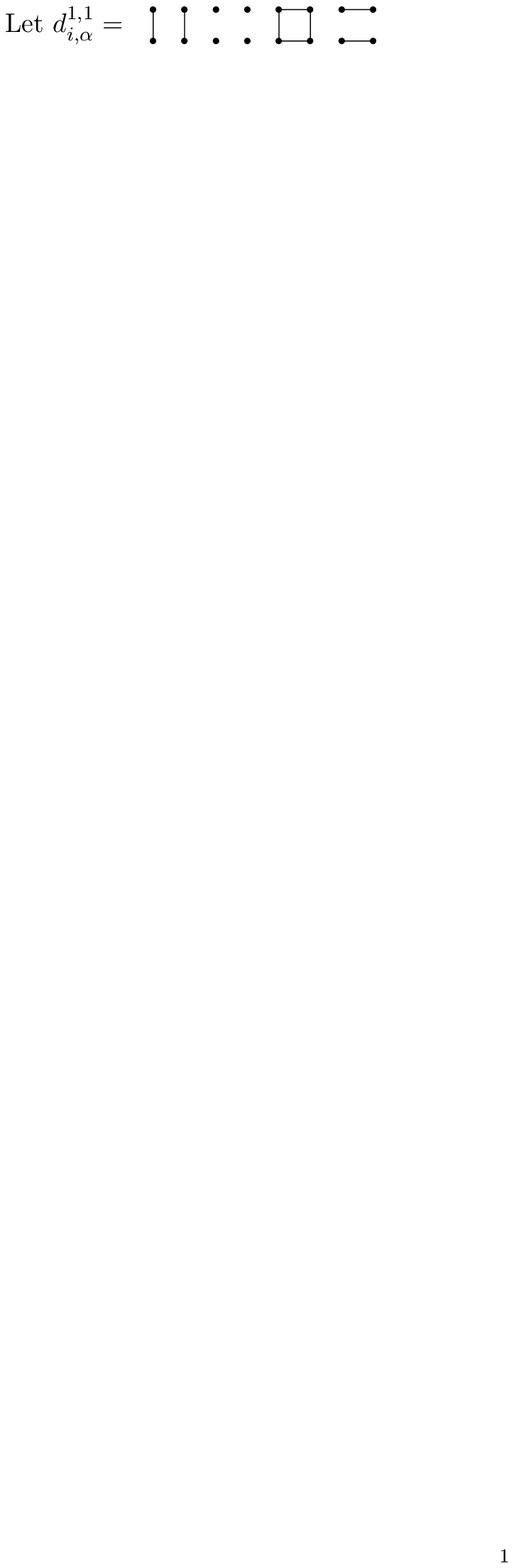}
\end{center}
The following are the diagrams coarser than $d^{1, 1}_{i, \alpha} \in J^8_{2 \times 1+ 1}$ where $\alpha = (1, 1, 1, 1).$
\begin{center}
\includegraphics{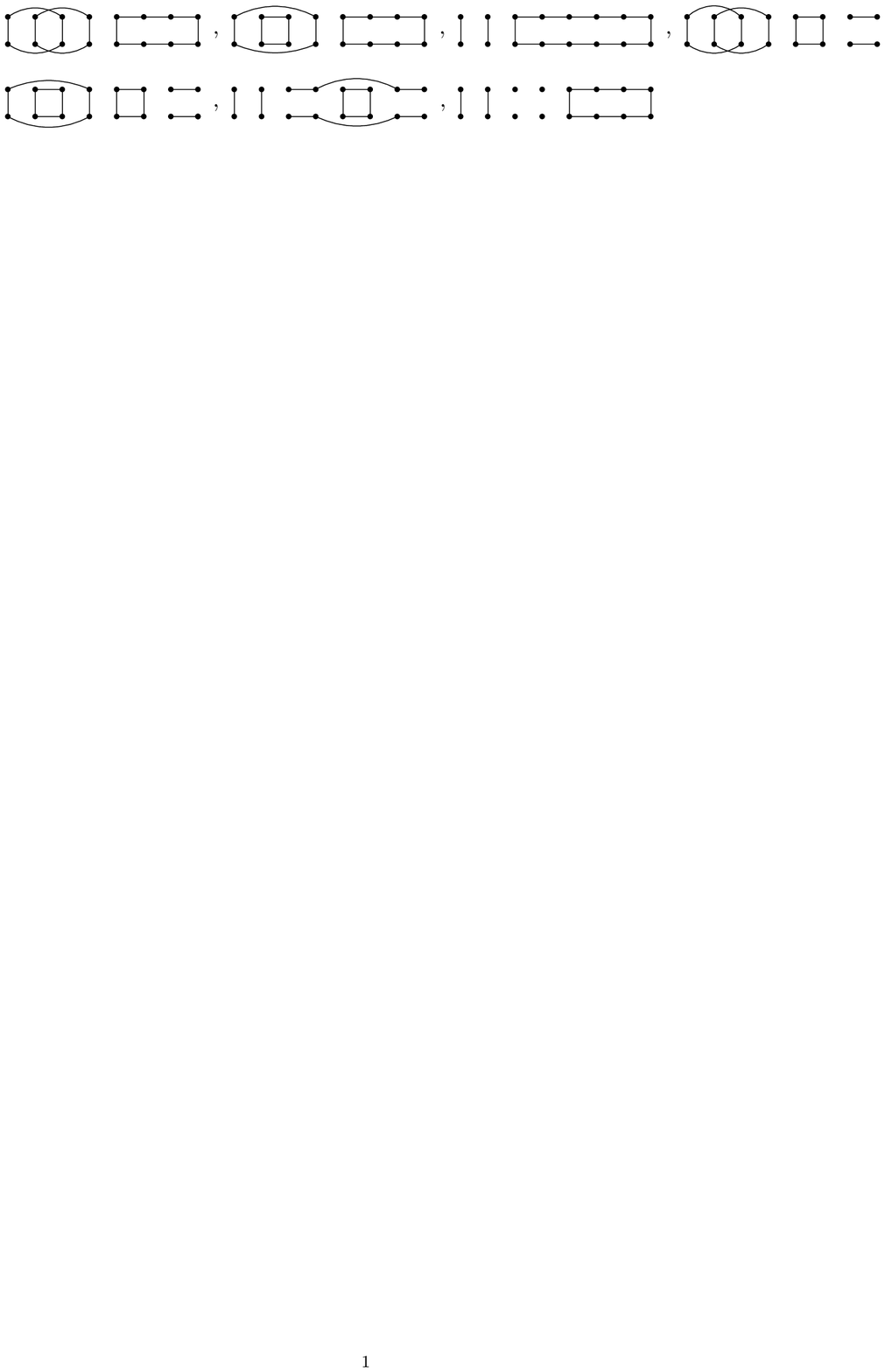}
\end{center}

\end{ex}
\vspace{-0.5cm}
\NI While working in arriving at the main result Theorem \ref{T3.32},  we obtain a beautiful combinatorial result.

\subsection{\textbf{Stirling numbers of second kind of the algebra of $\mathbb{Z}_2$-relations, signed partition algebras and partition algebras:}}

\begin{lem} \label{L3.14}
\mbox{ }
\begin{enumerate}
  \item[(a)]
  In the algebra of $\mathbb{Z}_2$-relations the number of diagrams having $2s_1 + s_2$ through classes and $2 p_1 + p_2$ horizontal edges which lie above and coarser than the diagram $ d^{r_1, r_2}_{i, \alpha}$ whose underlying partition is $\alpha$, where $\alpha \in \Omega^{r_1, r_2}_{s_1, s_2}$   is given by,

\begin{equation}\label{e3.1}
     \underset{ i = p_1}{\overset{r_1}\sum} 	{_{r_1}}C_{i} \ 2^{i - p_1} \ S(i, p_1) \left[ \underset{ j = 0}{\overset{r_1 - i}\sum} {_{r_1 - i}}C_j \ (2s_1 + s_2)^{r_1-i-j} \left[ \sum_{\substack{p_2 - j \leq l \leq r_2\\p_2-j \geq 0}} {_{r_2}}C_l \ s_2^{r_2 - l} \ S(l+j, p_2) \right] \right]
\end{equation}
with $p_1 \leq r_1$ and $r_1 - p_1 \geq p_2 - r_2$ where $S(i, p_1)$ and $S(l+j, p_2)$ are the Stirling numbers of the second kind.

  In particular,
\begin{enumerate}
  \item[(i)] if $r_1 = 0$ then  the number of diagrams having $p_2$ number of $\mathbb{Z}_2$-horizontal edges which lie above and coarser than the given $d^{0, r_2}_{i, \alpha}$ with $r_2$ number of $\mathbb{Z}_2$-horizontal edges is given by

\centerline{$\underset{ i = p_2}{\overset{r_2}\sum} \ {_{r_2}}C_{l} \ s_2^{r_2 - i} \ S(i, p_2).$}

  \item[(ii)] if $s_1 = 0, s_2 = 0$ then the number of diagrams having $2p_1+p_2$ horizontal edges which lie above and coarser than $d^{r_1, r_2}_{i, \alpha}$ is given by

\centerline{$\underset{i=p_1}{\overset{r_1}{\sum}} {_{r_1}}C_i \ 2^{i-p_1} \ S(i, p_1)  \ S(r_2 + r_1-i, p_2).$}
\item[(b)]In the partition algebras, the number of diagrams having $s$ through classes and $p$ horizontal edges which lie above and coarser than the diagram $R^{d^{r}_{i, \alpha}}$  whose underlying partition is $\alpha$, where $\alpha = [\alpha_1]^1 [\alpha_2]^2$ is as in Definition \ref{D3.1}(c) is given by,

\begin{equation}\label{e3.2}
     \underset{ i = p}{\overset{r}\sum} \ {_{r}}C_{l} \ s^{r - i} \ S(i, p).
\end{equation}
with $p \leq r$  where $S(i, p)$ is the Stirling numbers of the second kind.
\end{enumerate}

\end{enumerate}
\end{lem}

\begin{proof}
\NI \textbf{Step 1:} Reducing $2 r_1$ number of $\{e\}$-horizontal edges to $2 p_1$ number of $\{e\}$-horizontal edges.

\NI Choose $i$ pairs of $\{e\}$-horizontal edges from $r_1$ pair of $\{e\}$-horizontal edges of $d^{r_1, r_2}_{i,\alpha}$ such that $i \geq p_1.$

\NI For given $i$ pairs of $\{e\}$-horizontal edges, the number of ways to partition a set of $i$ pairs of $\{e\}$-horizontal edges into $p_1$ pairs of $\{e\}$-horizontal edges is given by the Stirling number of second kind $S(i, p_1).$

\NI We know that two pairs of $\{e\}$-horizontal edges can be combined together in two ways. Thus, $p_1$ number of $\{e\}$-horizontal edges can be obtained in $2^{i - p_1}$ ways.

\NI \textbf{Step 2:} Reducing $2(r_1 - i)$ number of $\{e\}$-horizontal edges together with $r_2$ number of $\mathbb{Z}_2$horizontal edges to obtain $p_2$ number of $\mathbb{Z}_2$-horizontal edges.

\NI Choose $j$ pair of $\{e\}$-horizontal edges from the $(r_1 - 1)$ pair of $\{e\}$-horizontal edges such that $0 \leq j \leq r_1 - i.$

\NI Choose $l$ number of $\mathbb{Z}_2$-horizontal edges from the $r_2$ number of horizontal edges such that $l \geq r_2 - j.$

\NI The number of ways to partition a set of $j$ pair of $\{e\}$-horizontal edges together with $l$ number of $\mathbb{Z}_2$-horizontal edges into $p_2$ number of $\mathbb{Z}_2$-horizontal edges are given by the Stirling number of the second kind $S(l+j, p_2).$

\NI \textbf{Step 3:} Combining the remaining horizontal edges with through classes.

\NI By combining the remaining $r_1 - i - j$ pair of $\{e\}$-horizontal edges to any one of the through classes, we obtain $(2s_1 + s_2)^{r_1 - i - j}$ number of diagrams.

\NI Also the remaining $r_2 - l$ number of $\mathbb{Z}_2$-horizontal edges can be combined only with $\mathbb{Z}_2$-through classes which can be done in $s_2^{r_2 - l}$ ways.

\NI Proof of (i) and (ii) follows from the proof of (a) and proof of (b) is similar to the proof of (a).
\end{proof}

\begin{notation}\label{N3.15}
\mbox{ }
\begin{enumerate}
  \item[(i)] \begin{enumerate}
               \item[(a)] The \textit{Stirling number of the second kind of algebra of $\mathbb{Z}_2$ relations} is denoted by $B_{2r_1 + r_2, 2p_1 + p_2}^{s_1, s_2}  $ where

                   $B_{2r_1 + r_2, 2p_1 + p_2}^{s_1, s_2} = $

\centerline{$  \underset{ i = p_1}{\overset{r_1}\sum} 	{_{r_1}}C_{i} \ 2^{i - p_1} \ S(i, p_1) \left[ \underset{ j = 0}{\overset{r_1 - i}\sum} {_{r_1 - i}}C_j \ (2s_1 + s_2)^{r_1-i-j} \left[ \underset{\substack{\ds p_2 - j \leq l \leq r_2\\ \ds p_2-j \geq 0}}{\sum} {_{r_2}}C_l \ s_2^{r_2 - l} \ S(l+j, p_2) \right] \right]$}

\NI with $p_1 \leq r_1, r_1 - p_1 \geq p_2 - r_2, 0 \leq s_1 \leq k , 0 \leq s_2 \leq k$ and $r_1+r_2+s_1+s_2 \leq k , 2p_1+p_2 < 2r_1+r_2.$
\item[(b)]The \textit{Stirling number of the second kind of signed partition algebra} is same as the Stirling number of the second kind of the algebra of $\mathbb{Z}_2$-relations.
               \item[(c)] The \textit{Stirling number of the second kind of  partition algebra} is denoted by $B_{r, p}^{s}  $ where

    $B^s_{r, p} =  \underset{ i = p}{\overset{r}\sum} \ {_{r}}C_{l} \ s^{r - i} \ S(i, p)$ with $p \leq r$ and $0 \leq s \leq k.$
             \end{enumerate}

  \item[(i)] $B^{s_1, s_2}_{2r_1+r_2, 2r_1+r_2} = 1$,  and $B^{s}_{r, r} = 1.$
  \item[(ii)] $B^{s_1, s_2}_{2r_1+r_2, 2p_1+p_2} = 0$ and $B^s_{r, p} = 0$ otherwise.
  \end{enumerate}

\end{notation}

\begin{lem}\label{L3.16} Let $B_{2r_1 + r_2, 2p_1 + p_2}^{s_1, s_2}$ and $B^s_{r, p}$ be as in Notation \ref{N3.15}. Then
\begin{enumerate}
  \item[(a)] $B^{s_1, s_2}_{2r_1 + r_2, 2p_1 + p_2} = B^{s_1, s_2}_{2r_1 + r_2-1, 2p_1+ p_2-1} + (s_2 + p_2) B^{s_1, s_2}_{2r_1 + r_2-1, 2p_1 + p_2}, \ \ \ p_1 \leq r_1 \text{ and }r_2 \geq 1.$

In particular, if $r_1 = 0 $ and $p_1 = 0$ then

\centerline{$B^{s_1, s_2}_{0+ r_2, 0+ p_2 } = B^{s_1, s_2}_{0+ r_2 - 1, 0+ p_2 - 1} + (s_2 + p_2) B^{s_1, s_2}_{0+ r_2 - 1, 0 + p_2}.$}

  \item[(b)] $B^{s}_{r, p} = B^{s}_{r-1, p-1} + (s + p) B^{s}_{r-1, p}, \ \ \ p \leq r.$
\end{enumerate}
\end{lem}

\begin{proof}
\NI \textbf{Proof of (a):}Consider

$ B^{s_1, s_2}_{2r_1 + r_2, 2p_1 + p_2}$
\begin{eqnarray*}
  &=& \underset{ i = p_1}{\overset{r_1}\sum} 	{_{r_1}}C_{i} \ 2^{i - p_1} \ S(i, p_1) \left[ \underset{ j = 0}{\overset{r_1 - i}\sum} {_{r_1 - i}}C_j \ (2s_1 + s_2)^{r_1-i-j} \left[ \sum_{\substack{p_2 - j \leq l \leq r_2\\p_2-j \geq 0}} {_{r_2}}C_l \ s_2^{r_2 - l} \ S(l+j, p_2) \right] \right]
     \end{eqnarray*}
   By using the identities ${_{r_2}}C_l = {_{r_2 - 1}}C_{l-1} + {_{r_2 - 1}}C_l$ and $S(l+j, p_2) = S(l+j-1, p_2-1) + p_2 S(l+j-1, p_2)$ we have,

 \centerline{$B^{s_1, s_2}_{2r_1 + r_2, 2p_1 + p_2} = B^{s_1, s_2}_{2r_1 + r_2-1, 2p_1+ p_2-1} + (p_2 + s_2) B^{s_1, s_2}_{2r_1 + r_2-1, 2p_1 + p_2}, \ \ \ p_1 \leq r_1 \text{ and } r_2 \geq 1.$}

 \NI  Proof of (b) is same as that of proof of (a).
\end{proof}

By example 3.10.4 in \cite{St}, $B^s_{r, p}$ will be called as \textbf{Generalized Stirling numbers} and $B^{s_1, s_2}_{2r_1 + r_2, 2p_1 + p_2}$ will be called as \textbf{$(s_1, s_2, r_1, r_2, p_1, p_2)$-Stirling numbers of the second kind} and it satisfies the following identity:

\begin{lem} \label{L3.17}
Let $B^{s_1, s_2}_{2r_1 + r_2, 2p_1 + p_2}$  be as in Notation \ref{N3.15}. Then

 \centerline{ $B^{s_1, s_2}_{2r_1 + r_2, 2p_1 + p_2} = B^{s_1, s_2}_{2(r_1-1) + r_2, 2(p_1 - 1)+ p_2} + B^{s_1, s_2}_{2(r_1 - 1) + r_2 + 1, 2p_1 + p_2 }+ (2 p_1 + 2 s_1) B^{s_1, s_2}_{2(r_1 -1) + r_2, 2p_1 + p_2},
$}

\NI with $p_1 \leq r_1 - 1$ and $(r_1 -1 ) - p_1 \geq p_2 - r_2.$

In particular,
\begin{enumerate}
  \item [(i)] if $p_2 = 0$ then

  \centerline{$B^{s_1, s_2}_{2r_1 + r_2, 2p_1 } = B^{s_1, s_2}_{2(r_1-1) + r_2, 2(p_1 - 1)} + (2 p_1 + 2 s_1 + s_2) B^{s_1, s_2}_{2(r_1 -1) + r_2, 2p_1 }, \ \ \ \ p_1 \leq r_1 - 1. $}
  \item[(ii)] if $r_2 = 0$ and $p_2 = 0$ then

  \centerline{$B^{s_1, s_2}_{2r_1 , 2p_1 } = B^{s_1, s_2}_{2(r_1-1) , 2(p_1 - 1)} + (2 p_1 + 2 s_1 + s_2) B^{s_1, s_2}_{2(r_1 -1) , 2p_1}, \ \ \ \ p_1 \leq r_1 - 1. $}
\end{enumerate}

\end{lem}

\begin{proof}
Consider

$B^{s_1, s_2}_{2r_1 + r_2, 2p_1 + p_2}$
\begin{eqnarray*}
   &=&  \underset{ i = p_1}{\overset{r_1}\sum} 	{_{r_1}}C_{i} \ 2^{i - p_1} \ S(i, p_1) \left[ \underset{ j = 0}{\overset{r_1 - i}\sum} {_{r_1 - i}}C_j \ (2s_1 + s_2)^{r_1-i-j} \left[ \sum_{\substack{p_2 - j \leq l \leq r_2\\p_2-j \geq 0}} {_{r_2}}C_l \ s_2^{r_2 - l} \ S(l+j, p_2) \right] \right] \\
   \end{eqnarray*}

By using the identities ${_{r_1}}C_i = {_{r_1 - 1}}C_{i-1} + {_{r_1 - 1}}C_i$, $ S(i, p_1) = S(i - 1, p_1 -1 ) + p_1 S(i - 1, p_1)$ and  Lemma \ref{L3.16} we have,

$B^{s_1, s_2}_{2r_1 + r_2, 2p_1 + p_2} = B^{s_1, s_2}_{2(r_1-1) + r_2, 2(p_1 - 1)+ p_2} + B^{s_1, s_2}_{2(r_1 - 1)\\ + r_2 + 1, 2p_1 + p_2 }+ (2 p_1 + 2 s_1) B^{s_1, s_2}_{2(r_1 -1) + r_2, 2p_1 + p_2}, $

\NI with $p_1 \leq r_1 - 1$ and $(r_1 -1 ) - p_1 \geq p_2 - r_2.$

\end{proof}

\begin{ex}\label{E3.18}

For fixed $s_1, s_2$, the table below gives the value of $B^{s_1, s_2}_{2r_1+r_2, 2p_1+p_2}$ as in Notation \ref
{N3.15}: \\

\hspace{-0.5cm}
\tiny{\begin{tabular}{|l|c|c|c|c|c|c|c|c|}
  \hline
  $2r_1+r_2 \Big \backslash 2p_1+p_2$ & $2.1+2$ & $2.2+0$ & $0+3$ & $2.1+1$ & $2.1+0$ & $0+2$ & $0+1$ & $0+0$ \\\hline
  $2.1+2$ & $1$ & $0$ & $1$ & $2s_2$ & $s_2^2$ & $2s_1+3s_2+3$ & $4s_1s_2+3s_2^2+$ & $2s_1s_2^2+s_2^3$ \\
  & & & & & & &$2s_1+3s_2+1$ & \\\hline
  $2.2+0$ & $0$ &$1$  & $0$ & $2$ & $4s_1+2s_2+2$ & $1$ & $4s_1+2s_2+1$ & $(2s_1+s_2)^2$ \\\hline
  $0+3$ & $0$ & $0$ & $1$ & $0$ & $0$ & $3s_2 + 3$ & $3s_2^2 + 3s_2 +1$ & $s_2^3$ \\\hline
  $2.1+1$ & $0$ & $0$ & $0$ & $1$ & $s_2$ & $1$ & $2s_1+2s_2+1$ &  $(2s_1+s_2)s_2$\\\hline
  $2.1+0$ & $0$ & $0$ & $0$ & $0$ & $1$ & $0$ & $1$ & $2s_1+s_2$ \\\hline
  $0+2$ & $0$ & $0$ & $0$ & $0$ & $0$ & $1$ & $2s_2+1$ & $s_2^2$  \\\hline
  $0+1$ & $0$ & $0$ & $0$ & $0$ & $0$ & $0$ & $1$ & $s_2$  \\
    \hline
\end{tabular}}

\end{ex}

\begin{lem}\label{L3.19}
\mbox{ }
\begin{enumerate}
  \item[(a)] In algebra of $\mathbb{Z}_2$-relations, let $d^{p_1, p_2}_{i, \alpha}, d^{r_1, r_2}_{j, \beta} \in J^{2k}_{2s_1+s_2}$ with $2p_1+p_2 < 2r_1 + r_2$ then $d^{p_1, p_2}_{i, \alpha}$ is coarser than $d^{r_1, r_2}_{j, \beta}$ if and only if $l\left( d^{p_1, p_2}_{i, \alpha} . d^{r_1, r_2}_{j, \beta} \right) = {2p_1+p_2}$ where $J^{2k}_{2s_1+s_2}$ is as in Notation \ref{N3.6}(a).
  \item[(b)] In signed partition algebras, let $d^{p_1, p_2}_{i, \alpha}, d^{r_1, r_2}_{j, \beta} \in \overrightarrow{J}^{2k}_{2s_1+s_2}$ with $2p_1+p_2 < 2r_1 + r_2$ then $d^{p_1, p_2}_{i, \alpha}$ is coarser than $d^{r_1, r_2}_{j, \beta}$ if and only if $l\left( d^{p_1, p_2}_{i, \alpha} . d^{r_1, r_2}_{j, \beta} \right) = {2p_1+p_2}$  where $\overrightarrow{J}^{2k}_{2s_1+s_2}$ is as in Notation \ref{N3.6}(b).
  \item[(c)] In partition algebras, let $ R^{d_{i, \alpha}^{p}}, R^{d_{j, \beta}^{r}} \in J^{k}_{s}$ with $p < r$ then $R^{d_{i, \alpha}^{p}}$ is coarser than $R^{d_{j, \beta}^{r}}$ if and only if $l\left( R^{d_{i, \alpha}^{p}} . R^{d_{j, \beta}^{r}} \right) = {p}$ where $J^{k}_{s}$ is as in Notation \ref{N3.6}(c).
\end{enumerate}
\end{lem}

\begin{proof}
\NI \textbf{Proof of (a):} $d_{i, \alpha}^{p_1, p_2}$ is coarser than $d_{j, \beta}^{r_1, r_2}$ if and only if every $\{e\}$- through class of $d_{j, \beta}^{r_1, r_2}$ is contained in a $\{e\}$-through class of $d_{i, \alpha}^{p_1, p_2}$,  every $\mathbb{Z}_2$-through class  of $d_{j, \beta}^{r_1, r_2}$ is contained in a $\mathbb{Z}_2$-through class of $d^{p_1, p_2}_{i, \alpha},$ every $\{e\}$-horizontal edge of $d_{j, \beta}^{r_1, r_2}$ is contained in either a $\{e\}$ or $\mathbb{Z}_2$-horizontal edge or $\{e\}$-through class of $d_{i, \alpha}^{p_1, p_2}$ and every $\mathbb{Z}_2$-horizontal edge of $d_{j, \beta}^{r_1, r_2}$ is contained in a $\mathbb{Z}_2$-horizontal edge or $\mathbb{Z}_2$-through class of $d_{i, \alpha}^{p_1, p_2}.$

Thus, the number of loops in the product $d_{i, \alpha}^{p_1, p_2}. d_{j, \beta}^{r_1, r_2}$ is $2p_1 + p_2.$

\NI Proof of (b) and (c) are similar to  proof of (a).
\end{proof}

\begin{lem} \label{L3.20}\cite{X}
Given any two diagrams $d^{r_1, r_2}_{i, \alpha}$ and $d^{r'_1, r'_2}_{j, \beta}$ such that $\sharp^p(d^{r_1, r_2}_{i, \alpha} . d^{r'_1, r'_2}_{j, \beta}) = 2s_1 + s_2$ then there exists a unique diagram which is the smallest diagram $d^{r''_1, r''_2}_{l, \gamma}$ among the diagrams coarser than both $d^{r_1, r_2}_{i, \alpha}$ and $d^{r'_1, r'_2}_{j, \beta}.$

Also, $l(d^{r''_1, r''_2}_{l, \gamma}. d^{r''_1, r''_2}_{l, \gamma}) = l(d^{r''_1, r''_2}_{l, \gamma} . d^{r_1, r_2}_{i, \alpha}) = l(d^{r''_1, r''_2}_{l, \gamma} . d^{r'_1, r'_2}_{j, \beta}).$
\end{lem}

\begin{proof}
The proof follows from Definition \ref{D2.13} and \cite{X}.
\end{proof}

\subsection{\textbf{Column Operations on the Gram Matrices of the Algebra of $\mathbb{Z}_2$-Relations, Signed Partition Algebras and Partition Algebras }}
\textbf{\\}

We now perform the column operations inductively on the Gram matrices of the algebra of $\mathbb{Z}_2$-relations,  signed partition algebras and partition algebras as follows:

Let $d^{0, 0}_{i, \alpha}$ be coarser than $d^{0, 1}_{j, \beta}.$ Then by Lemma \ref{L3.20},

\centerline{$l(d^{0, 0}_{i, \alpha} . d^{0, 0}_{i, \alpha}) = l(d^{0, 0}_{i, \alpha} . d^{0, 1}_{j, \beta}) = 0.$}

We apply the column operation: $L_{(j, \beta, 0, 1)} \rightarrow L_{(j, \beta, 0, 1)} - L_{(i, \alpha, 0, 0)}$ then the $((i, \alpha, 0, 0), (j, \beta, 0, 1))$-entry becomes

\centerline{$a_{(i, \alpha, 0, 0), (j, \beta, 0, 1)} - a_{(i, \alpha, 0, 0), (i, \alpha, 0, 0)} = 1 - 1 = 0.$}

Similarly, apply the column operations $L_{(j, \beta, r'_1, r'_2)} \rightarrow L_{(j, \beta, r'_1, r'_2)} - L_{(i, \alpha, r_1, r_2)}$ whenever $d^{r_1, r_2}_{i, \alpha}$ is coarser than $d^{r'_1, r'_2}_{j, \beta}.$

Then $b_{(i, \alpha, r_1, r_2), (j, \beta, r'_1, r'_2)}$ denotes the $((i, \alpha, r_1, r_2), (j, \beta, r'_1, r'_2))$-entry after all the column operations are carried out.
\begin{equation}\label{e3.2a}
  i.e.,  b_{(i, \alpha, r_1, r_2), (j, \beta, r'_1, r'_2)} = a_{(i, \alpha, r_1, r_2), (j, \beta, r'_1, r'_2)} - \underset{\substack{d^{r'''_1, r'''_2}_{l, \gamma} > d^{r_1, r_2}_{i, \alpha} \\ d^{r'''_1, r'''_2}_{l, \gamma} > d^{r'_1, r'_2}_{j, \beta}}}{\sum} b_{(i, \alpha, r_1, r_2), (l, \gamma, r'''_1, r'''_2)} - \underset{\substack{ d^{r''''_1, r''''_2}_{k', \delta'} > d^{r'_1, r'_2}_{j, \beta} \\ d^{r''''_1, r''''_2}_{k', \delta'} \not > d^{r_1, r_2}_{i, \alpha}}}{\sum} b_{(i, \alpha, r_1, r_2), (k', \delta', r''''_1, r''''_2)}
\end{equation}

\begin{lem}\label{L3.21}
\begin{enumerate}
  \item[(a)] In the algebra of $\mathbb{Z}_2$-relations and signed partition algebras, let $(i, \alpha, r_1, r_2) < (j, \beta, r'_1, r'_2).$
\begin{enumerate}
\item[(i)] If $d^{r_1, r_2}_{i, \alpha}$ is coarser than $d^{r'_1, r'_2}_{j, \beta}$  then

    \centerline{$b_{(j, \beta, r'_1, r'_2), (i, \alpha, r_1, r_2)} = b_{(i, \alpha, r_1, r_2), (i, \alpha, r_1, r_2)}.$}
  \item[(ii)] If $d^{r_1, r_2}_{i, \alpha}$ is not coarser than $d^{r'_1, r'_2}_{j, \beta}$ and $l(d^{r_1, r_2}_{i, \alpha} . d^{r'_1, r'_2}_{j, \beta}) \geq 0$ then

      \centerline{$b_{(i, \alpha, r_1, r_2), (j, \beta, r'_1, r'_2)} = 0$ and $ b_{(j, \beta, r'_1, r'_2), (i, \alpha, r_1, r_2)} = 0.$}
  \item[(iii)] If $d^{r_1, r_2}_{i, \alpha}$ is coarser than $d^{r'_1, r'_2}_{j, \beta}$ then

  \centerline{$ b_{(i, \alpha, r_1, r_2), (j, \beta, r'_1, r'_2)} = 0 $}

  \NI where $b_{(i, \alpha, r_1, r_2), (j, \beta, r'_1, r'_2)}$ is the $\left( (i, \alpha, r_1, r_2), (j, \beta, r'_1, r'_2)\right)$-th entry after all the column operations are carried out.
\end{enumerate}
\item[(b)] In  partition algebras, let $(i, \alpha, r) < (j, \beta, r').$
\begin{enumerate}
\item[(i)] If $R^{d^{r}_{i, \alpha}}$ is coarser than $R^{d^{r'}_{j, \beta}}$  then

    \centerline{$b_{(j, \beta, r'), (i, \alpha, r)} = b_{(i, \alpha, r), (i, \alpha, r)}.$}
  \item[(ii)] If $R^{d^{r}_{i, \alpha}}$ is not coarser than $R^{d^{r'}_{j, \beta}}$ and $l(R^{d^{r}_{i, \alpha}} . R^{d^{r'}_{j, \beta}}) \geq 0$ then

      \centerline{$b_{(i, \alpha, r), (j, \beta, r')} = 0$ and $ b_{(j, \beta, r'), (i, \alpha, r)} = 0.$}
  \item[(iii)] If $R^{d^{r}_{i, \alpha}}$ is coarser than $R^{d^{r'}_{j, \beta}}$ then

  \centerline{$ b_{(i, \alpha, r), (j, \beta, r')} = 0 $}

  \NI where $b_{(i, \alpha, r), (j, \beta, r')}$ is the $\left( (i, \alpha, r), (j, \beta, r')\right)$-th entry after all the column operations are carried out.
\end{enumerate}
\end{enumerate}

\end{lem}

\begin{proof}
\NI \textbf{proof of a(i):} It follows from equation (\ref{e3.2a} ), for
\begin{eqnarray*}
  b_{(j, \beta, r'_1, r'_2), (i, \alpha, r_1, r_2)} &=& a_{(j, \beta, r'_1, r'_2), (i, \alpha, r_1, r_2)} - \underset{d^{r''_1, r''_2}_{l, \gamma} > d_{i, \alpha}^{r_1, r_2} > d^{r'_1, r'_2}_{j, \beta}} {\sum} b_{(j, \beta, r'_1, r'_2), (l, \gamma, r''_1, r''_2)} \\
   &=& a_{(i, \alpha, r_1, r_2), (i, \alpha, r_1, r_2)} - \underset{d^{r''_1, r''_2}_{l, \gamma} > d_{i, \alpha}^{r_1, r_2} } {\sum} b_{(l, \gamma, r''_1, r''_2), (l, \gamma, r''_1, r''_2)} \ \ \ (\text{ By Lemma }\ref{L3.19} \text{ and induction})\\
   &=&  b_{(i, \alpha, r_1, r_2), (i, \alpha, r_1, r_2)}
\end{eqnarray*}

\NI  We prove the result by induction on $(i, \alpha, r_1, r_2).$

Let $d_{i, \alpha}^{0, 0}$ be coarser than $d^{r'_1, r'_2}_{j, \beta},$ by lemma \ref{L3.19} we have,

\begin{equation}\label{e3.3}
 l(d^{0, 0}_{i, \alpha} . d^{0, 0}_{i, \alpha}) = l(d^{0, 0}_{i, \alpha} . d^{r'_1, r'_2}_{j, \beta}) = 0
\end{equation}

for any diagram $d^{r''_1, r''_2}_{l, \gamma}$ which is coarser than $d^{r'_1, r'_2}_{j, \beta}$ and $d^{r''_1, r''_2}_{l, \gamma}$ but not coarser than $d^{0, 1}_{i, \alpha}, $  we have

\centerline{$b_{(i, \alpha, 0, 1), (l, \gamma, r'_1, r'_2)} = 0.$}

Thus, by applying the column operations $L_{(j, \beta, r'_1, r'_2)} \rightarrow L_{(j, \beta, r'_1, r'_2)} - L_{((i, \alpha, 0, 1)}$ and equation (\ref{e3.2a}) $((i, \alpha, 0, 0), (j, \beta, r'_1, r'_2))$-entry becomes

$b_{(i, \alpha, 0, 0), (j, \beta, r'_1, r'_2)} = a_{(i, \alpha, 0, 0), (j, \beta, r'_1, r'_2)} - a_{(i, \alpha, 0,0), (i, \alpha, 0, 0)} = 1- 1 = 0 \ \ (\text{by equation} (\ref{e3.3}))$

(ii) Suppose $d^{0, 1}_{i, \alpha}$ and  $d^{0, 1}_{j, \beta}$ such that $\sharp^p(d^{0, 1}_{i, \alpha} . d^{0, 1}_{j, \beta}) = 2s_1+s_2$ then by Lemma \ref{L3.19} $l(d^{0, 1}_{i, \alpha} . d^{0, 1}_{r_1, r_2}) = 0$ then there exists a unique diagram $d^{0, 0}_{k, \delta}$ coarser than both $d^{0, 1}_{i, \alpha}$ and $d^{0, 1}_{j, \beta}$ such that

\centerline{$l(d^{0, 0}_{k, \delta} . d^{0, 0}_{k, \delta}) = l(d^{0, 0}_{k, \delta} . d^{0, 1}_{i, \alpha}) = l(d^{0, 0}_{k, \delta} . d^{0, 1}_{j, \beta}) = 0.$}

Thus, when the column operation $L_{(j, \beta, 0, 1)} \rightarrow L_{(j, \beta, 0, 1)} - L_{(k, \delta, 0, 0)}$ is carried out,
\begin{equation}\label{e3.4}
    b_{(i, \alpha, 0, 1), (j, \beta, 0, 1)} = a_{(i, \alpha, 0, 1), (j, \beta, 0, 1)}  - a_{(i, \alpha, 0, 1), (k, \delta, 0, 0)} = 1- 1 = 0.
\end{equation}

\textbf{Proof of a(ii):} In general, Let   $d_{i,\alpha}^{r_1, r_2}$ be not coarser than $d^{r'_1, r'_2}_{j, \beta}$ such that $l(d^{r_1, r_2}_{i, \alpha} . d^{r'_1, r'_2}_{j, \beta}) \geq 0.$

Then by Lemma \ref{L3.20} there is a unique diagram $d^{r''_1, r''_2}_{k, \delta}$ coarser than both $d^{r_1, r_2}_{i, \alpha}$ and $d^{r'_1, r'_2}_{j, \beta}$ such that

\centerline{$l(d^{r''_1, r''_2}_{k, \delta} . d^{r''_1, r''_2}_{k, \delta}) = l(d^{r''_1, r''_2}_{k, \delta} . d^{r_1, r_2}_{i, \alpha}) = l(d^{r''_1, r''_2}_{k, \delta} . d^{r'_1, r'_2}_{j, \beta})$}

When the column operations are carried out inductively,

$b_{(i, \alpha, r_1, r_2), (j, \beta, r'_1, r'_2)} = a_{(i, \alpha, r_1, r_2), (j, \beta, r'_1, r'_2)} - \underset{\substack{d^{r'''_1, r'''_2}_{l, \gamma} > d^{r_1, r_2}_{i, \alpha} \\ d^{r'''_1, r'''_2}_{l, \gamma} > d^{r'_1, r'_2}_{j, \beta}}}{\sum} b_{(i, \alpha, r_1, r_2), (l, \gamma, r'''_1, r'''_2)} - \underset{\substack{ d^{r''''_1, r''''_2}_{k', \delta'} > d^{r'_1, r'_2}_{j, \beta} \\ d^{r''''_1, r''''_2}_{k', \delta'} \not > d^{r_1, r_2}_{i, \alpha}}}{\sum} b_{(i, \alpha, r_1, r_2), (k', \delta', r''''_1, r''''_2)}$

By induction hypothesis, each entry in the second summation becomes zero.

Thus, we have

$b_{(i, \alpha, r_1, r_2), (j, \beta, r'_1, r'_2)} = a_{(i, \alpha, r_1, r_2), (j, \beta, r'_1, r'_2)} - \underset{\substack{d^{r''_1, r''_2}_{l, \gamma} > d^{r_1, r_2}_{i, \alpha} \\ d^{r''_1, r''_2}_{l, \gamma} > d^{r'_1, r'_2}_{j, \beta}}}{\sum} b_{(i, \alpha, r_1, r_2), (l, \gamma, r''_1, r''_2)}.$

Also, by induction,
\begin{equation}\label{e3.5}
b_{(i, \alpha, r_1, r_2), (i', \alpha', r''''_1, r''''_2)} = b_{(i', \alpha', r''''_1, r''''_2), (i', \alpha', r''''_1, r''''_2)}.
\end{equation}

Thus,
\begin{eqnarray*}
  b_{(i, \alpha, r_1, r_2), (j, \beta, r'_1, r'_2)} &=& ( a_{(i, \alpha, r_1, r_2), (j, \beta, r'_1, r'_2)} - \underset{\substack{d^{r'''_1, r'''_2}_{l, \gamma} > d^{r_1, r_2}_{i, \alpha} \\ d^{r'''_1, r'''_2}_{l, \gamma} > d^{r'_1, r'_2}_{j, \beta} \\ d^{r'''_1, r'''_2}_{l, \gamma} \neq d^{r''_1, r''_2}_{k, \delta}}}{\sum} b_{(i, \alpha, r_1, r_2), (l, \gamma, r'''_1, r'''_2)}) - b_{(k, \delta, r''_1, r''_2), (k, \delta, r''_1,r''_2)} \\
  &=& b_{(k, \delta, r''_1, r''_2), (k, \delta, r''_1,r''_2)} - b_{(k, \delta, r''_1, r''_2), (k, \delta, r''_1,r''_2)} \\
   &=& b_{(k, \delta, r''_1, r''_2), (k, \delta, r''_1,r''_2)} - b_{(k, \delta, r''_1, r''_2), (k, \delta, r''_1,r''_2)} \ \ \ (\text{by equation } (\ref{e3.5}))\\
   &=& 0
\end{eqnarray*}

Thus, $\left((i, \alpha,, r_1, r_2), (j, \beta, r'_1, r'_2) \right)$-entry becomes zero after applying the column operations when $d^{r_1, r_2}_{i, \alpha}$ is not coarser than $d^{r'_1, r'_2}_{j, \beta}$ such that $l(d^{r_1, r_2}_{i, \alpha} . d^{r'_1, r'_2}_{j, \beta}) \geq 0. $

Also, $b_{(j, \beta, r'_1, r'_2), (i, \alpha, r_1, r_2)} = a_{(j, \beta, r'_1, r'_2), (i, \alpha, r_1, r_2)} - \underset{\substack{d^{r''_1, r''_2}_{l, \gamma} > d^{r_1, r_2}_{i, \alpha} \\ d^{r''_1, r''_2}_{l, \gamma} > d^{r'_1, r'_2}_{j, \beta}}}{\sum} b_{(j, \beta, r'_1, r'_2), (l, \gamma, r''_1, r''_2)}.$

\NI since $b_{(j, \beta, r'_1, r'_2), (k,\delta, r'''_1, r'''_2)}$ becomes zero by induction for all $d_{k, \delta}^ {r'''_1, r'''_2}$ coarser than $d^{r_1, r_2}_{i, \alpha}$ and not coarser than $d^{r'_1, r'_2}_{j, \beta}$  arguing as in proof of (ii),

\centerline{$b_{(j, \beta, r'_1, r'_2), (i, \alpha, r_1, r_2)} = 0.$}

\textbf{Proof of a(iii):} In general, let $d^{r_1, r_2}_{i, \alpha}$ be coarser than $d^{r'_1, r'_2}_{j, \beta}$, by Lemma \ref{L3.19}

\centerline{$ l(d^{r_1, r_2}_{i, \alpha} . d^{r_1, r_2}_{i, \alpha}) = l(d^{r_1, r_2}_{i, \alpha} . d^{r'_1, r'_2}_{j, \beta}) = 2r_1+r_2.$}

By induction hypothesis,

\begin{equation}\label{e3.6}
b_{(i,\alpha, r_1, r_2), (j, \beta, r'_1, r'_2)} = a_{(i, \alpha, r_1, r_2), (j, \beta, r'_1, r'_2)} - \underset{\substack{d^{r''_1, r''_2}_{l, \gamma} > d^{i, \alpha}_{r_1, r_2} \\ d^{r''_1, r''_2}_{l, \gamma} > d^{r'_1, r'_2}_{j, \beta}}}{\sum} b_{(i, \alpha, r_1, r_2), (l, \gamma, r''_1, r''_2)} \text{ and }
\end{equation}
\begin{equation}\label{e3.7}
b_{(i, \alpha, r_1, r_2), (i, \alpha, r_1, r_2)} = a_{(i, \alpha, r_1, r_2), (i, \alpha, r_1, r_2)} - \underset{d^{r''_1, r''_2}_{l, \gamma} > d^{r_1, r_2}_{i, \alpha}}{\sum} b_{(i, \alpha, r_1, r_2), (l, \gamma, r''_1, r''_2)}
\end{equation}

Thus, when the column operation $L_{(j, \beta, r'_1, r'_2)} \rightarrow L_{(j, \beta, r'_1, r'_2)} - L_{(i, \alpha, r_1, r_2)}$ is carried out the

\NI $((i, \alpha, r_1, r_2), (j, \beta, r'_1, r'_2))$-th entry of the block matrix $A_{2r_1+r_2, 2r'_1+r'_2}$ becomes zero.

i.e., $b_{(i, \alpha, r_1, r_2), (j, \beta, r'_1, r'_2)} = 0.$

\NI Proof of (b) is similar to proof of (a).
\end{proof}

\begin{thm}\label{T3.22}
\mbox{ }
\begin{enumerate}
  \item[(a)] After applying the column operations the diagonal entry $x^{2r_1+r_2}$ in the block matrix $A_{2r_1+r_2, 2r_1+r_2}$ for $ 0 \leq r_1+r_2 \leq k-s_1-s_2$ and the block matrix $ \overrightarrow{A}_{2r_1+r_2, 2r_1+r_2}$ for $0 \leq r_1+r_2 \leq k-s_1-s_2-1$ of the algebra of $\mathbb{Z}_2$-relations and signed partition algebras respectively are replaced by

 \begin{enumerate}
  \item[(i)] $ \underset{j=0}{\overset{r_1-1}{\prod}}[x^2-x-2(s_1+j)] \underset{l=0}{\overset{r_2-1}{\prod}} [x-(s_2+l)] $ \ \ \ \ \ \ if $r_1 \geq 1 $ and $r_2 \geq 1,$
  \item[(ii)] $ \underset{j=0}{\overset{r_2-1}{\prod}} [x-(s_2+j)]$ \ \ \ \ \ \ if $ r_1 = 0$ and $r_2 \neq 0,$
  \item[(iii)] $ \underset{j=0}{\overset{r_1-1}{\prod}}[x^2-x-2(s_1+j)]$ \ \ \ \ \ \ if $r_1 \neq 0$ and $r_2 = 0.$
\end{enumerate}
\NI Also, the diagonal elements in the block matrix $A_{2r_1+r_2, 2r_1+r_2}$ and $\overrightarrow{A}_{2r_1 +r_2, 2r_1+r_2}$ are the same.

  \item[(b)] After applying the column operations the diagonal entry $x^{r}$ in the block matrix $A_{r, r}$ for $0 \leq r \leq k $ is replaced by

      \centerline{$\underset{j=0}{\overset{r-1}{\prod}} [x-(s+j)]$ if $r \geq 1$ and $1$ if $r = 0.$}

 \NI Also, the diagonal elements in the block matrix $A_{r, r} $ are the same.

\end{enumerate}

\end{thm}
\begin{proof}
\NI \textbf{Proof of (a)(i):} The proof is by induction on the number of horizontal edges.

    Let  $d^{r_1, r_2}_{i, \alpha}$  be any diagram corresponding to the diagonal entry $x^{2r_1+r_2}$ in block matrix $A_{2r_1+r_2, 2r_1+r_2}$  having $2s_1 + s_2$ number of through classes and $r_1$ pairs of $\{e\}$-horizontal edges and $r_2$ number of $\mathbb{Z}_2$-horizontal edges.

   After applying column operations as mentioned earlier to eliminate the entries which lie above  corresponding to the  diagrams coarser than  $d^{r_1, r_2}_{i, \alpha}$, then by Lemma \ref{L3.19} and induction the diagonal entry $x^{2r_1+r_2}$ is replaced as
\begin{equation}\label{e3.8}
     x^{2r_1 + r_2} -\underset{\substack{\ds 0 \leq j \leq r_1 \\ \ds -r_2 \leq j' \leq r_1  \\ \ds -2j + j' < 0}}{\sum} B^{s_1, s_2}_{2r_1 + r_2, 2[r_1 - j] + r_2  + j'} \underset{l=0}{\overset{r_1-j-1}{\prod}} [x^2-x-2(s_1+l)]   \underset{f=0}{\overset{r_2 +j'-1}{\prod}} [x-(s_2+f)]
      \end{equation}

\NI where $B^{s_1, s_2}_{2r_1+r_2, 2p_1+p_2}$ gives the number of diagrams which  has $p_1$ pairs of $\{e\}$ horizontal edges and $p_2$ number of $\mathbb{Z}_2$ horizontal edges which lie above and coarser than $d^{r_1, r_2}_{i, \alpha}.$

 Fix $s$  and put  \begin{equation}\label{e3.9}
H_{2r_1 + r_2, s} = - \ \sum_{\substack{\ds 0 \leq j \leq r_1 \\ \ds -r_2 \leq j' \leq r_1 \\ \ds -2j + j' < 0 \text{ and }m- 2j + j' \geq 0}} \left( -1 \right)^{2j- j'} B^{s_1, s_2}_{2r_1 + r_2, 2[r_1 - j] + r_2 + j'} C_{2[r_1-j] + r_2 + j',s}
 \end{equation}
  where $C_{2r'_1 + r'_2, s}$ denote the coefficient of $x^s$ in $\underset{j=0}{\overset{r'_1-1}{\prod}}[x^2-x-2(s_1 +j)] \underset{l=0}{\overset{r'_2-1}{\prod}}[x-(s_2-l)]$ where $m = 2r_1+r_2-s.$

We shall claim that,

 \centerline{$H_{2r_1 + r_2, s} =(-1)^m C_{2r_1 + r_2 - 1, s} .$}

We shall prove this by using induction on $2r_1 + r_2.$
\begin{eqnarray*}
   H_{2r_1 + r_2, s} &=& - \ \ds \sum_{\substack{0 \leq j \leq r_1 \\ -r_2 \leq j' \leq r_1 \\ -2j + j' < 0 \text{ and } m-2j +j' \geq 0}} \left( -1 \right)^{2j - j'} B^{s_1, s_2}_{2r_1 + r_2 , 2[r_1 - j] + r_2 + j'} C_{2[r_1-j] + r_2 + j',s}\\
 \end{eqnarray*}
 where $m = 2r_1 + r_2 - s.$

 By using Lemma \ref{L3.16} and induction hypothesis,  equation (\ref{e3.9}) becomes,

 $ H_{2r_1+r_2, s} = $
 \begin{eqnarray*}
    & & - \ds \sum_{\substack{\ds 0 \leq j \leq r_1 \\ -r_2 \leq j' \leq r_1 \\ -2j + j' < 0 \text{ and } m-2j + j' \geq 0}} \left( -1 \right)^{2j - j'} \left\{ B^{s_1, s_2}_{2r_1 + r_2-1, 2[r_1 -j] + r_2 +  j'-1}  + (s_2 +r_2 +j') \ B^{s_1, s_2}_{2r_1+r_2-1, 2[r_1-j]+r_2+j' } \right\} \\
     & & \hspace{6.5cm} \ \left\{ C_{2[r_1-j] + r_2 + j'-1,s-1} + (s_2+r_2+j'-1) \  C_{2[r_1-j] + r_2 +j'-1,s} \right\}
 \end{eqnarray*}
The equation (\ref{e3.9}) can be rewritten as follows:

$ H_{2r_1+r_2, s} = $
\begin{eqnarray*}
   & & - \ \ds \sum_{\substack{\ds 0 \leq j \leq r_1 \\ \ds -r_2 \leq j' \leq r_1 \\ \ds  -2j + j' < 0 \text{ and } m-2j+j' \geq 0}} \ds \left( -1 \right)^{2j - j'} B^{s_1, s_2}_{2r_1 + r_2-1, 2[r_1 - j] + r_2 +  j'-1}    C_{2[r_1-j] + r_2 + j'-1,s-1} \\
    & & -  \sum_{\substack{\ds 0 \leq j \leq r_1 \\ \ds -r_2 \leq j' \leq r_1 \\ \ds  -2j + j' < 0 \text{ and } m-2j+j' \geq 0}} \ds (-1)^{2j -j'}(s_2 +r_2 +j'-1) B^{s_1, s_2}_{2r_1+r_2-1, 2[r_1-j]+r_2+j'-1 }  C_{2[r_1-j] + r_2 + j'-1, s}
      \end{eqnarray*}
      \begin{eqnarray*}
      & & - \ds \sum_{\substack{\ds 0 \leq j \leq r_1 \\ \ds -r_2 \leq j' \leq r_1 \\ \ds  -2j + j' < 0 \text{ and } m-2j+ j' \geq 0}} \ds (-1)^{2j-j'} (s_2+r_2+j') B^{s_1, s_2}_{2r_1 + r_2-1, 2[r_1 - j] + r_2 +  j'}  C_{2[r_1-j] + r_2 + j', s}\\
         & = & H_{2r_1+r_2-1, s-1} + (-1)^m (s_2+r_2-1) C_{2r_1+r_2-1, s}  \ \ \ \text{ (by canceling common terms)}\\
    & = & (-1)^m C_{2r_1+r_2-1, s-1} + (-1)^m (s_2+r_2-1) C_{2r_1+r_2-1, s}\ \ \ \text{ (by induction)}
\end{eqnarray*}

\NI Thus, equation (\ref{e3.9}) reduces to

 \centerline{$H_{2r_1 + r_2, s} = (-1)^m \ C_{2r_1 + r_2 -1, s-1} + (-1)^m \ (s_2 + r_2-1) \ C_{2r_1 + r_2-1, s} = (-1)^m C_{2r_1+ r_2, s}$}

\NI where $C_{2r_1 + r_2, s} =  C_{2r_1 + r_2 -1, s-1} + (s_2 + r_2-1) C_{2r_1 + r_2-1, s}.$

\NI The same proof works for the diagonal element in the block matrix $\overrightarrow{A}_{2r_1+r_2, 2r_1+r_2}$ for $0 \leq r_1+r_2 \leq k-s_1-s_2-1$ in signed partition algebras.

\textbf{Proof of (a)(iii):}  (a)(iii) can be proved in similar fashion as that of  (a)(i) by using Lemma \ref{L3.17} and $C_{2r_1, s} = (-1)^m \ C_{2(r_1-1), s-2} + (-1)^m \ C_{2(r_1-1), s-1}- (-1)^m \ 2 \ (s_1+r_1-1) \ C_{2(r_1-1), s}. $

 \NI Proof of (b) is same as that of proof of (a).
\end{proof}

\begin{lem}\label{L3.23}
 Let $d^{r_1, r_2}_{i, \alpha}, d^{r'_1, r'_2}_{j, \beta} \in J^{2k}_{2s_1+s_2}$ and $d^{r_1, r_2}_{i, \alpha}, d^{r'_1, r'_2}_{j, \beta} \in \overrightarrow{J}^{2k}_{2s_1+s_2}.$ The $((i, \alpha, r_1, r_2), (j,\beta, r'_1, r'_2))$-entry of the Gram matrices $G_{2s_1+s_2}^k$ of the algebra of $\mathbb{Z}_2$-relations and  $ \overrightarrow{G}_{2s_1+s_2}^k$ of the signed partition algebras remains zero even after applying  column operations inductively if the $\mathbb{Z}_2$-horizontal edge of $d^{r_1, r_2}_{i, \alpha}$ coincides with the $\{e\}$-through class of $d^{r'_1, r'_2}_{j, \beta}$ and vice versa.
\end{lem}
\begin{proof}
The proof follows from Definition \ref{D3.7} and there is no diagram in common which is  coarser than both $d^{r_1, r_2}_{i, \alpha}, d^{r'_1, r'_2}_{j, \beta}  \in \mathbb{J}^{2r_1+r_2}_{2s_1+s_2}.$
 \end{proof}

\begin{rem} \label{R3.24}
\begin{enumerate}
  \item[(a)] Let $d^{r_1, r_2}_{i, \alpha}, d^{r'_1, r'_2}_{j, \beta} \in J^{2k}_{2s_1+s_2}$ such that $\sharp^p(d^{r_1, r_2}_{i, \alpha} . d^{r'_1, r'_2}_{j, \beta}) < 2s_1+s_2.$ Place $d^{r_1, r_2}_{i, \alpha}$ above $d^{r'_1, r'_2}_{j, \beta}.$ Choose sub diagrams $d^{r_1-t'_1, r_2-t'_2} \in J^{2f}_{2(s_1-t_1)+s_2-t_2}$ of $d^{r_1, r_2}_{i, \alpha}$ and $d^{r'_1-t''_1, r'_2-t''_2} \in J^{2f}_{2(s_1-t_1)+s_2-t_2}$ of $d^{r'_1, r'_2}_{j, \beta}$ such that $l(d^{r_1 - t'_1, r_2-t'_2} . d^{r'_1-t''_1, r'_2-t''_2}) \geq 0$  with

      \NI $\sharp^p\left( \left(d^{r_1, r_2}_{i, \alpha} \setminus d^{r_1-t'_1, r_2-t'_2} \right) . \left(d^{r'_1, r'_2}_{j, \beta} \setminus d^{r'_1-t''_1, r'_2-t''_2} \right) \right) < 2t_1+t_2.$

\NI For the sake of convenience, we shall write $d^{r_1, r_2}_{i, \alpha} = d^{r_1-t'_1, r_2-t'_2} \otimes d^{l_1-f}_{l_1-f} , d^{r'_1, r'_2}_{j, \beta} = d^{r_1-t''_1, r_2-t''_2} \otimes d^{l_2-f}_{l_2-f}$ where $d^{l_1-f}_{l_1-f} = d^{r_1, r_2}_{i, \alpha} \setminus d^{r'_1-t''_1, r'_2-t''_2} $ and $ d^{l_2-f}_{l_2-f} = d^{r'_1, r'_2}_{j, \beta} \setminus d^{r'_1-t''_1, r'_2-t''_2}.$
  \item[(b)] Let $R^{d^{r}_{i, \alpha}}, R^{d^{r'}_{j, \beta}} \in J^{k}_{s}$ such that $\sharp^p(R^{d^{r}_{i, \alpha}} . R^{d^{r'}_{j, \beta}}) < 2.$ Place $R^{d^{r}_{i, \alpha}}$ above $R^{d^{r'}_{j, \beta}}.$ Choose sub diagrams $R^{d^{r-t'}} \in J^{f}_{s-t}$ of $R^{d^{r}_{i, \alpha}}$ and $R^{d^{r'-t''}} \in J^{f}_{s-t}$ of $R^{d^{r'}_{j, \beta}}$ such that $l(R^{d^{r-t'}} . R^{d^{r'-t''}}) \geq 0$  with $\sharp^p\left( \left(R^{d^{r}_{i, \alpha}} \setminus R^{d^{r-t'}} \right) . \left(R^{d^{r'}_{j, \beta}} \setminus R^{d^{r'-t''}} \right) \right) < t.$

\NI For the sake of convenience, we shall write $R^{d^{r}_{i, \alpha}} = R^{d^{r-t'}} \otimes d^{l_1-f}_{l_1-f} , R^{d^{r'}_{j, \beta}} = R^{d^{r-t''}} \otimes d^{l_2-f}_{l_2-f}$ where $d^{l_1-f}_{l_1-f} = R^{d^{r}_{i, \alpha}} \setminus R^{d^{r'-t''}} $ and $ d^{l_2-f}_{l_2-f} = R^{d^{r'}_{j, \beta}} \setminus R^{d^{r'-t''}}.$
\end{enumerate}
\end{rem}

\begin{notation} \label{N3.25}
\mbox{ }
\begin{enumerate}
  \item[(a)] Let $d^{r_1, r_2}_{i, \alpha}, d^{r_1, r_2}_{j, \beta}$ be as in Remark \ref{R3.24}(a) such that $\sharp^p \left( d^{r_1, r_2}_{i, \alpha}. d^{r_1, r_2}_{j, \beta}\right) < 2s_1 + s_2,$ so that the $((i, \alpha, r_1, r_2), (j, \beta, r_1, r_2))$-entry of the block matrix $A_{2r_1 + r_2, 2r_1 + r_2}$ in algebra of $\mathbb{Z}_2$-relations is zero and $0 \leq r_1 + r_2 \leq k - s_1 - s_2.$

\NI If $t'_1 = t''_1 = t_1, t'_2 = t''_2= t_2, $ put $d^{r_1, r_2}_{i, \alpha} = d^{l_1^{f}}_{l_1^{f}} \otimes d^{l_1-f}_{l_1-f}$ and $d^{r_1, r_2}_{j, \beta} = d^{l_2^{f}}_{l_2^{f}} \otimes d^{l_2-f}_{l_2-f}$ where $d^{l_1^f}_{l_1^f} \left(d^{l_2^{f}}_{l_2^{f}} \right)$ is the sub diagram of $d^{r_1, r_2}_{i, \alpha}\left( d^{r_1, r_2}_{j, \beta} \right), d^{l_1^f}_{l_1^f}, d^{l_2^f}_{l_2^f} \in \mathbb{J}^{2t_1+t_2}_{2t_1+t_2}$ and every $\{e\}$-through class $\left( \mathbb{Z}_2-\text{through class} \right)$ of $d^{l_1^f}_{l_1^f}$ is replaced by a $\{e\}$-horizontal edge $\left( \mathbb{Z}_2-\text{horizontal edge} \right)$ and vice versa.
  \item[(b)] Let $d^{r_1, r_2}_{i, \alpha}, d^{r_1, r_2}_{j, \beta} $ be as in Remark \ref{R3.24}(b) such that $\sharp^p \left( d^{r_1, r_2}_{i, \alpha}. d^{r_1, r_2}_{j, \beta}\right) < 2s_1 + s_2,$ so that the $((i, \alpha, r_1, r_2), (j, \beta, r_1, r_2))$-entry of the block matrix $\overrightarrow{A}_{2r_1 + r_2, 2r_1 + r_2}$ in algebra of $\mathbb{Z}_2$-relations is zero and $0 \leq r_1 + r_2 \leq k - s_1 - s_2 - 1.$

\NI  If $t'_1 = t''_1 = t_1, t'_2 = t''_2= t_2, $ put $d^{r_1, r_2}_{i, \alpha} = d^{l_1^{f}}_{l_1^{f}} \otimes d^{l_1-f}_{l_1-f}$ and $d^{r_1, r_2}_{j, \beta} = d^{l_2^{f}}_{l_2^{f}} \otimes d^{l_2-f}_{l_2-f}$ where $d^{l_1^f}_{l_1^f} \left(d^{l_2^{f}}_{l_2^{f}} \right)$ is the sub diagram of $d^{r_1, r_2}_{i, \alpha}\left( d^{r_1, r_2}_{j, \beta} \right), d^{l_1^f}_{l_1^f}, d^{l_2^f}_{l_2^f} \in \overrightarrow{\mathbb{J}}^{2t_1+t_2}_{2t_1+t_2}$ and every $\{e\}$-through class $\left( \mathbb{Z}_2-\text{through class} \right)$ of $d^{l_1^f}_{l_1^f}$ is replaced by a $\{e\}$-horizontal edge $\left( \mathbb{Z}_2-\text{horizontal edge} \right)$ and vice versa.
 \item[(c)]Let $R^{d^{r}_{i, \alpha}}, R^{d^{r}_{j, \beta}} \in \mathbb{J}^{r}_{s}$ such that $\sharp^p \left( R^{d^{r}_{i, \alpha}}. R^{d^{r}_{j, \beta}}\right) < s,$ so that the $((i, \alpha, r), (j, \beta, r))$-entry of the block matrix $A_{r, r}$ in the partition algebra is zero and $0 \leq r \leq k - s.$

 Put $  R^{d^{r}_{i, \alpha}}= d^{l_1}_{l_1} \otimes d^{l_1-f}_{l_1-f}$ and $ R^{d^{r}_{j, \beta}}= d^{l_2}_{l_2} \otimes d^{l_2-f}_{l_2-f}$ where $d^{l_1}_{l_1} \left(d^{l_2}_{l_2} \right)$ is the sub diagram of $ R^{d^{r}_{i, \alpha}} \left( R^{d^{r}_{j, \beta}} \right), d^{l_1}_{l_1}, d^{l_2}_{l_2} \in \mathbb{J}^{t}_{t}$ and every through class of $d^{l_1}_{l_1}$ is replaced by a  horizontal edge and vice versa.
                    \end{enumerate}
\end{notation}

\begin{ex}\label{E3.26}
This example illustrates Notation \ref{N3.25}.

\begin{center}
\includegraphics[height=2.5cm, width=15cm]{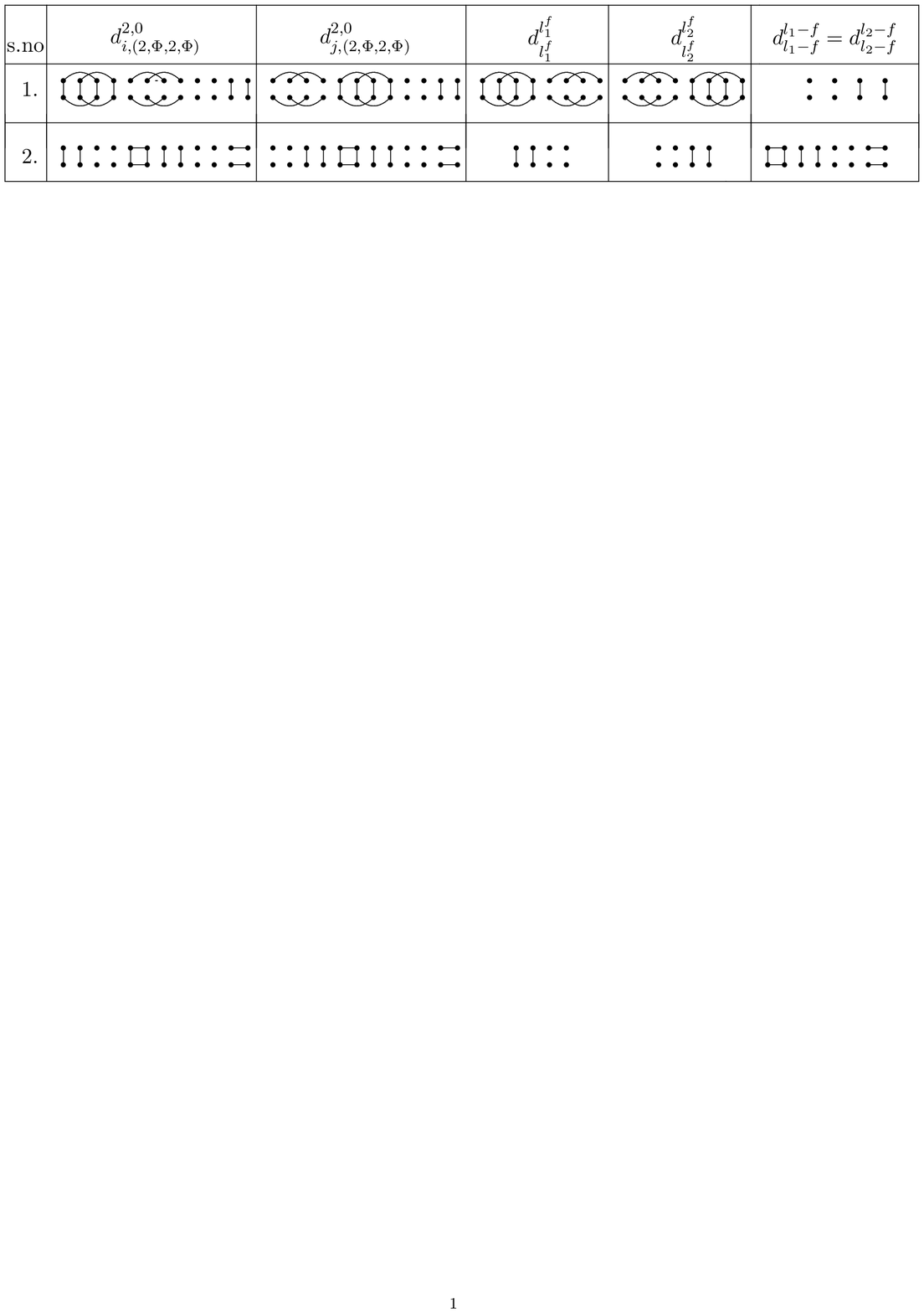}
\end{center}
\end{ex}

\begin{lem}\label{L3.27}
Let $(i, \alpha, r_1, r_2) < (j, \beta, r'_1, r'_2).$
\begin{enumerate}
  \item[(a)] Let $d^{r_1, r_2}_{i, \alpha} , d^{r'_1, r'_2}_{j, \beta}  \in  J^{2k}_{2s_1 + s_2} $ such that

      $\sharp^p \left( d^{r_1, r_2}_{i, \alpha} . d^{r'_1, r'_2}_{j, \beta} \right) < 2s_1 + s_2$ with $d^{r_1, r_2}_{i, \alpha}  = d^{r_1-t'_1, r_2-t'_2} \otimes d^{l_1-f}_{l_1-f}$ and $ d^{r'_1, r'_2}_{j, \beta} = d^{r'_1-t''_1, r'_2-t''_2 } \otimes d^{l_2-f}_{l_2-f}$ where $d^{r_1-t'_1, r_2-t'_2}, d^{r'_1-t''_1, r'_2-t''_2}$ are as in Remark \ref{R3.24}(a).
  \item[(b)]Let $d^{r_1, r_2}_{i, \alpha} , d^{r'_1, r'_2}_{j, \beta}  \in  \overrightarrow{J}^{2k}_{2s_1 + s_2} $ such that $\sharp^p \left( d^{r_1, r_2}_{i, \alpha} . d^{r'_1, r'_2}_{j, \beta} \right) < 2s_1 + s_2$ with $d^{r_1, r_2}_{i, \alpha}  = d^{r_1-t'_1, r_2-t'_2} \otimes d^{l_1-f}_{l_1-f}$ and $ d^{r'_1, r'_2}_{j, \beta} = d^{r'_1-t''_1, r'_2-t''_2 } \otimes d^{l_2-f}_{l_2-f}$ where $d^{r_1-t'_1, r_2-t'_2}, d^{r'_1-t''_1, r'_2-t''_2}$ are as in Remark \ref{R3.24}(a).
      Then

\centerline{$b_{(i, \alpha, r_1, r_2), (j, \beta, r'_1, r'_2)} = 0,$} 
\NI if any one of the following conditions hold:
\begin{enumerate}
  \item[(i)] $2r_1+r_2 < 2r'_1+r'_2$ or
  \item[(ii)] if $2r_1+r_2 = 2r'_1+r'_2$ then $r_1+r_2 < r'_1+r'_2$ or
  \item[(iii)] $t''_1 \neq t_1$ or $t''_2 \neq t_2$ or
  \item[(iv)] $2r_1+r_2-(2t'_1+t'_2) < 2r'_1+r_2 - (2t''_1+t''_2)$
\end{enumerate}

  \item[(c)] Let $R^{d^{r}_{i, \alpha}}, R^{d^{r'}_{j, \beta}} \in \mathbb{J}^{r'}_{s} $ such that $\sharp^p \left( R^{d^{r}_{i, \alpha}}, R^{d^{r'}_{j, \beta}} \right) < s$ with $R^{d^{r}_{i, \alpha}} = d^{r-t'} \otimes R^{d^r_{i, \alpha}} \setminus d^{r-t'}$ and $ R^{d^{r'}_{j, \beta}} = d^{r'-t''} \otimes R^{d^{r'}_{j, \beta}} \setminus d^{r'-t''}$ where $d^{r-t'} \in \mathbb{J}^{t'}_{t}, d^{r'-t''} \in \mathbb{J}^{t''}_{t}, R^{d^r_{i, \alpha}} \setminus d^{r-t'} \in \mathbb{J}^{r - t'}_{s-t}$ and $ R^{d^{r'}_{j, \beta}} \setminus d^{r'-t''} \in \mathbb{J}^{r'- t'' }_{s-t}.$

     Then

\centerline{$b_{(i, \alpha, r), (j, \beta, r')} = 0,$}
\NI if any one of the following conditions hold:
\begin{enumerate}
  \item[(i)] $r' < r$
  \item[(ii)] $t'' \neq t$
  \item[(iii)] $r-t' < r' - t''$
\end{enumerate}

\end{enumerate}
\end{lem}

\begin{proof}
\NI \textbf{Proof of (a):}The proof is by induction on the conditions

\begin{enumerate}
  \item[(i)] $2r_1+r_2 < 2r'_1+r'_2$ or
  \item[(ii)] if $2r_1+r_2 = 2r'_1+r'_2$ then $r_1+r_2 < r'_1+r'_2$ or
  \item[(iii)] $t''_1 \neq t_1$ or $t''_2 \neq t_2$ or
  \item[(iv)] $2r_1+r_2-(2t'_1+t'_2) < 2r'_1+r_2 - (2t''_1+t''_2)$
\end{enumerate}

Since $\sharp^p \left( d^{r_1, r_2}_{i, \alpha} . d^{r'_1, r'_2}_{j, \beta}\right) < 2s_1+s_2$ which implies that $a_{(i, \alpha, r_1, r_2), (j, \beta, r'_1, r'_2)} =0.$

After applying column operations inductively we get,

\begin{equation}\label{e3.10}
 b_{(i, \alpha, r_1, r_2), (j, \beta, r'_1, r'_2)}= - \underset{\substack{\ds d^{r''_1, r''_2}_{l, \gamma} > d^{r_1, r_2}_{i, \alpha}\\ \ds d^{r''_1, r''_2}_{l, \gamma} > d^{r'_1, r'_2}_{j, \beta} }}\sum b_{(l, \gamma, r''_1, r''_2), (l, \gamma, r''_1, r''_2)} - \underset{\substack{\ds d^{r''_1, r''_2}_{l, \gamma} > d^{r'_1, r'_2}_{j, \beta} \\ \ds d^{r''_1, r''_2}_{l, \gamma} \ngtr d^{r_1, r_2}_{i, \alpha}}}\sum b_{(i, \alpha, r_1, r_2), (l, \gamma, r''_1, r''_2)}
\end{equation}

\NI  Suppose that $\sharp^p \left( d^{r''_1, r''_2}_{l, \gamma} . d^{r_1, r_2}_{i, \alpha}\right) = 2s_1+s_2$ then by  Lemma \ref{L3.20} and induction hypothesis,

\centerline{$b_{(i, \alpha, r_1, r_2), (l, \gamma, r''_1, r''_2)} = 0.$}

\NI  Suppose that $\sharp^p \left(d^{r''_1, r''_2}_{l, \gamma} . d^{r_1, r_2}_{i, \alpha}\right) < 2s_1+s_2$ then by using induction on any one of the conditions (i), (ii), (iii) and (iv)

\centerline{$b_{(i, \alpha, r_1, r_2), (l, \gamma, r''_1, r''_2)} = 0$,}

By Lemma \ref{L3.20},  there exists a unique diagram  $d^{l_3-f}_{l_3-f}$  coarser than both $d^{l_2-f}_{l_2-f}$ and $d^{l_1-f}_{l_1-f}$ and  $d^{l_3^f}_{l_3^f} \in \mathbb{J}^{2t_1+t_2}_{2t_1+t_2}$ which is coarser than $d^{r'_1-t''_1, r'_2-t''_2}.$ Denote $d^{l_3^f}_{l_3^f} \otimes d^{l_3-f}_{l_3-f}$ by $d^{r''''_1, r''''_2}_{k, \delta}.$

It is clear that, $d^{r''''_1, r''''_2}_{k, \delta} $ is coarser than $d^{r'_1, r'_2}_{j, \beta}.$

Thus, after applying the column operations $L_{(j, \beta, r'_1, r'_2)} \rightarrow L_{(j, \beta, r'_1, r'_2)} - L_{(k, \delta, r''''_1, r''''_2)}$ we get,

$ b_{(i, \alpha, r_1, r_2), (j, \beta, r'_1, r'_2)}$
\begin{eqnarray*}
  &=& - \underset{\substack{\ds d^{r''_1, r''_2}_{l, \gamma} > d^{r_1, r_2}_{i, \alpha} \\ \ds d^{r''_1, r''_2}_{l, \gamma} > d^{r''''_1, r''''_2}_{k, \delta}}} \sum b_{(l, \gamma, r'''_1, r'''_2), (l, \gamma, r'''_1, r'''_2)} -  \underset{\substack{\ds d^{r''_1, r''_2}_{l, \gamma} > d^{r''''_1, r''''_2}_{k, \delta} \\ \ds d^{r''_1, r''_2}_{l, \gamma} \ngtr d^{r_1, r_2}_{i, \alpha}}} \sum b_{(i, \alpha, r_1, r_2), (l, \gamma, r''_1, r''_2)} - b_{(i, \alpha, r_1, r_2), (k, \delta, r''''_1, r''''_2)} \\
   &=& b_{(i, \alpha, r_1, r_2), (k, \delta, r''''_1, r''''_2)} - b_{(i, \alpha, r_1, r_2), (k, \delta, r''''_1, r''''_2)} \\
   &=& 0.
\end{eqnarray*}
Proof of (b) and (c) are same as that of proof of (a).
\end{proof}

\begin{notation}\label{N3.28}
Put,

\begin{enumerate}
  \item[(i)] $\phi^{s_1, s_2}_{2r_1 + r_2}(x) = \underset{j=0}{\overset{r_1-1}{\prod}}[ x^2-x-2(s_1+j)] \underset{l=0}{\overset{r_2-1}{\prod}} [x-(s_2+l)], \ \ \ r_1 \geq 1, r_2 \geq 1. $
      \item[(ii)] $\phi^{s_1, s_2}_{2r_1 + 0}(x) = \underset{j=0}{\overset{r_1-1}{\prod}}[ x^2-x-2(s_1+j)], \ \ \ r_2 = 0. $
          \item[(iii)] $\phi^{s_1, s_2}_{2.0 + r_2}(x) =  \underset{l=0}{\overset{r_2-1}{\prod}} [x-(s_2+l)], \ \ \  r_1 = 0. $
  \item[(iv)] $\phi^{s_1, s_2}_{0+0}(x) = 1$ and $\phi^{s_1, s_2}_{2r_1 + r_2}(x) = 0$ if any one of $r_1, r_2 < 0.$
  \item[(v)] $\phi^s_r(x) = \underset{l=0}{\overset{r-1}{\prod}} [x-(s+l)], \ \ \ r \geq 1$
\item[(vi)] $\phi^s_0(x) = 1$ and $\phi^s_r = 0$ if $r < 0.$
\end{enumerate}
\end{notation}

\NI Now, we derive the following relation between the polynomials which are needed in the following Lemmas:

\begin{lem} \label{L3.29}
\mbox{ }
\begin{enumerate}
  \item[(i)] $\phi^{s_1+t, s_2}_{2(r_1-t)+r_2}(x) = \phi^{s_1-t, s_2}_{2(r_1-t)+r_2}(x) - \underset{m=1}{\overset{2t}{\sum}} {_{2t}}C_m \ {_{r_1-t}}C_m \ 2^m \ m! \ \phi^{s_1+t, s_2}_{2(r_1-t-m)+ r_2}(x).$
  \item[(ii)] $\phi^{s_1, s_2+t}_{2r_1+r_2 -t}(x) = \phi^{s_1, s_2-t}_{2r_1+ r_2-t}(x) - \underset{m=1}{\overset{2t}{\sum}} {_{2t}}C_m \ {_{r_2-t}}C_m  \ m! \ \phi^{s_1, s_2+t}_{2r_1+ r_2 -t -m}(x).$
      \item[(iii)] In general,
      \begin{eqnarray*}
        \phi^{s_1+t_1, s_2+t_2}_{2(r_1-t_1)+r_2-t_2}(x) &=& \phi^{s_1-t_1, s_2-t_2}_{2(r_1-t_1)+r_2-t_2}(x) - \underset{k=1}{\overset{2t_1}{\sum}} 2t_1C_k \ (r_1-t_1)C_k \ 2^k \ k! \ \phi^{s_1+t_1, s_2-t_2}_{2(r_1-t_1-k) + r_2-t_2}(x) \\
         & & - \underset{k'=1}{\overset{2t_2}{\sum}} 2t_2C_{k'} \ (r_2-t_2)C_{k'} \ k'! \ \phi^{s_1-t_1, s_2+t_2}_{2(r_1-t_1) + r_2-t_2-k'}(x) \\
         & &  - \underset{k=1}{\overset{2t_1}{\sum}} \underset{k'=1}{\overset{2t_2}{\sum}}2t_1C_k \ (r_1-t_1)C_k \ 2^k \ k! \ 2t_2C_{k'}   \ (r_2-t_2)C_{k'} \ k'! \\
          & & \phi^{s_1+t_1, s_2+t_2}_{2(r_1-t_1-k) + r_2-t_2-k'}(x)
      \end{eqnarray*}

\end{enumerate}
\NI where $\phi^{s_1+t, s_2}_{2(r_1-t)+r_2}(x) = \underset{l=0}{\overset{r_1-t-1}{\prod}} [x^2-x-2(s_1+t+l)] \underset{l'=0}{\overset{r_2-1}{\prod}} [x-(s_2 + l')]$ and \\ $\phi^{s_1, s_2+t}_{2r_1+r_2-t}(x) = \underset{l=0}{\overset{r_1-1}{\prod}} [x^2-x-2(s_1+l)] \underset{l'=0}{\overset{r_2-t-1}{\prod}} [x-(s_2 +t+ l')].$
\end{lem}

\begin{proof}

\NI \textbf{Proof of (i):}We shall prove this by using induction on $r_1-t$ and $r_2.$
Consider
\begin{equation}\label{e3.11}
     \phi^{s_1-t, s_2}_{2(r_1-t)+r_2}(x) - \underset{m=1}{\overset{2t}{\sum}} {_{2t}}C_m \ {_{(r_1-t)}}C_m \ 2^m \ m! \ \phi^{s_1+t, s_2}_{2[r_1-t-m]+r_2}(x).
 \end{equation}
 \begin{eqnarray*}
    &=& \phi^{s_1-t, s_2}_{2(r_1-t-1)+r_2}(x)(x^2-x-2(s_1+r_1-2t-1)) \\
     & & - \underset{m=1}{\overset{2t}{\sum}} {_{2t}}C_m \ {_{(r_1-t)}}C_m \ 2^m \ m! \ \phi^{s_1+t, s_2}_{2(r_1-t-m)+r_2}(x) \\
     &=& \left(\phi^{s_1-t, s_2}_{2(r_1-t-1)+r_2}(x) +  \underset{m=1}{\overset{2t}{\sum}} {_{2t}}C_m \ {_{(r_1-t-1)}}C_m \ 2^m \ m! \ \phi^{s_1+t, s_2}_{2(r_1-t-m-1)+r_2}(x)\right)\\
     & & (x^2-x-2(s_1+r_1-2t-1)) - \underset{m=1}{\overset{2t}{\sum}} {_{2t}}C_m \ {_{(r_1-t)}}C_m \ 2^m \ m! \ \phi^{s_1+t, s_2}_{2(r_1-t-m)+r_2}(x)  \\
     & & \hspace{6cm} \text{ (by induction)}\\
    &=& \left(\phi^{s_1-t, s_2}_{2(r_1-t-1)+r_2}(x) +  \underset{m=1}{\overset{2t}{\sum}} {_{2t}}C_m \ {_{(r_1-t-1)}}C_m \ 2^m \ m! \ \phi^{s_1+t, s_2}_{2(r_1-t-m-1)+r_2}(x)\right)\\
    & & \hspace{4cm}(x^2-x-2(s_1+r_1-2t-1)) \\
    & & - \underset{m=1}{\overset{2t}{\sum}} {_{2t}}C_m \ \left({_{(r_1-t-1)}}C_m + {_{r_1-t-1}}C_{m-1} \right) \ 2^m \ m! \ \phi^{s_1+t, s_2}_{2(r_1-t-m-1)+r_2}(x)\\
    & & \hspace{4cm} (x^2-x-2(s_1+r_1-m-1))
     \end{eqnarray*}
       \begin{eqnarray*}
    &=&  \phi^{s_1+t, s_2}_{2(r_1-t-1)+r_2}(x) (x^2-x-2(s_1+r_1-2t-1)) \\
    & &- \underset{m=1}{\overset{2t}{\sum}} {_{2t}}C_m \ {_{(r_1-t-1)}}C_{m-1}  \ 2^m \ m!  \phi^{s_1+t, s_2}_{2(r_1-t-m)+r_2}(x)
 \\
     & &  + \underset{m=1}{\overset{2t}{\sum}} {_{2t}}C_m \ {_{(r_1-t-1)}}C_m  \ 2^m \ m! \ (4t-2m) \phi^{s_1+t, s_2}_{2(r_1-t-m-1)+r_2}(x)\\
            &=& \phi^{s_1+t, s_2}_{2(r_1-t-1)+r_2}(x) (x^2-x-2(s_1+r_1-2t-1)) - 4t \phi^{s_1+t, s_2}_{2(r_1-t-1)+ r_2}(x) \\
    &  & \hspace{0.5cm} \text{ (by canceling the common terms)}\\
    &=& \phi^{s_1+t, s_2}_{2(r_1-t-1)+r_2}(x) \left( x^2-x-2(s_1+r_1-1)\right)\\
    &=& \phi^{s_1+t, s_2}_{2(r_1-t)+r_2}(x)
 \end{eqnarray*}

\NI Proof of (ii) is similar to the proof of (i) and proof of (iii) follows from (i) and (ii).
\end{proof}

\begin{lem}\label{L3.30}
\mbox{ }
\begin{enumerate}
  \item[(a)]After performing the column operations to eliminate the non-zero entries corresponding to the diagrams coarser than both $d^{r_1, r_2}_{i, \alpha}$ and $d^{r_1, r_2}_{j, \alpha} $, the zero in the $((i, \alpha, r_1, r_2), (j, \beta, r_1, r_2))$ entry of the block matrix $A_{2r_1+r_2, 2r_1+r_2} $ for $0 \leq  r_1 + r_2 \leq k-s_1-s_2$ in algebra of $\mathbb{Z}_2$-relations  is replaced by

       \centerline{$-2^{t_1} \ t_1! \ t_2! \ x^{2(r_1-t_1)+r_2-t_2}$}
       \NI where $d^{r_1, r_2}_{i, \alpha}$ and $ d^{r_1, r_2}_{j, \beta}$ are as in Notation \ref{N3.25}(a).
      \item[(b)] After performing the column operations to eliminate the non-zero entries corresponding to the diagrams coarser than both $d^{r_1, r_2}_{i, \alpha}$ and $d^{r_1, r_2}_{j, \beta}$, the zero in the $((i,, \alpha, r_1, r_2), (j, \beta, r_1, r_2))$ entry of the block matrix $ \overrightarrow{A}_{2r_1+r_2, 2r_1+r_2}$ for $0 \leq r_1, r_2, r_1+r_2 \leq k -s_1-s_2 -1$ in the signed partition algebra  is replaced by

          \centerline{$-2^{t_1} \ t_1! \ t_2! \ x^{2(r_1-t_1)+r_2-t_2}.$}

          \NI where $d^{r_1, r_2}_{i, \alpha}$ and $ d^{r_1, r_2}_{j, \beta}$ are as in Notation \ref{N3.25}(b).
  \item[(c)]After performing the column operations to eliminate the non-zero entries corresponding to the diagrams coarser than both $R^{d^r_{i, \alpha}} $ and $R^{d^{r}_{j, \beta}} $, the zero in the $((i, \alpha, r), (j, \beta, r))$ entry of the block matrix $A_{r, r} $ for $o \leq r \leq k-s$ in partition algebra is replaced by

      \centerline{$ - \ t! \ x^{r-t}.$}

      \NI where  $R^{d^r_{i, \alpha}}$ and $R^{d^r_{j, \beta}}$ are as in Notation \ref{N3.25}(c).
\end{enumerate}
\end{lem}
\begin{proof}
\NI \textbf{Proof of (a):}
We shall prove this by induction on $t_1$ and  $t_2.$

\NI \textbf{Case (i):} Let $t_1 =1$ and $t_2 = 1.$

We know that the diagrams coarser than both $d^{r_1, r_2}_{i, \alpha} $ and $d^{r_1, r_2}_{j, \beta} $ are obtained if and only if the pair of $\{e\}$-through classes and the pair of $\{e\}$-horizontal edges of $d^{l_1^{f}}_{l_1^{f}}$ or $d^{l_2^{f}}_{l_2^{f}}$ is connected by an $\{e\}$-horizontal edge which can be done in two ways and $\mathbb{Z}_2$ horizontal edge and $\mathbb{Z}_2$-through class of $d^{l_1^{f}}_{l_1^{f}}$ or $d^{l_2^{f}}_{l_2^{f}}$ is connected by a $\mathbb{Z}_2$-edge which can be done in one way. Also $d^{l_1-f}_{l_1-f}$ and $d^{l_2-f}_{l_2-f}$ have $2(r_1-1)+r_2-1$ horizontal edges then after performing the column operations the zero in the $((i, \alpha, r_1, r_2), (j, \beta, r_1, r_2))$-entry of the block matrix $A_{2r_1+r_2, 2r_1+r_2}$ is replaced by

 \centerline{$ -2 \underset{l=0}{\overset{r_1-1}{\sum}} \underset{l'=0}{\overset{r_2-1+l}{\sum}} B^{s_1, s_2}_{2(r_1-1)+r_2-1, 2(r_1-1-l)+r_2-1+l'} \ \ \phi^{s_1, s_2}_{2[r_1-1-l]+ r_2-1+l'}(x) $}

 \NI which is equal to

 $-2 \phi^{s_1, s_2}_{2[r_1-1]+r_2-1}(x) -2 \underset{l=1}{\overset{r_1-1}{\sum}} \underset{l'=1}{\overset{r_2-1+l}{\sum}} B^{s_1, s_2}_{2(r_1-1)+r_2, 2(r_1-1-l)+r_2-1+l'} \ \ \phi^{s_1, s_2}_{2[r_1-1-l]+r_2-1+l'}(x)$

 \NI By Theorem \ref{T3.22} we know that,

 \begin{equation}\label{e3.12}
  \phi^{s_1, s_2}_{2[r_1-1]+r_2-1}(x) = x^{2(r_1-1)+r_2-1} - \underset{l=1}{\overset{r_1-1}{\sum}} \underset{l'=1}{\overset{r_2-1+l}{\sum}} B^{s_1, s_2}_{2(r_1-1)+r_2-1, 2(r_1-1-l)+ r_2-1+l'}
  \phi^{s_1, s_2}_{2[r_1-1-l]+r_2-1+l'}(x)
 \end{equation}

 Substituting equation (\ref{e3.12}) in the above expression and canceling the common terms we get,

\centerline{$-2x^{2(r_1-1)+r_2-1}.$}

 In general, the diagrams coarser than both $d^{r_1, r_2}_{i, \alpha} $ and $d^{r_1, r_2}_{j, \beta}$ are obtained if and only if $t_1$ pairs of $\{e\}$-through classes($t_2$ number of $(\mathbb{Z}_2)$-through classes) and $t_1$ pairs of $\{e\}$-horizontal edges($t_2$ number of $(\mathbb{Z}_2)$-horizontal edges) of $d^{l_1^{f}}_{l_1^{f}}$ or $d^{l_2^{f}}_{l_2^{f}} $ is connected by an $\{e\}$-horizontal edges($(\mathbb{Z}_2)$-horizontal edges) which can be done in $2^{t_1} \ t_1! \ t_2!$ ways. Also $d^{l_1-f}_{l_1-f}$ and $d^{l_2-f}_{l_2-f}$ have $2(r_1-t_1)+r_2-t_2$ horizontal edges then after performing the column operations to eliminate the non-zero entries corresponding to the diagrams coarser than both $d^{r_1, r_2}_{i, \alpha}$ and $d^{r_1, r_2}_{j, \alpha}$ the zero in the $((i, \alpha, r_1, r_2), (j, \beta, r_1, r_2))$-entry of the block matrix $A_{2r_1+r_2, 2r_1+r_2}$ is replaced by

 \centerline{$ -2^{t_1} \ t_1! \ t_2! \ \underset{l=0}{\overset{r_1-t_1}{\sum}} \underset{l'=0}{\overset{r_2-t_2+l}{\sum}} B^{s_1, s_2}_{2(r_1-t_1)+r_2-t_2, 2(r_1-t_1-l)+r_2-t_2+l'} \ \ \phi^{s_1, s_2}_{2[r_1-t_1-l]+ r_2-t_2+l'}(x) $}

 \NI which is equal to

 $-2^{t_1} \ t_1! \ t_2! \left( \phi^{s_1, s_2}_{2[r_1-t_1]+r_2-t_2}(x) - \underset{l=1}{\overset{r_1-t_1}{\sum}} \underset{l'=1}{\overset{r_2-t_2+l}{\sum}} B^{s_1, s_2}_{2(r_1-t_1)+r_2-t_2, 2(r_1-t_1-l)+r_2-t_2+l'} \\  \phi^{s_1, s_2}_{2[r_1-t_1-l]+r_2-t_2+l'}(x)\right)$

 Substituting equation (\ref{e3.12}) in the above expression and canceling the common terms we get,

\centerline{$-2^{t_1} \ t_1! \ t_2! \ x^{2(r_1-t_1)+r_2-t_2}.$}

\NI Proof of (b) and (c) are similar to the proof of (a).
\end{proof}

\begin{prop} \label{P3.31}
\mbox{ }
\begin{enumerate}
\item[(a)] For $0 \leq r_1+r_2 \leq k-s_1-s_2,$ after performing the column operations  to eliminate the non-zero entries which lie above corresponding to the diagrams coarser than $d^{r_1, r_2}_{j, \beta}$, then the $((i, \alpha, r_1, r_2), (j, \beta, r_1, r_2))$-entry of the block matrix $A_{2r_1+r_2, 2r_1+r_2}$  and

    \item[(b)] For $0 \leq r_1+r_2 \leq k-s_1-s_2-1,$ after performing the column operations  to eliminate the non-zero entries which lie above corresponding to the diagrams coarser than $d^{r_1, r_2}_{j, \beta}$, then the $((i, \alpha, r_1, r_2), (j, \beta, r_1, r_2))$-entry of the block matrix $\overrightarrow{A}_{2r_1+r_2, 2r_1+r_2}$ for $0 \leq r_1+r_2, r_1, r_2 \leq k - s_1-s_2 - 1$ are replaced as
\begin{enumerate}
  \item[(i)] $(-1)^{t_1+t_2}\ (t_1)! \ (t_2)! \ 2^{t_1} \underset{j=t_1}{\overset{r_1-1}{\prod}}[x^2-x-2(s_1+j)] \underset{l=t_2}{\overset{r_2-1}{\prod}} [x-(s_2+l)]$ \ \ \ \ if $r_1 \geq 1$ and $r_2 \geq 1,$
  \item[(ii)] $(-1)^{t_2} \  \ (t_2)! \ \underset{l=t_2}{\overset{r_2-1}{\prod}} [x-(s_2+l)]$ \ \ \ \ if $r_1 = 0$ and $r_2 \neq 0,$
  \item[(iii)] $(-1)^{t_1}\ (t_1)!  \ 2^{t_1} \underset{j=t_1}{\overset{r_1-1}{\prod}}[x^2-x-2(s_1+j)]$ \ \ \ \ if $r_1 \neq 0$ and $r_2 = 0,$
\end{enumerate}
 \NI where $ d^{r_1, r_2}_{j, \beta}$ is  as in Notation \ref{N3.25}.
\item[(c)] After performing the column operations  to eliminate the non-zero entries which lie above corresponding to the diagrams coarser than $R^{d^{r'}_{j, \beta}}$, then the $((i, \alpha, r), (j, \beta, r))$-entry is replaced by

    \centerline{$(-1)^{t}\ t!  \underset{j=t}{\overset{r-1}{\prod}} [x-(s+l)].$}
    \NI where $R^{d^{r}_{j, \beta}}$ is as in Notation \ref{N3.25}.
\end{enumerate}
\end{prop}

\begin{proof}
\NI \textbf{Proof of (a):}
We shall prove this by using induction on $t_1, t_2$,  the number of horizontal edges and the index of the diagram $(j, \beta, r_1, r_2).$

\NI By Lemma \ref{L3.27} the $((i, \alpha, r_1, r_2), (j, \beta, r_1, r_2))$ entry $b_{(i, \alpha, r_1, r_2), (j, \beta, r_1, r_2)}$ is given by
\begin{equation}\label{e3.13}
   b_{(i, \alpha, r_1, r_2), (j, \beta, r_1, r_2)}= - \underset{\substack{\ds d^{r''_1, r''_2}_{l, \gamma} > d^{r_1, r_2}_{i, \alpha} \\ \ds d^{r''_1, r''_2}_{l, \gamma} > d^{r'_1, r'_2}_{j, \beta}}}{\sum} b_{(l, \gamma, r''_1, r''_2), (l, \gamma, r''_1, r''_2)} -  \underset{\substack{\ds d^{r''_1, r''_2}_{l, \gamma} > d^{r_1, r_2}_{j, \beta} \\ \ds d^{r''_1, r''_2}_{l, \gamma} \ngtr d^{r_1, r_2}_{i, \alpha}}}{\sum} b_{(i, \alpha, r_1, r_2), (l, \gamma, r''_1, r''_2)}
\end{equation}

\NI \textbf{Case (i):}Let $t_1 = 0, t_2 = 1, r_1=0, r_2 = 1$ and $d^{l_1-f}_{l_1-f}$ and $d^{l_2-f}_{l_2-f}$ have $2s_1+s_2-1$ through classes and no horizontal edge. After applying column operations to eliminate the non-zero entries corresponding to the diagrams coarser than both $d^{l_1-f}_{l_1-f}$ and $d^{l_2-f}_{l_2-f}$ then by Lemma \ref{L3.30} and equation (\ref{e3.13}) the $((i, \alpha, 0, 1), (j, \beta, 0, 1))$-entry $b_{(i, \alpha, 0, 1), (j, \beta, 0, 1)}$ of the block matrix $A_{2\times 0+1, 2 \times 0+1}$ is given by

\centerline{$b_{(i, \alpha, 0, 1), (j, \beta, 0, 1)} = (-1) \ 1!.$}
\NI since there is no diagram coarser than $d^{0, 1}_{j, \beta}$ alone.

\NI \textbf{Case (ii):} Let $t_1 = 1, t_2 = 0, r_1 = 1, r_2 = 0$ and $d^{l_1-f}_{l_1-f}$ and $d^{l_2-f}_{l_2-f}$ have $2(s_1-1)+s_2$ through classes and no horizontal edge. After applying column operations to eliminate the non-zero entries corresponding to the diagrams coarser than both $d^{1, 0}_{i, \alpha}$ and $d^{1, 0}_{j, \beta}$ then by Lemma \ref{L3.30} and equation (\ref{e3.13}) the $((i, \alpha, 1, 0), (j, \beta, 1, 0))$-entry $b_{(i, \alpha, 1, 0), (j, \beta, 1, 0)}$ of the block matrix $A_{2 \times 1+0, 2 \times 1+0}$ is given by

\centerline{$b_{(i, \alpha, 1, 0), (j, \beta, 1, 0)} = (-1) \ 2 \ 1!.$}
\NI Since there is no diagram coarser than $d^{1, 0}_{j, \beta}$ alone.

   In general, suppose that the diagrams $d^{l_1-f}_{l_1-f}$ and $d^{l_2-f}_{l_2-f}$ have $2(s_1-t_1) + s_2 -t_2$ through classes and have $r_1-t_1$ pair of $\{e\}$-horizontal edges and $r_2 - t_2$ number of $\mathbb{Z}_2$-horizontal edges then after performing column operations to eliminate the coarser elements of $d^{r_1, r_2}_{i, \alpha}$ and $d^{r_1, r_2}_{j, \beta}$ having $t'$ pair of $\{e\}$-through classes ($\{e\}$-horizontal edges) with $t'< t,$ the $0$ in the $((i, \alpha, r_1, r_2), (j, \beta, r_1, r_2))$-entry $b_{(i, \alpha, r_1, r_2), (j, \beta, r_1, r_2)}$ of the block matrix $A_{2r_1+r_2, 2r_1+r_2}$ is replaced by $- \ (t_1)! \ (t_2)! \ 2^{t_1}\ x^{2(r_1 - t_1)+r_2-t_2}$ inductively.

      For, $0 \leq f' \leq t_1$ and $0 \leq f'' \leq t_2,$ the number of diagrams obtained by joining $f'$ pairs of $\{e\}$ through classes ($f''$ numbers of $\mathbb{Z}_2$ through classes) with $f'$ pairs of $\{e\}$-horizontal edges($f''$ numbers $\mathbb{Z}_2$ horizontal edges)  in $d^{l_2^f}_{l_2^f}$ is given by $(t_1C_{f'})^2 \ (t_2C_{f''})^2 \ 2^{f'} \ f'! \ f''!.$ The number of diagrams which are coarser than $d^{r_1, r_2}_{j, \beta}$ but not coarser than $d^{r_1, r_2}_{i, \alpha}$ having $(r_1 - t_1-l)$-pairs of $\{e\}$-horizontal edges and $r_2-t_2-l'$ number of $\mathbb{Z}_2$-horizontal edges is given by

   \NI \centerline{$\ds \underset{m=0}{\overset{2t_1 - 2f'}{\sum}} \ \underset{m'=0}{\overset{2t_2 - 2f''}{\sum}} \ (r_1-t_1-l+m)C_m (2t_1 - 2f')C_m \ 2^m \ m! \ (r_2-t_2-l'+m')C_{m'} \ (2t_2 - 2f'')C_{m'}   $}
  \begin{equation}\label{e3.14}
  \ds (m')! \ \left(t_1C_{f'}\right)^2 \ f'! \ 2^{f'} \ \left( t_2C_{f''} \right)^2 \ f''!  \ B^{s_1-t_1, s_2-t_2}_{2(r_1-t_1)+r_2-t_2, 2(r_1-t_1-l+m)+r_2-t_2-l'+m'}
   \end{equation}

    (\ref{e3.14}) is obtained by choosing $m$ pairs of $\{e\}$-horizontal edges  $( m' \text{ number of } \mathbb{Z}_2$-horizontal edges) for every diagram coarser than $d^{l_2-f}_{l_2-f}$ having $r_1-t_1-(l-m)$ pairs of $\{e\}$-horizontal edges $(r_2-t_2-(l'-m')$ number of $\mathbb{Z}_2$-horizontal edges) and choose $m$ pairs of $\{e\}$-connected components($m'$ number of $\mathbb{Z}_2$-connected components) from $d^{l_2^f}_{l_2^f}.$ Connecting the chosen $m$ pairs of $\{e\}$-horizontal edges  from $d^{l_2-f}_{l_2-f}$ to the $m$ pairs of $\{e\}$-connected components of $d^{l_2^f}_{l_2^f}$ by $\{e\}$-horizontal edge will give $2^m \ m! \ (m')!$ number of diagrams having $r_1-t_1-l$ pairs of $\{e\}$-horizontal edges. $m$ and $m'$ cannot exceed $2t_1 - 2f'$ and $2t_2-2f''$ respectively, since $d^{l_2^f}_{l_2^f}$ has $2t_1 - 2f'$-pairs of $\{e\}$-components and $2t_2-f''$ number of $\mathbb{Z}_2$-components.

$ b_{(i, \alpha, r_1, r_2), (j, \beta, r_1, r_2)}  $
   \begin{eqnarray*}
      & = & -2^{t_1} \ t_1! \ t_2! \ x^{2(r_1-t_1)+r_2-t_2}   - (-1)^{t_1+t_2} \ (t_1)! \ (t_2)! \ 2^{t_1} \Big \{ \underset{(l, l') \neq (0, 0)}{\underset{l=0}{\overset{r_1-t_1}{\sum}} \underset{l'=-l}{\overset{r_2-t_2}{\sum}}} \underset{m=0}{\overset{2t_1}{\sum}} \underset{m'=0}{\overset{2t_2}{\sum}} 2t_1C_m \\
      & & \hspace{3cm} (r_1-t_1-l+m)C_m \ 2^m \ m! \ 2t_2C_{m'} \  (r_2-t_2-l'+m')C_{m'} \ (m')! \  \\
       & & \hspace{3cm} B^{s_1-t_1, s_2-t_2}_{2(r_1-t_1)+r_2-t_2, 2(r_1-t_1-l+m)+r_2-t_2-l'+m'} \ \phi^{s_1+t_1, s_2+t_2}_{2(r_1-t_1-l)+r_2-t_2-l'}(x) \Big\} \\
      & & - \underset{\substack{(f', f'') \neq (0, 0) \\ (f', f'') \neq (t_1, t_2)}}{\underset{f'=0}{\overset{t_1}{\sum}} \underset{f''=0}{\overset{t_2}{\sum}}} \underset{l=0}{\overset{r_1-t_1}{\sum}} \underset{l'=-l}{\overset{r_2-t_2}{\sum}} \underset{m=0}{\overset{2t_1-2f'}{\sum}} \underset{m'=0}{\overset{2t_2-f''}{\sum}} \left({_{t_1}}C_{f'} \right)^2 \ 2^{f'} \ f'! \ \left({_{t_2}}C_{f''} \right)^2  \ f''!  (-1)^{t_1-f'} \  2^{t_1-f'} \ (t_1-f')!  \\
      & & (-1)^{t_2-f''} \ (t_2-f'')! \ 2(t_1-f')C_m \ (r_1-t_1-l+m)C_m \ 2^m \ m! \  2( t_2-f'')C_{m'}
              \end{eqnarray*}
      \begin{eqnarray*}
            & & (r_2-t_2-l'+m')C_{m'} \ (m')! \ B^{s_1-t_1+f', s_2-t_2+f''}_{2(r_1-t_1)+r_2-t_2, 2(r_1-t_1-l+m)+r_2-t_2-l'+m'}  \phi^{s_1+t_1-f', s_2+t_2-f''}_{2(r_1-t_1-l)+r_2-t_2-l'}(x)\\
            & = &  -2^{t_1} \ t_1! \ t_2! \ x^{2(r_1-t_1)+r_2-t_2}   - (-1)^{t_1+t_2} \ (t_1)! \ (t_2)! \ 2^{t_1} \Big \{ \underset{(l, l') \neq (0, 0)}{ \underset{l=0}{\overset{r_1-t_1}{\sum}} \underset{l'=-l}{\overset{r_2-t_2}{\sum}}} \underset{m=0}{\overset{2t_1}{\sum}} \underset{m'=0}{\overset{2t_2}{\sum}} 2t_1C_m \\
      & & \hspace{3cm} (r_1-t_1-l+m)C_m \ 2^m \ m! \ 2t_2C_{m'} \  (r_2-t_2-l'+m')C_{m'} \ (m')! \  \\
       & & \hspace{3cm} B^{s_1-t_1, s_2-t_2}_{2(r_1-t_1)+r_2-t_2, 2(r_1-t_1-l+m)+r_2-t_2-l'+m'} \ \phi^{s_1+t_1, s_2+t_2}_{2(r_1-t_1-l)+r_2-t_2-l'}(x) \Big\} \\
       & & - \underset{\substack{(f', f'') \neq (0, 0) \\ (f', f'') \neq (t_1, t_2)}}{\underset{f'=0}{\overset{t_1}{\sum}} \underset{f''=0}{\overset{t_2}{\sum}} } \left({_{t_1}}C_{f'} \right)^2 \ 2^{f'} \ f'! \ \left({_{t_2}}C_{f''} \right)^2  \ f''!  (-1)^{t_1-f'} \  2^{t_1-f'} \ (t_1-f')!\ (-1)^{t_2-f''} \ (t_2-f'')!              \\
      & & \Big\{\underset{l=0}{\overset{r_1-t_1}{\sum}} \underset{l'=-l}{\overset{r_2-t_2}{\sum}} \underset{m=0}{\overset{2t_1-2f'}{\sum}} \underset{m'=0}{\overset{2t_2-f''}{\sum}}  \ 2(t_1-f')C_m \ (r_1-t_1-l+m)C_m \ 2^m \ m! \  2( t_2-f'')C_{m'}\\
            & &  (r_2-t_2-l'+m')C_{m'} \ (m')! \ B^{s_1-t_1+f', s_2-t_2+f''}_{2(r_1-t_1)+r_2-t_2, 2(r_1-t_1-l+m)+r_2-t_2-l'+m'}  \phi^{s_1+t_1-f', s_2+t_2-f''}_{2(r_1-t_1-l)+r_2-t_2-l'}(x) \\
       & & \hspace{10cm}-  \phi^{s_1+t_1-f', s_2+t_2-f''}_{2(r_1-t_1)+r_2-t_2}(x) \Big\} \\
       & = &  -2^{t_1} \ t_1! \ t_2! \ x^{2(r_1-t_1)+r_2-t_2}  - (-1)^{t_1+t_2} \ (t_1)! \ (t_2)! \ 2^{t_1} \Big \{ \underset{(l, l') \neq (0, 0)}{\underset{l=0}{\overset{r_1-t_1}{\sum}} \underset{l'=-l}{\overset{r_2-t_2}{\sum}}} \underset{m=0}{\overset{2t_1}{\sum}} \underset{m'=0}{\overset{2t_2}{\sum}} 2t_1C_m \\
        & & (r_1-t_1-l+m)C_m \ 2^m \ m! \ 2t_2C_{m'} \ (r_2-t_2-l'+m')C_{m'} \ (m')! \\
      & & B^{s_1-t_1, s_2-t_2}_{2(r_1-t_1)+r_2-t_2, 2(r_1-t_1-l+m)+r_2-t_2-l'+m'} \  \phi^{s_1+t_1, s_2+t_2}_{2(r_1-t_1-l)+r_2-t_2-l'}(x) \Big\}\\
       & & - \underset{f'=0}{\overset{t_1}{\sum}} \underset{f''=0}{\overset{t_2}{\sum}} \ (t_1C_{f'})^2 \ (-1)^{t_1-f'} \ 2^{f'} \ f'! \  (-1)^{t_2-f''} (t_2C_{f''})^2 \ f''! \ 2^{t_1-f'} \ (t_1-f')! \  (t_2-f'')! \ x^{2(r_1-t_1)+r_2-t_2}\\
      & &  + (-1)^{t_1+t_2} \ 2^{t_1} \ t_1! \ t_2!   x^{2(r_1-t_1)+r_2-t_2}  + 2^{t_1} \ t_1! \ t_2! \ x^{2(r_1-t_1)+r_2-t_2}\\
      & = & (-1)^{t_1+t_2} \ (t_1)! \ (t_2)! \ 2^{t_1} \left\{ x^{2(r_1-t_1)+r_2-t_2}  -  \underset{(l, l') \neq (0, 0)}{\underset{l=0}{\overset{r_1-t_1}{\sum}} \underset{l'=-l}{\overset{r_2-t_2}{\sum}} } \underset{m=0}{\overset{2t_1}{\sum}} \underset{m'=0}{\overset{2t_2}{\sum}} 2t_1C_m \ (r_1-t_1-l+m)C_m \ 2^m \ m! \ \right. \\
      & & \left. 2t_2C_{m'} \  (r_2-t_2-l'+m')C_{m'} \ (m')! \ B^{s_1-t_1, s_2-t_2}_{2(r_1-t_1)+r_2-t_2, 2(r_1-t_1-l+m)+r_2-t_2-l'+m'} \ \phi^{s_1+t_1, s_2+t_2}_{2(r_1-t_1-l)+r_2-t_2-l'}(x) \right\} \\
            \end{eqnarray*}
            expanding and using Lemma \ref{L3.30} we get,
            \begin{eqnarray}
             &=& (-1)^{t_1+t_2} \ (t_1)! \ (t_2)! \ 2^{t_1} \left\{ x^{2(r_1-t_1)+r_2-t_2} - \underset{(l, l') \neq (0, 0)}{\underset{l=0}{\overset{r_1-t_1}{\sum}} \underset{l' = -l}{\overset{r_2-t_2}{\sum}}} B^{s_1-t_1, s_2-t_2}_{2(r_1-_1)+r_2-t_2, 2(r_1-t_1-l)+r_2-t_2-l'} \right.\nonumber \\
       & & \hspace{8cm} \phi^{s_1-t_1, s_2-t_2}_{2(r_1-t_1-l)+r_2-t_2-l'}(x)  \ \ \ \ \text{putting } m=0, m' = 0 \nonumber
         \end{eqnarray}
         \begin{eqnarray}
         & & + \underset{(l, l') \neq (0, 0)}{\underset{l=0}{\overset{r_1-t_1}{\sum}} \underset{l' = -l}{\overset{r_2-t_2}{\sum}}} \underset{k'=1}{\overset{2t_2}{\sum}} 2t_2C_{k'} \ (r_2-t_2-l')C_{k'} \ k'!    \nonumber \\
         & & \hspace{3cm} B^{s_1-t_1, s_2-t_2}_{2(r_1-t_1)+r_2-t_2, 2(r_1-t_1-l)+r_2-t_2-l'} \ \phi^{s_1-t_1, s_2+t_2}_{2(r_1-t_1-l)+r_2-t_2-l'-k'}(x)\\
                & & - \underset{(l, l') \neq (0, 0)}{\underset{l=0}{\overset{r_1-t_1}{\sum}} \underset{l' = -l}{\overset{r_2-t_2}{\sum}}}  \underset{m'=1}{\overset{2t_2}{\sum}} 2t_2C_{m'} \ (r_2-t_2-l'+m')C_{m'} \ m'! \nonumber \\
        & & \hspace{3cm}B^{s_1-t_1, s_2-t_2}_{2(r_1-t_1)+r_2-t_2, 2(r_1-t_1-l)+r_2-t_2-l'+m'} \ \phi^{s_1-t_1, s_2+t_2}_{2(r_1-t_1-l)+r_2-t_2-l'}(x)\\
               & & + \underset{(l, l') \neq (0, 0)}{\underset{l=0}{\overset{r_1-t_1}{\sum}} \underset{l' = -l}{\overset{r_2-t_2}{\sum}}} \underset{k=1}{\overset{2t_1}{\sum}} 2t_1C_{k} \ (r_1-t_1-l)C_{k} \ 2^k \ k! \nonumber \\
       & & \hspace{2cm} B^{s_1-t_1, s_2-t_2}_{2(r_1-t_1)+r_2-t_2, 2(r_1-t_1-l)+r_2-t_2-l'} \ \phi^{s_1+t_1, s_2-t_2}_{2(r_1-t_1-l-k)+r_2-t_2-l'}(x)\\
        & & - \underset{(l, l') \neq (0, 0)}{\underset{l=0}{\overset{r_1-t_1}{\sum}} \underset{l' = -l}{\overset{r_2-t_2}{\sum}}} \underset{m=1}{\overset{2t_1}{\sum}}  2t_1C_{m} \ (r_1-t_1-l+m)C_{m} \ 2^m \ m! \nonumber\\
              & & \hspace{2cm} B^{s_1-t_1, s_2-t_2}_{2(r_1-t_1)+r_2-t_2, 2(r_1-t_1-l+m)+r_2-t_2-l'} \ \phi^{s_1+t_1, s_2-t_2}_{2(r_1-t_1-l)+r_2-t_2-l'}(x)\\
              & & - \underset{(l, l') \neq (0, 0)}{\underset{l=0}{\overset{r_1-t_1}{\sum}} \underset{l' = -l}{\overset{r_2-t_2}{\sum}}} \underset{k=1}{\overset{2t_1}{\sum}} \underset{k'=1}{\overset{2t_2}{\sum}} 2t_1C_{k} \ (r_1-t_1-l)C_{k} \ 2^k \ k! \ 2t_2C_{k'} \ (r_2-t_2-l')C_{k'} \ k'!  \nonumber \\
        & & \hspace{2cm} B^{s_1-t_1, s_2-t_2}_{2(r_1-t_1)+r_2-t_2, 2(r_1-t_1-l)+r_2-t_2-l'} \ \phi^{s_1+t_1, s_2+t_2}_{2(r_1-t_1-l-k)+r_2-t_2-l'-k'}(x)\\
                                 & & - \underset{(l, l') \neq (0, 0)}{\underset{l=0}{\overset{r_1-t_1}{\sum}} \underset{l' = -l}{\overset{r_2-t_2}{\sum}}} \underset{m=1}{\overset{2t_1}{\sum}} \underset{k'=1}{\overset{2t_2}{\sum}} 2t_1C_{m} \ (r_1-t_1-l+m)C_{m} \ 2^m \ m! \ 2t_2C_{k'} \ (r_2-t_2-l')C_{k'} \ k'!  \nonumber \\
      & &  B^{s_1-t_1, s_2-t_2}_{2(r_1-t_1)+r_2-t_2, 2(r_1-t_1-l+m)+r_2-t_2-l'} \ \phi^{s_1+t_1, s_2+t_2}_{2(r_1-t_1-l)+r_2-t_2-l'-k'}(x)\\
                            & & + \underset{(l, l') \neq (0, 0)}{\underset{l=0}{\overset{r_1-t_1}{\sum}} \underset{l' = -l}{\overset{r_2-t_2}{\sum}}} \underset{k=1}{\overset{2t_1}{\sum}} \underset{m'=1}{\overset{2t_2}{\sum}} 2t_1C_{k} \ (r_1-t_1-l)C_{k} \ 2^k \ k! \ 2t_2C_{m'} \ (r_2-t_2-l'+m')C_{m'} \ m'! \nonumber  \\
     & &  B^{s_1-t_1, s_2-t_2}_{2(r_1-t_1)+r_2-t_2, 2(r_1-t_1-l)+r_2-t_2-l'+m'} \ \phi^{s_1+t_1, s_2+t_2}_{2(r_1-t_1-l-k)+r_2-t_2-l'}(x)  \\
      & & + \underset{(l, l') \neq (0, 0)}{\underset{l=0}{\overset{r_1-t_1}{\sum}} \underset{l' = -l}{\overset{r_2-t_2}{\sum}}} \underset{m=1}{\overset{2t_1}{\sum}} \underset{m'=1}{\overset{2t_2}{\sum}} 2t_1C_{m} \ (r_1-t_1-l+m)C_{m} \ 2^m \ m! \ 2t_2C_{m'} \ (r_2-t_2-l'+m')C_{m'} \ m'! \nonumber \\
      & &B^{s_1-t_1, s_2-t_2}_{2(r_1-t_1)+r_2-t_2, 2(r_1-t_1-l+m)+r_2-t_2-l'+m'} \ \phi^{s_1+t_1, s_2+t_2}_{2(r_1-t_1-l)+r_2-t_2-l'}(x)
      \end{eqnarray}

     \NI  Putting $(l= 0, l' = m')$ in equation (3.17), $(l=m, l' =0)$ in equation (3.19), $(l=m, l' = m')$ in equation  (3.23) and canceling the common terms , we get

\NI $ b_{(i, \alpha, r_1, r_2), (j, \beta, r_1, r_2)} = $
        \begin{eqnarray*}
& & (-1)^{t_1+t_2} \ (t_1)! \ (t_2)! \ 2^{t_1} \left\{ \phi^{s_1-t_1, s_2-t_2}_{2(r_1-t_1)+r_2-t_2}(x)\right.\\
  & & - \underset{m=1}{\overset{2t_1}{\sum}} 2t_1C_m \ (r_1-t_1)C_m \ 2^m \ m! \ \phi^{s_1+t_1, s_2-t_2}_{2(r_1-t_1-m)+r_2-t_2}(x)  \\
            & & -  \underset{m'=1}{\overset{2t_2}{\sum}} 2t_2C_{m'} \ (r_2-t_2)C_{m'}  \ m'! \ \phi^{s_1-t_1, s_2+t_2}_{2(r_1-t_1)+r_2-t_2-m'}(x) \\
            & & \left. - \underset{m=1}{\overset{2t_1}{\sum}} \underset{m'=1}{\overset{2t_2}{\sum}} 2t_1C_m \ (r_1-t_1)C_m \ 2^m \ m! \ 2t_2C_{m'} \ (r_2-t_2)C_{m'} \ m'! \ \phi^{s_1+t_1, s_2+t_2}_{2(r_1-t_1-m)+r_2-t_2-m'}(x)  \right\}\\
            &=&  (-1)^{t_1+t_2} \ (t_1)! \ (t_2)! \ 2^{t_1}  \phi^{s_1+t_1, s_2+t_2}_{2(r_1-t_1)+r_2-t_2}(x)
         \end{eqnarray*}
     \NI  Therefore the $((i, \alpha, r_1, r_2), (j, \beta, r'_1, r'_2))$-entry in the block matrix $A_{2r_1+r_2, 2r_1+r_2}$ is replaced as

    $$(-1)^{t_1+t_2}\ (t_1)! \ (t_2)! \ 2^{t_1} \underset{l=t_1}{\overset{r_1-1}{\prod}}[x^2-x-2(s_1+l)] \underset{m=t_2}{\overset{r_2-1}{\prod}} [x-(s_2+m)].$$

\NI Proof (b) and (c) are similar to the proof of (a).
\end{proof}

Now, we have the main theorem of this section.

\subsection{Main Theorem}
\begin{thm}\label{T3.32}
\mbox{ }
\begin{enumerate}
\item[(a)] Let $\widetilde{G}^k_{2s_1+s_2}$ be the matrix similar to the Gram matrix $G_{2s_1+s_2}^k$ of the algebra of $\mathbb{Z}_2$-relations which is obtained after the column operations and the corresponding row operations on $G_{2s_1+s_2}^k.$ Then

 \centerline{$\widetilde{G}^k_{2s_1 + s_2} = \left(\underset{\ds 0 \leq r_1 + r_2 \leq k-s_1-s_2 }{\bigoplus} \widetilde{A}_{2r_1+r_2, 2r_1+r_2} \right)$}

 \item[(b)]Let $\widetilde{\overrightarrow{G}}^k_{2s_1+s_2}$ be the matrix similar to the Gram matrix $\overrightarrow{G}_{2s_1+s_2}^k$ of signed partition algebras which is obtained after the column operations and the corresponding row operations on $\overrightarrow{G}_{2s_1+s_2}^k.$ Then

 \centerline{$\widetilde{\overrightarrow{G}}^k_{2s_1 + s_2} = \left(\underset{\substack{\ds 0 \leq r_1 \leq k-s_1-s_2-1 \\ \ds 0 \leq r_2 < k-s_1-s_2-1 \\ \ds 2r_1+r_2 \leq 2k - 2s_1 - 2s_2-1}}{\bigoplus} \widetilde{\overrightarrow{A}}_{2r_1+r_2, 2r_1+r_2} \right)\bigoplus \widetilde{\overrightarrow{A}}_{\rho}$}
\NI where
   \begin{enumerate}
   \item[(i)] the diagonal element of $\widetilde{A}_{2r_1+r_2, 2r_1+r_2}$ and $\widetilde{\overrightarrow{A}}_{2r_1+r_2, 2r_1+r_2}$ are given by

$\begin{array}{llll}
   1. & \underset{j=0}{\overset{r_1-1}{\prod}} [x^2-x-2(s_1+j)] \underset{l=0}{\overset{r_2-1}{\prod}} [x-(s_2+l)] & if & r_1 \geq 1, r_2 \geq 1 \\
   2. & \underset{j=0}{\overset{r_1-1}{\prod}} [x^2-x-2(s_1+j)]  & if & r_2 = 0 \\
   3. & \underset{l=0}{\overset{r_2-1}{\prod}} [x-(s_2+l)] & if &  r_1 = 0
 \end{array}
$
     \item[(ii)] The entry $b_{(i, \alpha, r_1, r_2), (j, \beta, r_1, r_2)}$ of the block matrix $\widetilde{A}_{2r_1 +r_2,2r_1+r_2}$ and $\widetilde{\overrightarrow{A}}_{2r_1+r_2, 2r_1+r_2}$ are replaced by

\hspace{-1cm}$\begin{array}{llll}
  1. & (-1)^{t_1 + t_2} \ 2^{t_1} \ (t_1)! \ (t_2)! \underset{j=0}{\overset{r_1-t_1-1}{\prod}} [x^2-x-2(s_1 + t_1+ j)] \underset{l=0}{\overset{r_2-t_2-1}{\prod}} [x-(s_2 + t_2 + l)]  & if & r_1 \geq 1, r_2 \geq 1 \\
  2. & (-1)^{t_1 } \ 2^{t_1} \ (t_1)! \ \underset{j=0}{\overset{r_1-t_1-1}{\prod}} [x^2-x-2(s_1 + t_1+ j)] & if & r_2 = 0 \\
  3. & (-1)^{t_2} \  (t_2)! \underset{l=0}{\overset{r_2-t_2-1}{\prod}} [x-(s_2 + t_2 + l)]  & if &  r_1 = 0
 \end{array}
$

 \NI whenever $d^{r_1, r_2}_{i, \alpha}$ and $d^{r_1, r_2}_{j, \beta}$ is  as in Remark \ref{R3.24}(a), Notation \ref{N3.25} and Proposition \ref{P3.31}.

\item[(iii)] All other entries  in the block matrix $\widetilde{A}_{2r_1+r_2, 2r_1+r_2}$ and $\widetilde{\overrightarrow{A}}_{2r_1+r_2, 2r_1+r_2}$ are zero.
   \end{enumerate}

  The underlying partitions of the diagrams corresponding to the entries of the block matrix $\widetilde{\overrightarrow{A}}_{2r_1+r_2, 2r_1+r_2}$ are $\alpha = [\alpha_1]^1 [\alpha_2]^2 [\alpha_3]^3 [ \alpha_4]^4$ with $\alpha_1 = \left( \alpha_{11}, \cdots, \alpha_{1s_1}\right), \alpha_2 = \left(  \alpha_{21}, \cdots, \alpha_{2s_2} \right), \alpha_{3} = \left(\alpha_{31}, \cdots, \alpha_{3r_1} \right), \alpha_{4} = \left(  \alpha_{41}, \cdots,  \alpha_{4r_2}\right)$ such that atleast one of $\alpha_{1i}, \alpha_{2j}, \alpha_{3l}, \alpha_{4m}$ is greater than $1$  for $1 \leq i \leq s_1, 1 \leq j \leq s_2, 1 \leq l \leq r_1$ and $1 \leq m\leq r_2.$

  Since the diagrams corresponding to the partition $\rho = [\rho_1]^1 [\rho_2]^2 [\rho_3]^3 [\rho_4]^4$ with $|\rho_{1i}| = 1, \ \forall 1 \leq i \leq s_1, |\rho_{2j}| = 1 \ \forall 1 \leq j \leq s_2, |\rho_{3m}| = 0 \ \forall 1 \leq m \leq r_1$ and $|\rho_{4l}| = 1 \ \forall 1 \leq l \leq r_2$ does not belong to the signed partition algebra. Thus the block corresponding to the diagrams whose underlying partition is $\rho$ is studied separately.

 \item[(b)$'$]Let $\widetilde{\overrightarrow{A}}_{\rho}$ where  the partition $\rho$ is such that each $\rho_{1i}, \rho_{2j}, \rho_{3l}, \rho_{4m}$ is equal to $1$ for $1 \leq i \leq s_1, 1 \leq j \leq s_2, 1 \leq l \leq r_1$ and $1 \leq m\leq r_2$ and $\widetilde{\overrightarrow{A}}_{\rho}$ is the block sub matrix corresponding to the diagrams whose underlying partition is $\rho.$

\begin{enumerate}
  \item[(i)] The $((i, \rho, r'_1, r'_2), (i, \rho, r'_1, r'_2))$-entry $x^{2r'_1+r'_2}$ in the matrix $\widetilde{\overrightarrow{A}}_{\rho}$  is replaced by

\centerline{$\underset{j=0}{\overset{r'_1-1}{\prod}} [x^2-x-2(s_1+j)] \underset{l=0}{\overset{l=r'_2-1}{\prod}} [x-(s_2+l)] + \underset{l=0}{\overset{k-s_1-s_2-1}{\prod}} [x-(s_2+l)]$}

\NI where $1 \leq r'_1 \leq k - s_1 - s_2 $ and  $r'_2 = k - s_1 - s_2 - r'_1.$
  \item[(ii)] The zero in the $((i, \rho, r'_1, r'_2), (j, \rho, r'_1, r'_2))$-entry  in the block matrix $\widetilde{\overrightarrow{A}}_{\rho}$ is replaced by

  $(-1)^{t_1+t_2} \ 2^{t_1} \ (t_1)! \ (t_2)! \ \underset{j=0}{\overset{r'_1-t_1-1}{\prod}} [x^2-x-2(s_1+t_1+j)]  \underset{l=0}{\overset{r'_2-t_2-1}{\prod}} [x-(s_2+l+t_2)]\\ + \underset{l=0}{\overset{k-s_1-s_2-1}{\prod}} [x-(s_2+l)]$

  \NI where $d^{r'_1, r'_2}_{i, \rho}$ and $d^{r'_1, r'_2}_{j, \rho}$  are as in Remark \ref{R3.24}(a),  Notation \ref{N3.25} and Proposition \ref{P3.31} where $1 \leq i, j \leq 2k - 2s_1 - 2s_2$ and $i \neq j.$

  \item[(iii)] If $\sharp^p \left(d^{r'_1, k-s_1-s_2-r'_1}_{i, \rho} . d^{r_1, k-s_1-s_2-r_1}_{j, \rho}\right) = 2s_1 + s_2$ then the  $((i, \rho, r'_1, k-s_1-s_2-r'_1), (j, \rho, r_1, k-s_1-s_2-r_1))$-entry in the block matrix $\widetilde{\overrightarrow{A}}_{\rho}$ is  replaced by

  \centerline{$ (-1)^{r_1 + r'_1} \ \underset{l=0}{\overset{k-s_1-s_2-1}{\prod}} [x-(s_2+l)]$}

\NI where  $1 \leq i, j \leq 2k - 2s_1 - 2s_2$ and $i \neq j.$

\NI \item[(iv)] All other entries in the block matrix  $\widetilde{\overrightarrow{A}}_{\rho}$ are zero.
\end{enumerate}
\item[(c)] Let $\widetilde{G}^k_{s}$ be the matrix similar to the Gram matrix $G_{s}^k$ which is obtained after the column operations and the row operations on $G_{s}^k.$ Then

 \centerline{$\widetilde{G}^k_{s} = \left(\underset{0 \leq r \leq k-s}{\bigoplus} \widetilde{A}_{r, r} \right)$}
\NI where
   \begin{enumerate}
   \item[(i)] The diagonal element of $\widetilde{A}_{r, r}$ is given by

 \centerline{$ \underset{l=0}{\overset{r-1}{\prod}} [x-(s+l)]$}
     \item[(ii)] The entry $b_{(i, \alpha, r), (j, \beta, r)}$ of the block matrix $\widetilde{A}_{r, r}$ is replaced by

     \centerline{$ (-1)^{t}  \ (t)! \  \underset{j=t}{\overset{r-1}{\prod}}  [x-(s  + l)] $} whenever $R^{d^{r}_{i, \alpha}}$ and $R^{d^{r}_{j, \beta}}$ are  as in Remark \ref{R3.24}(b),   Notation \ref{N3.25} and Proposition \ref{P3.31}.

\item[(iii)] All other entries in the block matrix $\overrightarrow{A}_{r, r}$ are zero.
   \end{enumerate}

\end{enumerate}
\end{thm}
\begin{proof}

\NI \textbf{Proof of (a):} Every entry $x^{2r_1+r_2}$ in the sub block matrix $\widetilde{A}_{2r_1+r_2, 2r'_1+r'_2}$ is also replaced by

 \centerline{$\underset{j=0}{\overset{r_1-1}{\prod}} [x^2-x-2(s_1+j)] \underset{l=0}{\overset{r_2-1}{\prod}} [x-(s_2+l)]$}

We continue to do the column operations for all the diagrams whose underlying partition is $\alpha$ where $\alpha = [\alpha_1]^1 [ \alpha_2]^2 [\alpha_3] [ \alpha_4]^4$ with $\alpha_1 = \left( \alpha_{11}, \cdots, \alpha_{1s_1}\right), \alpha_2 = \left(  \alpha_{21}, \cdots, \alpha_{2s_2} \right), \alpha_{3} = \left(\alpha_{31}, \cdots, \alpha_{3r_1} \right),  \alpha_{4} = \left(  \alpha_{41}, \cdots,  \alpha_{4r_2}\right)$ such that atleast one of $\alpha_{1i}, \alpha_{2j}, \alpha_{3l}, \alpha_{4m}$ is greater than $1$ and hence the above entry gets eliminated.

Thus, from Lemmas \ref{L3.19} and \ref{L3.27} it follows that the rectangular sub matrix $\widetilde{A}_{2r_1+r_2, 2r'_1+r'_2}$ with $2r_1+r_2 \neq 2r'_1+r'_2$ becomes zero after all the column operations are carried out.

 After applying the row operations corresponding to the column operations performed in Lemmas \ref{L3.23},  \ref{L3.27}, Proposition \ref{P3.31},  and Theorem \ref{T3.22}, the Gram matrix $G_{2s_1+s_2}^k$ which is similar to a matrix $\widetilde{G}^k_{2s_1+s_2}$  decomposes as a direct sum of block matrices.

\centerline{i.e.,$\widetilde{G}^k_{2s_1 + s_2} = \left(\underset{\ds 0 \leq r_1 +r_2 \leq k-s_1-s_2}{\bigoplus} \widetilde{A}_{2r_1+r_2, 2r_1+r_2} \right) $}

\NI where the diagonal element of $\widetilde{A}_{2r_1+r_2, 2r_1+r_2}$ is given by

\centerline{$\underset{j=0}{\overset{r_1-1}{\prod}} [x^2-x-2(s_1+j)] \underset{l=0}{\overset{r_2-1}{\prod}} [x-(s_2+l)]$.}

\NI Result (i) follows from Theorem \ref{T3.22}(a), result (ii) follows from Proposition \ref{P3.31}(a) and result (iii) follow from Lemmas \ref{L3.21}, \ref{L3.23}, and \ref{L3.27}(a) respectively.

\NI \textbf{Proof of (b)$'$:}

 The column operations corresponding to the diagrams whose underlying partition is $\rho$ where

 \NI $\rho = [\rho_1]^1 [ \rho_2]^2 [\rho_3]^3 [ \rho_4]^4$ where $|\rho_{1i}| = 1, \ \forall \ 1 \leq i \leq s_1, |\rho_{2 j}| = 1 \ \forall \ 1 \leq j \leq s_2, |\rho_{3m}| = 0, \ \forall \ 1 \leq m \leq r_1$ and $|\rho_{4l}| = 1 \ \forall \ 1 \leq l \leq r_2$ such that $s_1 + s_2 + r_2 = k$ with $s_1 \lvertneqq k$ cannot be carried out for the block matrix $\widetilde{\overrightarrow{A}}_{\rho}$, since those diagrams do not belong to the signed partition algebra.

\NI \textbf{Proof of (i):}We prove the result by induction.

\NI \textbf{Case (i):} Let $d^{1, k-s_1-s_2-1}_{i, \rho}$ be a diagram in $\overrightarrow{\mathbb{J}}^{2.1 + k-s_1-s_2-1}_{2s_1+s_2}$, after the column operations the $((i, \rho, 1, k-s_1-s_2-1), (i, \rho, 1, k-s_1-s_2-1))$-entry corresponding to the diagram $d^{1, k-s_1-s_2-1}_{i, \rho}$ will be replaced by

\centerline{$\phi^{s_1, s_2}_{2.1+k-s_1-s_2-1}(x) + \phi^{s_1, s_2}_{2.0+k-s_1-s_2}(x)$}

\NI since the signed partition algebra does not contain diagrams with $k-s_1-s_2$ number of $\mathbb{Z}_2$-horizontal edges.

\NI \textbf{Case (ii):}  Let $d^{2, k-s_1-s_2-2}_{i, \rho}$ be a diagram in $\overrightarrow{\mathbb{J}}^{2.2 + k-s_1-s_2-2}_{2s_1+s_2}.$

After applying the column operations $L_{(i, \rho, 2, k-s_1-s_2-2)} \rightarrow L_{(i, \rho, 2, k-s_1-s_2-2)} - L_{(k, \alpha, r_1, r_2)}$ for all $d^{r_1, r_2}_{k, \alpha}$ where $d^{r_1, r_2}_{k, \alpha} \in \overrightarrow{\mathbb{J}}_{2s_1+s_2}^{2r_1+r_2}$ with $r_1+r_2+s_1+s_2 \leq k-1,$ the $((i, \rho, 2, k-s_1-s_2-2), (i, \rho, 2, k-s_1-s_2-2))$-entry will be replaced by

\centerline{$\phi^{s_1, s_2}_{2.2+k-s_1-s_2-2}(x) + 2 \phi^{s_1, s_2}_{2.1+k-s_1-s_2-1}(x) + \phi^{s_1, s_2}_{2.0+k-s_1-s_2}(x)$}

Again applying the column operations inside the block matrix $\widetilde{\overrightarrow{A}}_{\rho}$, the $((i, \rho, 2, k-s_1-s_2-2), (i, \rho, 2, k-s_1-s_2-2))$-entry becomes

\centerline{$\phi^{s_1, s_2}_{2.2+k-s_1-s_2-2}(x) + 2 \phi^{s_1, s_2}_{2.1+k-s_1-s_2-1}(x) + \phi^{s_1, s_2}_{2.0+k-s_1-s_2}(x)-2\left[ \phi^{s_1, s_2}_{2.1 + k - s_1-s_2-1}(x) + \phi^{s_1, s_2}_{2.0 + k - s_1-s_2}(x)\right]$}

$ = \phi^{s_1, s_2}_{2.2+k-s_1-s_2-2}(x) - \phi^{s_1, s_2}_{2.0 + k- s_1-s_2}(x).$

After applying the row operations corresponding to the column operations which is performed to obtain the above $((i, \rho, 2, k-s_1-s_2-2), (i, \rho, 2, k-s_1-s_2-2))$-entry, the $((i, \rho, 2, k-s_1-s_2-2), (i, \rho, 2, k-s_1-s_2-2))$-entry  further becomes

$\phi^{s_1, s_2}_{2.2+k-s_1-s_2-2}(x) - \phi^{s_1, s_2}_{2.0 + k- s_1-s_2}(x) + 2 \phi^{s_1, s_2}_{2.0 + k-s_1-s_2}(x)$

$ = \phi^{s_1, s_2}_{2.2+k-s_1-s_2-2}(x) + \phi^{s_1, s_2}_{2.0 + k- s_1-s_2}(x).$

In general, Let $d^{j, k-s_1-s_2-j}_{i, \rho}$ be a diagram in $\overrightarrow{\mathbb{J}}^{2.j + k-s_1-s_2-j}_{2s_1+s_2}.$

After applying the column operations, by induction the $((i, \rho, j, k-s_1-s_2-j), (i, \rho, j, k-s_1-s_2-j))$-entry of the matrix $\widetilde{\overrightarrow{A}}_{\rho}$ becomes

$\phi^{s_1, s_2}_{2j + k- s_1-s_2 -j}(x) + \underset{m=1}{\overset{j-1}{\sum}} {_{j}}C_m \ \phi^{s_1, s_2}_{2(j-m)+ k-s_1-s_2-j+m}(x) + \phi^{s_1, s_2}_{2.0 + k - s_1-s_2}(x)$

 $ - \underset{m=1}{\overset{j-1}{\sum}} {_{j}}C_m \left( \phi^{s_1, s_2}_{2(j-m)+ k-s_1-s_2-j+m}(x) + \phi^{s_1, s_2}_{2.0 + k - s_1-s_2}(x)\right)$

 $= \phi^{s_1, s_2}_{2j + k- s_1-s_2 -j}(x) - \underset{m=1}{\overset{j-1}{\sum}} {_{j}}C_m \ \phi^{s_1, s_2}_{2.0+ k-s_1-s_2}(x) + \phi^{s_1, s_2}_{2.0 + k - s_1-s_2}(x).$

 After applying the row operations the $((i, \rho, j, k-s_1-s_2-j), (i, \rho, j, k-s_1-s_2-j))$-entry  further becomes

 $\phi^{s_1, s_2}_{2j + k- s_1-s_2 -j}(x) - \underset{m=1}{\overset{j-1}{\sum}} {_{j}}C_m \ \phi^{s_1, s_2}_{2.0+ k-s_1-s_2}(x) + \phi^{s_1, s_2}_{2.0 + k - s_1-s_2}(x) + \underset{m=1}{\overset{j-1}{\sum}} {_{j}}C_m \ \phi^{s_1, s_2}_{2.0+ k-s_1-s_2}(x)$

 $ = \phi^{s_1, s_2}_{2j + k- s_1-s_2 -j}(x) + \phi^{s_1, s_2}_{2.0 + k- s_1-s_2 }(x)$

Thus, for a diagram $d^{r'_1, k-s_1-s_2-r'_1}_{i, \rho} \in \overrightarrow{\mathbb{J}}_{2s_1+s_2}^{2r'_1+k-s_1-s_2-r'_1}$ the $((i, \rho, r'_1, k-s_1-s_2-r'_1), (i, \rho, r'_1, k-s_1-s_2-r'_1))$-entry is replaced as

\centerline{$ \underset{j=0}{\overset{r'_1-1}{\prod}}[x^2-x-2(s_1+j)] \underset{l=0}{\overset{k-s_1-s_2-r'_1-1}{\prod}}[x-(s_2+l)] +  \underset{l=0}{\overset{k-s_1-s_2-1}{\prod}}[x-(s_2+l)].$}

\textbf{Proof of (ii):} The proof follows from  Proposition \ref{P3.31}(b) and it is similar to the Proof of (1), whenever $d^{r'_1, r'_2}_{i, \rho}$ and $d^{r'_1, r'_2}_{j, \rho}$ are as in Notation \ref{N3.25}.

\textbf{Proof of (iii):}

 \NI \textbf{Case (i):}Let $d^{1, k-s_1-s_2-1}_{i, \rho} \in \overrightarrow{\mathbb{J}}^{2.1 + k-s_1-s_2-1}_{2s_1+s_2}$ and $d^{2, k-s_1-s_2-2}_{j, \rho} \in \overrightarrow{\mathbb{J}}^{2.2+k-s_1-s_2-2}_{2s_1+s_2}$ such that number of $\{e\}$-horizontal edges in $d^{1, k-s_1-s_2-1}_{i, \rho}$ is greater than the number of $\{e\}$-horizontal edges in $d^{2, k-s_1-s_2-2}_{j, \rho}$ then $l \left(d^{1, k-s_1-s_2-1}_{i, \rho} . d^{2, k-s_1-s_2-2}_{j, \rho} \right) \leq 2.1+k-s_1-s_2-1.$

 There will be two diagrams say $d^{1, k-s_1-s_2-1}_{i', \rho}$ and $d^{1, k-s_1-s_2-1}_{i'', \rho}$ coarser than $d^{2, k-s_1-s_2-2}_{j, \rho}.$

 \NI \textbf{Subcase (i):} Suppose $l  \left(d^{1, k-s_1-s_2-1}_{i, \rho} . d^{2, k-s_1-s_2-2}_{j, \rho} \right) = 2.1 + k - s_1- s_2 -1$ then

 \centerline{$a_{(i, \rho, 1, k-s_1-s_2-1), (j, \rho, 2, k-s_1-s_2-2)} = \phi^{s_1, s_2}_{2.1+k-s_1-s_2-1}(x) + \phi^{s_1, s_2}_{2.0+k-s_1-s_2}(x).$}

\NI Also,

$\begin{array}{lll}
            a_{(i, \rho, 1, k-s_1-s_2-1), (i', \rho, 1, k-s_1-s_2-1)} & = & \phi^{s_1, s_2}_{2.1+k-s_1-s_2-1}(x) + \phi^{s_1, s_2}_{2.0+k-s_1-s_2}(x) \text{ and } \\
            a_{(i, \rho, 1, k-s_1-s_2-1), (i'', \rho, 1, k-s_1-s_2-1)} & = &  \phi^{s_1, s_2}_{2.0+k-s_1-s_2}(x)
          \end{array}$

\NI or

$\begin{array}{lll}
            a_{(i, \rho, 1, k-s_1-s_2-1), (i', \rho, 1, k-s_1-s_2-1)} & = & \phi^{s_1, s_2}_{2.0+k-s_1-s_2}(x) \\
           a_{(i, \rho, 1, k-s_1-s_2-1), (i'', \rho, 1, k-s_1-s_2-1)}  & = & \phi^{s_1, s_2}_{2.1+k-s_1-s_2-1}(x) + \phi^{s_1, s_2}_{2.0+k-s_1-s_2}(x)
          \end{array}$

\NI After applying the column operations the $((i, \rho, 1, k-s_1-s_2-1), (j, \rho, 2, k-s_1-s_2-2)), $-entry in $\widetilde{\overrightarrow{A}}_{\rho}$ becomes

\NI $b_{(i, \rho, 1, k-s_1-s_2-1), (j, \rho, 2, k-s_1-s_2-2)}  $
\begin{eqnarray*}
   & = & a_{(i, \rho, 1, k-s_1-s_2-1), (j, \rho, 2, k-s_1-s_2-2)} - a_{(i, \rho, 1, k-s_1-s_2-1), (i', \rho, 1, k-s_1-s_2-1)} - a_{(i, \rho, 1, k-s_1-s_2-1), (i'', \rho, 1, k-s_1-s_2-1)} \\
   &=&  -\phi^{s_1, s_2}_{2.0+k-s_1-s_2}(x)
\end{eqnarray*}
\NI \textbf{Subcase (ii):}Suppose $l \left(d^{1, k-s_1-s_2-1}_{i, \rho} . d^{2, k-s_1-s_2-2}_{j, \rho} \right) < 2.1 + k - s_1- s_2 -1$ then

\centerline{$a_{(i, \rho, 1, k-s_1-s_2-1), (j, \rho, 2, k-s_1-s_2-2)} = \phi^{s_1, s_2}_{2.0+k-s_1-s_2}(x).$}

\NI Also,

 $\begin{array}{lll}
   a_{(i, \rho, 1, k-s_1-s_2-1), (i', \rho, 1, k-s_1-s_2-1)}  & = & \phi^{s_1, s_2}_{2.0+k-s_1-s_2}(x) \text{ and } \\
   a_{(i, \rho, 1, k-s_1-s_2-1), (i'', \rho, 1, k-s_1-s_2-1)}  & = &  \phi^{s_1, s_2}_{2.0+k-s_1-s_2}(x)
  \end{array}
 $

\NI After applying the column operations the $(i, \rho, 1, k-s_1-s_2-1), (j, \rho, 2, k-s_1-s_2-2)$-entry in $\widetilde{\overrightarrow{A}}_{\rho}$ becomes

\NI $b_{(i, \rho, 1, k-s_1-s_2-1), (j, \rho, 2, k-s_1-s_2-2)}$
\begin{eqnarray*}
   &=& a_{(i, \rho, 1, k-s_1-s_2-1), (j, \rho, 2, k-s_1-s_2-2)} - a_{(i, \rho, 1, k-s_1-s_2-1), (i', \rho, 1, k-s_1-s_2-1)} - a_{(i, \rho, 1, k-s_1-s_2-1), (i'', \rho, 1, k-s_1-s_2-1)} \\
   &=&  -\phi^{s_1, s_2}_{2.0+k-s_1-s_2}(x)
\end{eqnarray*}

 In general, let $d^{r'_1, k-s_1-s_2-r'_1}_{i, \rho} \in \overrightarrow{\mathbb{J}}^{2r'_1 + k-s_1-s_2-r'_1}_{2s_1+s_2}$ and $d^{r_1, k-s_1-s_2-r_1}_{j,\rho} \in \overrightarrow{\mathbb{J}}^{2r_1+k-s_1-s_2-r_1}_{2s_1+s_2}$ such that the number of $\{e\}$-horizontal edges in $d^{r_1, k-s_1-s_2-r_1}_{j, \rho}$ is strictly greater than the number of $\{e\}$-horizontal edges in $d^{r'_1, k-s_1-s_2-r'_1}_{i, \rho}$ then $l \left(d^{r'_1, k-s_1-s_2-r'_1}_{i, \rho} . d^{r_1, k-s_1-s_2-r_1}_{j, \rho} \right) \leq 2r'_1+k-s_1-s_2-r'_1.$

After applying the column operations the $((i, \rho, r'_1, k-s_1-s_2-r'_1), (j, \rho, r_1, k-s_1-s_2-r_1))$-entry becomes

\NI $ b_{(i, \rho, r'_1, k-s_1-s_2-r'_1), (j, \rho, r_1, k-s_1-s_2-r_1)}$
\begin{eqnarray*}
  & & \left(\underset{m=1}{\overset{r_1-1}{\sum}} (-1)^{m-1} \ {_{r_1}}C_m - 1 \right) \phi^{s_1, s_2}_{2.0+k-s_1-s_2}(x) \\
   &=& (-1)^{r_1+1} \phi^{s_1, s_2}_{2.0+k-s_1-s_2}(x)
  \end{eqnarray*}

After applying row operations the $((i, \rho, r'_1, k-s_1-s_2-r'_1), (j, \rho, r_1, k-s_1-s_2-r_1))$-entry  further becomes
\begin{eqnarray*}
   b_{(i, \rho, r'_1, k-s_1-s_2-r'_1), (j, \rho, r_1, k-s_1-s_2-r_1)} &=& \left(\underset{m=1}{\overset{r'_1-1}{\sum}} (-1)^{m-1} \ {_{r'_1}}C_m - 1\right) (-1)^{r_1+1}\phi^{s_1, s_2}_{2.0+k-s_1-s_2}(x) \\
   &=& (-1)^{r'_1+r_1} \ \phi^{s_1, s_2}_{2.0+k-s_1-s_2}(x)
  \end{eqnarray*}

Thus, the $((i, \rho, r'_1, k-s_1-s_2-r'_1), (j, \rho, r_1, k-s_1-s_2-r_1))$-entry of the block matrix $\widetilde{\overrightarrow{A}}_{\rho}$ is given by

\centerline{$ (-1)^{r_1+r'_1} \ \phi^{s_1, s_2}_{2.0+k-s_1-s_2}(x).$}

\NI Proof of (b) and (c) are similar to the proof of (a).
\end{proof}

\begin{rem}\label{R3.33}
\begin{enumerate}
  \item[(a)]$\widetilde{G}^k_{0+0} = \underset{\ds 0 \leq r_1 + r_2 \leq k } \oplus \widetilde{A}_{2r_1 + r_2, 2r_1+r_2} $

  \item[(b)]$\widetilde{\overrightarrow{G}}^k_{0+0} = \underset{\substack{\ds 0 \leq r_1 \leq k-1 \\ \ds 0 \leq r_2 \leq k-1 \\ \ds 2r_1 + r_2 \leq 2k - 1}} \oplus \widetilde{\overrightarrow{A}}^k_{2r_1 + r_2, 2r_1+r_2} \oplus \widetilde{\overrightarrow{A}}_{\rho}$

\NI where  $\widetilde{A}_{2r_1+r_2, 2r_1+r_2}$ and $\widetilde{\overrightarrow{A}}_{2r_1+r_2, 2r_1+r_2}$ are the diagonal block matrices whose diagonal entry is given by
\begin{enumerate}
  \item[(i)] $\underset{j=0}{\overset{r_1-1}{\prod}}[x^2-x-2j]  \underset{l=0}{\overset{r_2-1}{\prod}} [x - l], \ \ \ r_1 \geq 1, r_2 \geq 1,$
  \item[(ii)] $  \underset{l=0}{\overset{r_2-1}{\prod}} [x - l], \ \ \ r_1 = 0, $
  \item[(iii)] $\underset{j=0}{\overset{r_1-1}{\prod}}[x^2-x-2j], \ \ \  r_2 = 0,$
\end{enumerate}

 \item[(b)']Let $\widetilde{\overrightarrow{A}}_{\rho}$ where the partition $\rho$ is such that $\rho_{1i} = \Phi, \rho_{2j}  = \Phi, \rho_{3l} = 1, \rho_{4m} = 1$ for $1 \leq i \leq s_1, 1 \leq j \leq s_2, 1 \leq l \leq r_1$ and $1 \leq m \leq r_2$ and $\widetilde{\overrightarrow{A}}_{\rho}$ is the block sub matrix corresponding to the diagrams whose underlying partition is $\rho.$

\NI The $((i, \rho, r'_1, r'_2), (i, \rho, r'_1, r'_2))$-entry $x^{2r'_1+r'_2}$ of the matrix $\widetilde{\overrightarrow{A}}_{\rho}$ is replaced by

      \centerline{$\underset{j=0}{\overset{r'_1-1}{\prod}}[x^2 -x - 2j] \underset{l=0}{\overset{r'_2-1}{\prod}}[x-l] +  \underset{l=0}{\overset{r'_1+r'_2-1}{\prod}}[x-l].$}

  \item[(c)] $G^k_{0} = \underset{0 \leq r \leq k } \oplus \widetilde{A}_{r, r} $

\centerline{$  \underset{l=0}{\overset{r-1}{\prod}} [x - l].$ }
\end{enumerate}
\end{rem}

\section{\textbf{Semisimplicity of Signed Partition Algebras}}

Semisimplicity of the algebra of $\mathbb{Z}_2$-relations and partition algebras are already discussed in  [15] and [2] respectively. In this paper, we give an alternate approach to show that the partition algebras and the algebra of $\mathbb{Z}_2$-relations are semisimple. We also study about the semisimplicity of signed partition algebras.

\begin{defn}\cite{K1} \label{D4.1}
Let $s = 2s_1 + s_2.$ For $0 \leq s \leq 2k$ and $((s, (s_1, s_2)), ((\lambda_1, \lambda_2), \mu)) \in \Lambda' \\ \left(((s, (s_1, s_2)), ((\lambda_1, \lambda_2), \mu)) \in \overrightarrow{\Lambda}' \right),$ put $\lambda = (\lambda_1, \lambda_2).$

The left cell module $W \left[(s, (s_1, s_2)), ((\lambda_1, \lambda_2), \mu) \right]\left(\overrightarrow{W} \left[(s, (s_1, s_2)), ((\lambda_1, \lambda_2), \mu) \right] \right)$ for the cellular algebra $\mathscr{A} \left[ A_k^{\mathbb{Z}_2}\right]\left(\mathscr{A} \left[ \overrightarrow{A}_k^{\mathbb{Z}_2}\right] \right)$ is defined as follows:

\begin{enumerate}
  \item[(i)] $W \left[(s, (s_1, s_2)), ((\lambda_1, \lambda_2), \mu) \right]$ is a free $\mathscr{A}$-module with basis

\centerline{$\left\{C^{m^{\lambda}_{s_{\lambda}} m^{\mu}_{s_{\mu}}}_{S} \ \Big| \ S = (d, P) \in M^k \left[(s, (s_1, s_2))\right]\right\}$}
and $A_k^{\mathbb{Z}_2}$-action is defined on the basis element  by $a$

$a C^{m^{\lambda}_{s_{\lambda}} m^{\mu}_{s_{\mu}}}_{S} \equiv \underset{ (S', s') \in M'^k \Big[ \big(s, (
s_1, s_2) , ((\lambda_1, \lambda_2), \mu)\big) \Big]}\sum  C^{r_a(S', S) m^{\lambda}_{s'_{\lambda}} m^{\mu}_{s'_{\mu}}}_{S'}  $

$ \hspace{8cm} \text{ mod } A_k^{\mathbb{Z}_2}
\Big( < \big(s, (s_1, s_2), ((\lambda_1, \lambda_2), \mu)
\big)\Big),$

\NI where $(S, w) = ((d, P), ((s_{\lambda_1}, s_{\lambda_2}), s_{\mu})), (S', s') = ((d', P'), ((s'_{\lambda_1}, s'_{\lambda_2}), s'_{\mu}))$  $r_a(S', S)$ is as in 3(a)(i) and (b)(i) of Theorem 5.4.

  \item[(ii)] \NI $\overrightarrow{W} \left[(s, (s_1, s_2)), ((\lambda_1, \lambda_2), \mu) \right]$ is a free $\mathscr{A}$-module with basis

\centerline{$\left\{ \overrightarrow{C}^{m^{\lambda}_{s_{\lambda}} m^{\mu}_{s_{\mu}}}_{\overrightarrow{S}} \ \Big| \ \overrightarrow{S} = (d, P) \in \overrightarrow{M}^k \left[(s, (s_1, s_2)) \right]\right\}$}
and $\overrightarrow{A}_k^{\mathbb{Z}_2}$-action is defined on the basis element  by $\overrightarrow{a}$

$\overrightarrow{a} \overrightarrow{C}^{m^{\lambda}_{s_{\lambda}} m^{\mu}_{s_{\mu}}}_{S} \equiv \underset{ (S', s') \in \overrightarrow{M}'^k \Big[ \big(s, (
s_1, s_2) , ((\lambda_1, \lambda_2), \mu)\big) \Big]}\sum  \overrightarrow{C}^{r_{\overrightarrow{a}}(S', S) m^{\lambda}_{s'_{\lambda}} m^{\mu}_{s'_{\mu}}}_{S'}  $

$\hspace{8cm}  \text{ mod } \overrightarrow{A}_k^{\mathbb{Z}_2}
\Big( < \big(s, (s_1, s_2), ((\lambda_1, \lambda_2), \mu)
\big)\Big),$

\NI where $(S, w) = ((d, P), ((s_{\lambda_1}, s_{\lambda_2}), s_{\mu})), (S', s') = ((d', P'), ((s'_{\lambda_1}, s'_{\lambda_2}), s'_{\mu}))$  $r_a(S', S)$ is as in 3(a)(ii) and (b)(ii) of Theorem 5.4.

\end{enumerate}

\end{defn}

\begin{lem}\cite{K1}\label{L4.2}
\begin{itemize}
  \item[(i)]  $C^{m^{\lambda}_{s_{\lambda}, s_{\lambda}} m^{\mu}_{s_{\mu}, s_{\mu}}}_{S, S} \ C^{m^{\lambda}_{t_{\lambda}, t_{\lambda}} m^{\mu}_{t_{\mu}, t_{\mu}}}_{T, T} \ \equiv \ \Phi_1((S, s), (T, t)) \ C^{m^{\lambda}_{s_{\lambda}, t_{\lambda}} m^{\mu}_{s_{\mu}, t_{\mu}}}_{S, T} \ $

       $\hspace{9cm} \text{ mod } \left[ \tiny{ A_k^{\mathbb{Z}_2} <(s, (s_1, s_2), ((\lambda_1, \lambda_2), \mu)}\right]$

     \NI where

       $\begin{array}{llll}
      \Phi_1((S, s),  (T, t)) & = & x^{l(P \vee P')} \phi^{\lambda}_{\delta_1}(s_{\lambda}, t_{\lambda}) \ \phi^{\mu}_{\delta_2}(s_{\mu}, t_{\mu})& \text{ when conditions (a) and (b)}  \\
      & & & \text{ of Definition } \ref{D2.13} \text{ are satisfied}\\
      & = &  0 & \text{Otherwise}
      \end{array}$
  \item[(ii)]  $\overrightarrow{C}^{m^{\lambda}_{s_{\lambda}, s_{\lambda}} m^{\mu}_{s_{\mu}, s_{\mu}}}_{S, S} \ \overrightarrow{C}^{m^{\lambda}_{t_{\lambda}, t_{\lambda}} m^{\mu}_{t_{\mu}, t_{\mu}}}_{T, T} \ \equiv \ \overrightarrow{\Phi}_1((S, s), (T, t)) \ C^{m^{\lambda}_{s_{\lambda}, t_{\lambda}} m^{\mu}_{s_{\mu}, t_{\mu}}}_{S, T} $

       $\hspace{7cm}\text{ mod } \left[ \overrightarrow{A}_k^{\mathbb{Z}_2} <(s, (s_1, s_2), ((\lambda_1, \lambda_2), \mu)\right]$

     \NI where

       $\begin{array}{llll}
      \overrightarrow{\Phi}_1((S, s),  (T, t)) & = & x^{l(P \vee P')} \phi^{\lambda}_{\delta_1}(s_{\lambda}, t_{\lambda}) \ \phi^{\mu}_{\delta_2}(s_{\mu}, t_{\mu})& \text{ when conditions (a) and (b)}  \\
      & & & \text{ of Definition }  \ref{D2.13} \text{ are satisfied}\\
      & = &  0 & \text{Otherwise}
      \end{array}$

\end{itemize}

        \NI  $(S, s) = ((d, P), ((s_{\lambda_1}, s_{\lambda_2}), s_{\mu})),  (T, t) = ((d', P'), ((t_{\lambda_1}, t_{\lambda_2}), t_{\mu})),  l(P \vee P')\left( l( P \vee P')\right) $ denotes the number of connected components in $ d'.d''$  excluding the union of all the connected components of $ P \text{ and } P' $,

          \NI $ m^{\lambda}_{s_{\lambda}, s_{\lambda}} \delta_{1} m^{\lambda}_{t_{\lambda}, t_{\lambda}} \equiv \phi^{\lambda}_{\delta_1}(s_{\lambda}, t_{\lambda}) m^{\lambda}_{s'_{\lambda}, t_{\lambda}}  \text{mod } \mathscr{H} \left(< (\lambda_1, \lambda_2)\right), m^{\mu}_{s_{\mu}, s_{\mu}} \delta_2 m^{\mu}_{t_{\mu}, t_{\mu}} \equiv \phi^{\mu}_{\delta_2}(s_{\mu}, t_{\mu})m^{\mu}_{s'_{\mu}, t_{\mu}}  \text{mod } \mathscr{H}' \left( < \mu \right)$

         \NI $\text{as in Lemma 1.7 [1]}.$

\end{lem}

\begin{defn}\cite{K1} \label{D4.3}
For $\big(s, (s_1, s_2), ((\lambda_1, \lambda_2), \mu)\big) \in
\Lambda'\left(\big(s, (s_1, s_2), ((\lambda_1, \lambda_2), \mu)\big) \in
\overrightarrow{\Lambda}' \right),$  the bilinear map $\phi_{s_1, s_2}^{\lambda, \mu} \left( \overrightarrow{\phi}_{s_1, s_2}^{\lambda, \mu}\right)$ is
defined as
\begin{enumerate}
  \item[(i)] $ \phi_{s_1, s_2}^{\lambda, \mu} \big(  C^{m^{\lambda}_{s_{\lambda}, s_{\lambda}} m^{\mu}_{s_{\mu}, s_{\mu}}}_{(d, P)},   C^{m^{\lambda}_{t_{\lambda}, t_{\lambda}} m^{\mu}_{t_{\mu}, t_{\mu}}}_{(d', P')} \big) = \Phi_1((S, s), (T, t)), \ \ (S, s), (T, t) \in
M'^k \big[s, (s_1, s_2), ((\lambda_1, \lambda_2), \mu) \big] $
  \item[(ii)] $ \overrightarrow{\phi}_{s_1, s_2}^{\lambda, \mu} \big(  \overrightarrow{C}^{m^{\lambda}_{s_{\lambda}, s_{\lambda}} m^{\mu}_{s_{\mu}, s_{\mu}}}_{(d, P)},   \overrightarrow{C}^{m^{\lambda}_{t_{\lambda}, t_{\lambda}} m^{\mu}_{t_{\mu}, t_{\mu}}}_{(d', P')} \big) = \Phi_1((S, s), (T, t)), \ \ (S, s), (T, t) \in
\overrightarrow{M}'^k \big[s, (s_1, s_2), ((\lambda_1, \lambda_2), \mu) \big] $
\end{enumerate}

\NI where $\Phi_1((S, s), (T, t))\left(\overrightarrow{\Phi}_1((S, s), (T, t)) \right)$ is as in Lemma \ref{L4.2}.

Put
\begin{enumerate}
  \item[(i)] $G^{\lambda, \mu}_{2s_1+s_2} = \left(\Phi_1((S, s), (T, t)) \right)_{(S, s), (T, t) \in M'^k\big[s, (s_1, s_2), ((\lambda_1, \lambda_2), \mu) \big]}$

      where

 $\begin{array}{llll}
      \Phi_1((S, s), (T, t)) & = & x^{l(P_i \vee P_j)} \phi^{\lambda}_{\delta_1}(s_{\lambda}, t_{\lambda}) \ \phi^{\mu}_{\delta_2}(s_{\mu}, t_{\mu})& \text{ when conditions (a) and (b)}  \\
      & & & \text{ of Definition }  \ref{D2.13} \text{ are satisfied}\\
       & = &  0 & \text{Otherwise}
      \end{array}$

\NI where $(S, s) = ((d_i, P_i), ((s_{\lambda_1}, s_{\lambda_2}), s_{\mu})), (T, t) = ((d_j, P_j), ((t_{\lambda_1}, t_{\lambda_2}), t_{\mu}))$

  \item[(ii)]$\overrightarrow{G}^{\lambda, \mu}_{2s_1+s_2} = \left(\overrightarrow{\Phi}_1((S, s), (T, t)) \right)_{(S, s), (T, t) \in \overrightarrow{M}'^k\big[s, (s_1, s_2), ((\lambda_1, \lambda_2), \mu) \big]}$

      where

 $\begin{array}{llll}
      \overrightarrow{\Phi}_1((S, s), (T, t)) & = & x^{l(P_i \vee P_j)} \phi^{\lambda}_{\delta_1}(s_{\lambda}, t_{\lambda}) \ \phi^{\mu}_{\delta_2}(s_{\mu}, t_{\mu})& \text{ when conditions (a) and (b)}  \\
      & & & \text{ of Definition }  \ref{D2.13} \text{ are satisfied}\\
       & = &  0 & \text{Otherwise}
      \end{array}$

\NI where $(S, s) = ((d_i, P_i), ((s_{\lambda_1}, s_{\lambda_2}), s_{\mu})), (T, t) = ((d_j, P_j), ((t_{\lambda_1}, t_{\lambda_2}), t_{\mu})),$
\end{enumerate}

        \NI $l(P_i \vee P_j) $ denotes the number of  connected components in $ d'.d''$  excluding the union of all the connected components of $ P_i \text{ and } P_j $, $ m^{\lambda}_{s_{\lambda}, s_{\lambda}} \delta_{1} m^{\lambda}_{t_{\lambda}, t_{\lambda}} \equiv \phi^{\lambda}_{\delta_1}(s_{\lambda}, t_{\lambda}) m^{\lambda}_{s'_{\lambda}, t_{\lambda}}  \text{mod } \mathscr{H} \left(< (\lambda_1, \lambda_2)\right),$

        $m^{\mu}_{s_{\mu}, s_{\mu}} \delta_2 m^{\mu}_{t_{\mu}, t_{\mu}} \equiv \phi^{\mu}_{\delta_2}(s_{\mu}, t_{\mu})m^{\mu}_{s'_{\mu}, t_{\mu}}  \text{mod } \mathscr{H}' \left( < \mu \right) \text{as in Lemma 1.7 [1]}.$

    \NI  $G^{\lambda, \mu}_{2s_1+s_2}\left( \overrightarrow{G}^{\lambda, \mu}_{2s_1+s_2}\right)$ is called the \textbf{Gram matrix of the cell module} $W \left[(s, (s_1, s_2)), ((\lambda_1, \lambda_2), \mu)\right] \\ \left( \overrightarrow{W} \left[(s, (s_1, s_2)), ((\lambda_1, \lambda_2), \mu)\right] \right).$
\end{defn}

\begin{defn}\label{D4.4}

Let $\left\{C_{S_{i, \alpha}^{r_1, r_2}}^{m^{\lambda}_{s_{\lambda}} m^{\mu}_{s_{\mu}}} \right\}_{(S_{i, \alpha}^{r_1, r_2}, t_l) \in M'^k[(s, (s_1, s_2)), ((\lambda_1, \lambda_2), \mu)]} \left\{\left\{\overrightarrow{C}_{S_{i, \alpha}^{r_1, r_2}}^{m^{\lambda}_{s_{\lambda}} m^{\mu}_{s_{\mu}}} \right\}_{(S_{i, \alpha}^{r_1, r_2}, t_l) \in \overrightarrow{M}'k[(s, (s_1, s_2)), ((\lambda_1, \lambda_2), \mu)]}\right\}$

\NI be the basis of the cell module $W[(s, (s_1, s_2)), ((\lambda_1, \lambda_2), \mu)] \left(\overrightarrow{W}[(s, (s_1, s_2)), ((\lambda_1, \lambda_2), \mu)]\right)$,  where $S_{i, \alpha}^{r_1, r_2} = (d_i, P_i), t_l = ((t^l_{\lambda_1}, t^l_{\lambda_2}), t^l_{\mu})$ .

\NI Now, we shall introduce the ordering on the basis of the cell module $W[(s, (s_1, s_2)), ((\lambda_1, \lambda_2), \mu)]$ as follows:

$(S_{i, \alpha}^{r_1, r_2}, t_l) < (S_{j, \beta}^{r'_1, r'_2}, t_k) $

\begin{enumerate}
  \item[(i)] if $(i, \alpha, r_1, r_2) < (j, \beta, r'_1, r'_2)$ as in Definition \ref{D3.7} and
  \item[(ii)] if $(i, \alpha, r_1, r_2) = (j, \beta, r'_1, r'_2)$ then $(S_{i, \alpha}^{r_1, r_2}, t_l),  (S_{j, \beta}^{r'_1, r'_2}, t_k)$  can be indexed arbitrarily.
 \end{enumerate}

The above ordering can be used for the basis of the cell module $\overrightarrow{W}[(s, (s_1, s_2)), ((\lambda_1, \lambda_2), \mu)]$

Arrange the basis of the cell module $W[(s, (s_1, s_2)), ((\lambda_1, \lambda_2), \mu)]$ and $\overrightarrow{W}[(s, (s_1, s_2)), ((\lambda_1, \lambda_2), \mu)]$ according to the order defined above and obtain the Gram matrix $G^{\lambda, \mu}_{2s_1+s_2}$ and $\overrightarrow{G}^{\lambda, \mu}_{2s_1+s_2}$ corresponding to the cell modules $W[(s, (s_1, s_2)), ((\lambda_1, \lambda_2), \mu)]$ and $\overrightarrow{W}[(s, (s_1, s_2)), ((\lambda_1, \lambda_2), \mu)]$ respectively.

\end{defn}

\begin{thm}\label{T4.5}
\mbox{ }
\begin{enumerate}
\item [(i)] The algebra of $\mathbb{Z}_2$-relations $A_k^{\mathbb{Z}_2}(x)$, signed partition algebras $\overrightarrow{A}_k^{\mathbb{Z}_2}(x)$ and partition algebras $A_k(x) $ are semisimple over $\mathbb{K}(x)$ where $\mathbb{K}$ is a field of characteristic zero where $x$ is an indeterminate.
    \item [(ii)]Suppose that the characteristic of the field $\mathbb{K}$ is 0, then
    \begin{enumerate}
                  \item[(a)] the algebra of $\mathbb{Z}_2$-relations $A_k^{\mathbb{Z}_2}(q)$ is semisimple if and only if $q$ is not a root of the polynomial $f(x)$ where $f(x) = \underset{\lambda, \mu} \prod \underset{2s_1+s_2 = 0}{\overset{2k}{\prod}} \det G^{\lambda, \mu}_{2s_1+s_2}$ where $x=q$ and $q \in \mathbb{C}.$
                  \item[(b)]  the signed
partition algebra $\overrightarrow{A}_k^{\mathbb{Z}_2}(q)$ is
semisimple if and only if $q$ is not a root of the polynomial
$f(x)$ where $f(x) = \underset{\lambda, \mu} \prod \underset{2s_1+s_2 = 0}{\overset{2k}{\prod}} \det \overrightarrow{G}^{\lambda, \mu}_{2s_1+s_2}.$
                  \item[(c)] the partition algebra $A_k(x)$ is semisimple if and only if $q$ is not a root of the polynomial $f(x)$ where $f(x) = \underset{\lambda} \prod \underset{s = 0}{\overset{k}{\prod}} \det G^{\lambda}_{s}.$
                \end{enumerate}
     \item[(iii)] \begin{enumerate}
                    \item[(a)]      In particular,  $G^{\lambda, \mu}_{2s_1 +s_2}$  coincides with $G_{2s_1+s_2}^k$ if
\begin{enumerate}
  \item[1.] $\lambda = ([s_1], \Phi)$ and $\mu = [s_2]$ when $s_1, s_2 \neq 0,$
  \item[2.] $\lambda = (\Phi, \Phi)$ and $\mu = [s_2]$ when $s_1 = 0, s_2 \neq 0,$
  \item[3.] $\lambda = ([s_1], \Phi)$ and $\mu = \Phi$ when $s_1 \neq 0, s_2 = 0$
  \item[4.] $\lambda = (\Phi, \Phi)$ and $\mu = \Phi$ when $s_1, s_2 = 0,$
\end{enumerate}
for $0 \leq s_1 \leq k, 0 \leq s_2 \leq k, 0 \leq s_1+s_2 \leq k.$
                    \item[(b)]      In particular,  $\overrightarrow{G}^{\lambda, \mu}_{2s_1 +s_2}$  coincides with $\overrightarrow{G}_{2s_1+s_2}^k$ if
\begin{enumerate}
  \item[1.] $\lambda = ([s_1], \Phi)$ and $\mu = [s_2]$ when $s_1, s_2 \neq 0,$
  \item[2.] $\lambda = (\Phi, \Phi)$ and $\mu = [s_2]$ when $s_1 = 0, s_2 \neq 0,$
  \item[3.] $\lambda = ([s_1], \Phi)$ and $\mu = \Phi$ when $s_1 \neq 0, s_2 = 0$
  \item[4.] $\lambda = (\Phi, \Phi)$ and $\mu = \Phi$ when $s_1, s_2 = 0,$
\end{enumerate}
for $0 \leq s_1 \leq k-1, 0 \leq s_2 \leq k-1, 0 \leq s_1+s_2 \leq k-1.$

                    \item[(c)]  $G^{\lambda}_{s}$  coincides with $G_{s}^k$ if
\begin{enumerate}
  \item[1.] $\lambda = s$  when $s \neq 0,$
  \item[2.] $\lambda = \Phi$ when $s = 0$
 \end{enumerate}
for $0 \leq s \leq k.$
                  \end{enumerate}
\item[(iii)$'$]\begin{enumerate}
               \item[(a)] If $q$ is a root of the polynomial  $$f(x) =  \prod_{2s_1+s_2 = 0}^{2k} \det G_{2s_1+s_2}^k$$ where $\det G_{2s_1+s_2}^k = \underset{\substack{\ds 0 \leq r_1 \leq k-s_1-s_2 \\ \ds 0 \leq r_2 \leq k-s_1-s_2 \\ \ds 2r_1+r_2 \leq 2k-2s_1-2s_2}} \prod \det \widetilde{A}_{2r_1+r_2, 2r_1+r_2} $, $\widetilde{A}_{2r_1+r_2, 2r_1+r_2}$ is as in Theorem \ref{T3.32} then the algebra $A^{\mathbb{Z}_2}_k(q)$ is not semisimple.

 \NI   In particular, $q$ is an integer such that $ 0 \leq q \leq k-2$ and $q$ is a root of the polynomial $x^2-x-2r', 0 \leq r' \leq k-2$ then $A^{\mathbb{Z}_2}_k(q)$ is not semisimple.
               \item[(b)] If $q$ is a root of the polynomial  $$f(x) =  \prod_{2s_1+s_2 = 0}^{2k} \det \overrightarrow{G}_{2s_1+s_2}^k$$ where $\det \overrightarrow{G}_{2s_1+s_2}^k = \underset{\substack{\ds 0 \leq r_1 \leq k-s_1-s_2-1 \\ \ds 0 \leq r_2 \leq k-s_1-s_2-1 \\ \ds 2r_1+r_2 \leq 2k-2s_1-2s_2-1}} \prod \det \widetilde{\overrightarrow{A}}_{2r_1+r_2, 2r_1+r_2} \prod \det \widetilde{\overrightarrow{A}}_{\rho}, $

                   $\widetilde{\overrightarrow{A}}_{2r_1+r_2, 2r_1+r_2}$ and $\widetilde{\overrightarrow{A}}_{\rho}$ are as in Theorem \ref{T3.32} then the algebra $\overrightarrow{A}^{\mathbb{Z}_2}_k(q)$ is not semisimple.

 \NI   In particular, $q$ is an integer such that $ 0 \leq q \leq k-2$ and $q$ is a root of the polynomial $x^2-x-2r', 0 \leq r' \leq k-2$ then $\overrightarrow{A}^{\mathbb{Z}_2}_k(q)$ is not semisimple.

               \item[(c)] If $q$ is a root of the polynomial  $$f(x) =  \prod_{s = 0}^{k} \det G_{s}^k$$ where $\det G_{s}^k = \underset{0 \leq r \leq k-s } \prod \det \widetilde{A}_{r, r} $, $\widetilde{A}_{r, r}$ is as in Theorem \ref{T3.32} then the algebra $A_k(q)$ is not semisimple.


             \end{enumerate}
\item [(iv)] The algebra of $\mathbb{Z}_2$-relations$\left(A_k^{\mathbb{Z}_2}(q)\right)$, signed partition algebra$\left(\overrightarrow{A}_k^{\mathbb{Z}_2}(q)\right)$ and the partition algebra$(A_k(q))$ over a field of characteristics $0$ are quasi-hereditary for $q \neq 0.$
\end{enumerate}
\end{thm}
\begin{proof}

\NI \textbf{Proof of (i):}

The matrix  of the bilinear form associated to the cell module $\overrightarrow{W}[(s, (s_1, s_2)), ((\lambda_1, \lambda_2), \mu)]$ as defined in   Definition 4.3(ii) with respect to the ordering of the basis as in Definition \ref{D4.4} is rewritten as follows:

\centerline{$\overrightarrow{G}^{\lambda, \mu}_{2s_1+s_2} = \left( g_{(i, \alpha, r_1, r_2), (j, \beta, r'_1, r'_2)}\right)_{1 \leq (i, \alpha, r_1, r_2), (j, \beta, r'_1, r'_2) \leq \overrightarrow{f}_{2s_1+s_2}}$}

\NI where $g_{(i, \alpha, r_1, r_2), (j, \beta, r'_1, r'_2)} = a_{(i, \alpha, r_1, r_2), (j, \beta, r'_1, r'_2)} B^{\lambda, \mu}_{\delta_1, \delta_2}, $

 $\begin{array}{lll}
     a_{(i, \alpha, r_1, r_2), (j, \beta, r'_1, r'_2)} & = & \left\{
    \begin{array}{ll}
     x^{l(P_i \vee P_j)} , & \hbox{if conditions (a) and (b) of Definition \ref{D2.13} are satisfied;} \\
      0, & \hbox{Otherwise.}
    \end{array}
  \right.
  \\
    B^{\lambda, \mu}_{\delta_1, \delta_2}& = &  B^{\lambda}_{\delta_1} \otimes B^{\mu}_{\delta_2}
  \end{array}$

\NI with $B^{\lambda}_{\delta_1} = \left(\phi^{\lambda}_{\delta_1}(s_{\lambda}, t_{\lambda})\right)$ and  $B^{\mu}_{\delta_2} = \left( \phi^{\mu}_{\delta_2}(s_{\mu}, t_{\mu})\right),$  $B^{\lambda}_{\delta_1}$ and $B^{\mu}_{\delta_2}$ are the matrices of the non-degenerate bilinear forms associated to the cell module $W^{\lambda}$ and $W^{\mu}$ of the cellular algebras of $K[\mathbb{Z}_2 \wr \mathfrak{S}_{s_1}]$ and $K[\mathfrak{S}_{s_2}]$ respectively as in Theorem 3.8 in [1] and $\delta_1$ and $\delta_2$ depends on the product of the diagrams $d^{r_1, r_2}_{i, \alpha}$ and $d^{r'_1, r'_2}_{j, \beta}.$

\begin{eqnarray*}
  \overrightarrow{G}^{\lambda, \mu}_{2s_1+s_2} &=& \left(a_{(i, \alpha, r_1, r_2), (j, \beta, r'_1, r'_2)} B^{\lambda, \mu}_{\delta_1, \delta_2} \right)_{1 \leq (i, \alpha, r_1, r_2), (j, \beta, r'_1, r'_2) \leq \overrightarrow{f}_{2s_1+s_2}}
\end{eqnarray*}

 The $g_{(i, \alpha, r_1, r_2), (i,\alpha, r_1, r_2)} = a_{(i, \alpha, r_1, r_2), (i, \alpha, r_1, r_2)} \ A$ where $A = B^{\lambda, \mu}_{1, 1} =B^{\lambda}_{1} \otimes B^{\mu}_{1}.$

\NI Thus, the leading coefficient of the Gram matrix is $\left( \det A \right)^{\overrightarrow{f}_{2s_1+s_2} \times \text{dim } \overrightarrow{W}[(s, (s_1, s_2)), ((\lambda_1, \lambda_2), \mu)]}$ which is non-zero over a characteristic zero.

\NI Therefore, the algebra $\overrightarrow{A}_k^{\mathbb{Z}_2}(x)$ is semisimple. The proof for the algebra of $\mathbb{Z}_2$-relations and the partition algebras are similar to the above proof.

\NI \textbf{Proof of (ii):}

 \NI By Theorem 3.8 in \cite{GL}, $\overrightarrow{A}_k^{\mathbb{Z}_2}$ is semisimple if and only if $\det \ G^{\lambda, \mu}_{2s_1+s_2} \neq 0$ for all $s_1, s_2$ and for all $\lambda, \mu,$ since

 $\det \ G^{\lambda, \mu}_{2s_1+s_2} \neq 0$ if and only if $\Phi$ is non-degenerate.

\NI \textbf{Proof of (iii)(b):}

Now, $\overrightarrow{G}^{\lambda, \mu}_{2s_1+s_2} = \overrightarrow{G}_{2s_1+s_2}^k, $ if
\begin{enumerate}
  \item $\lambda = ([s_1], \Phi)$ and $\mu = [s_2]$ when $s_1, s_2 \neq 0,$
  \item $\lambda = (\Phi, \Phi)$ and $\mu = [s_2]$ when $s_1 = 0, s_2 \neq 0,$
  \item $\lambda = ([s_1], \Phi)$ and $\mu = \Phi$ when $s_1 \neq 0, s_2 = 0$
\end{enumerate}
for $0 \leq s_1 \leq k-1, 0 \leq s_2 \leq k-1, 0 \leq s_1+s_2 \leq k-1, $since $A$ is the $1 \times 1$ identity matrix,

if $\lambda = (\Phi, \Phi)$ and $\mu = \Phi$ when $s_1, s_2 = 0,$ then $\overrightarrow{G}^{\Phi, \Phi}_{2s_1+s_2}$ coincides with $\overrightarrow{G}^k_{0+0}.$

\NI \textbf{Proof of (iii)(b):} If $q$ is a root of $f(x) = \underset{\substack{\ds 0 \leq r_1 \leq k-s_1-s_2-1 \\ \ds 0 \leq r_2 \leq k-s_1-s_2-1 \\ \ds 2r_1+r_2 \leq 2k-2s_1-2s_2-1}} \prod \det \widetilde{\overrightarrow{A}}_{2r_1+r_2, 2r_1+r_2} \prod \det \widetilde{\overrightarrow{A}}_{\rho}.$

Then $\det \overrightarrow{G}^k_{2s_1 +s_2} = 0 = \det \overrightarrow{G}^{(([s_1], \Phi), [s_2])}_{2s_1+s_2}.$

Thus, the algebra $\overrightarrow{A}_k^{\mathbb{Z}_2}$ is not semisimple.

In particular, by Remark \ref{R3.33} if $q$ is an integer such that $ 0 \leq q \leq k - 2$ and $q$ is a root of polynomial $x^2 - x - 2r', 0 \leq r' \leq k - 2$ then the algebra  $\overrightarrow{A}_k^{\mathbb{Z}_2}$ is not semisimple.

Proof of (a) and (c) are similar to the proof of (b).

\NI \textbf{Proof of (iv):} It follows from Remark 3.10 in \cite{GL} and Theorem 5.4 in \cite{K1}.

\end{proof}

\begin{center}
\textbf{Appendix}
\end{center}

The following is an example of Gram matrix in $\overrightarrow{A}_3^{\mathbb{Z}_2}(x).$

Let $s_1 = 1$ and $s_2 = 0.$ The following  are the diagrams in $J^6_{2 \times 1 + 0}.$

\begin{center}
\includegraphics{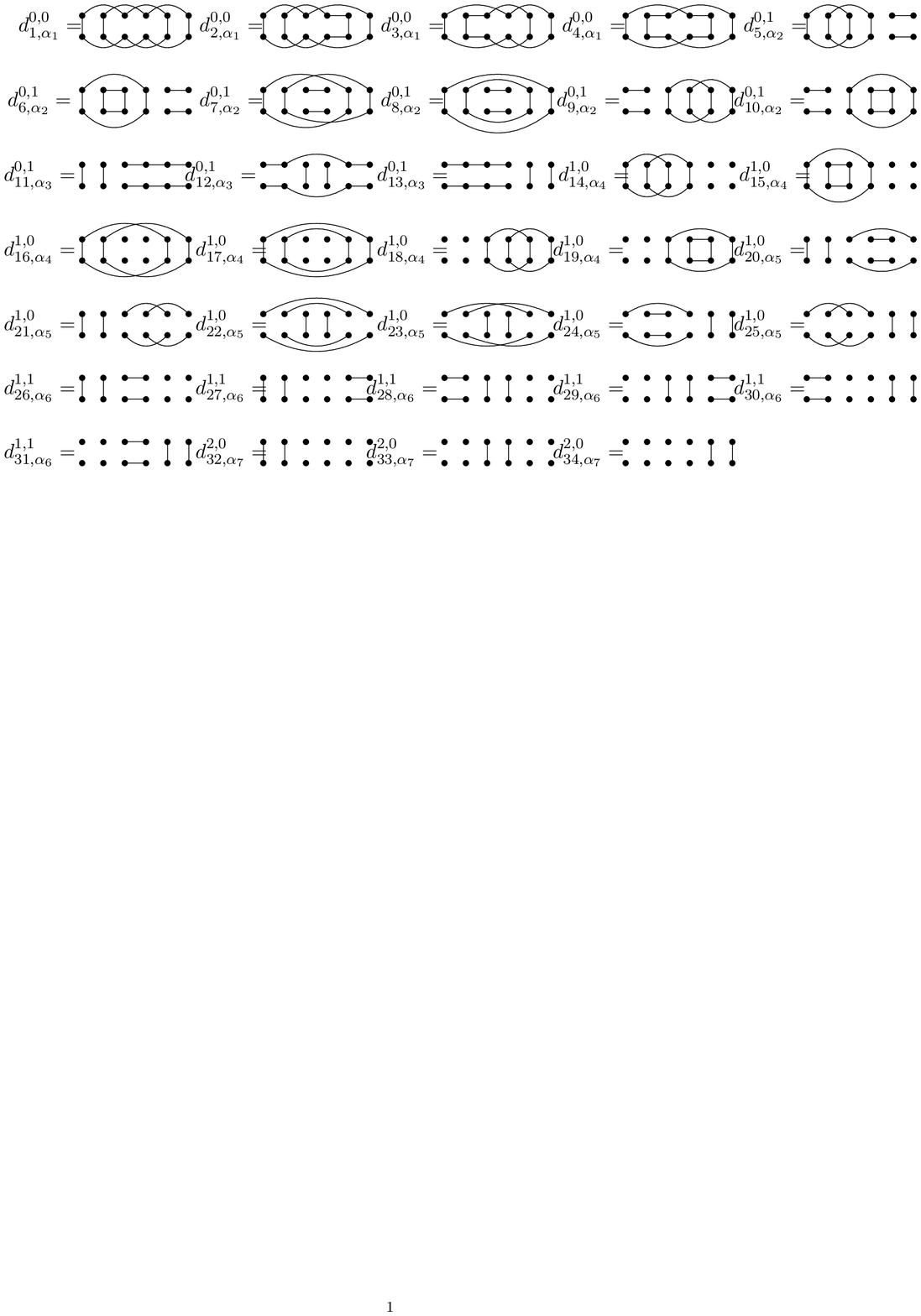}
\end{center}

\NI where $\alpha_1 = (3, \Phi, \Phi, \Phi), \alpha_2 = (2, \Phi, \Phi, 1), \alpha_3 = (1, \Phi, \Phi, 2), \alpha_4 = (2, \Phi, 1, \Phi), \alpha_5 = (1, \Phi, 0, 2), \alpha_6 = (1, \Phi, 1, 1)$ , $\alpha_7 = (1, 0, 1^2, 0)$ and $d_{i, \alpha}^{r_1, r_2}$ is a diagram having $r_1$ number of pairs of $\{e\}$-horizontal edges, $r_2$ number of $\mathbb{Z}_2$-horizontal edges and $\alpha$ is the underlying partition of $d^{r_1, r_2}_{i, \alpha}.$

\begin{landscape}
\thispagestyle{empty}
\scalebox{0.6}{
\begin{tabular}{|c|c|c|c|c||c|c|c|c|c|c|c|c|c||c|c|c|c|c|c|c|c|c|c|c|c||c|c|c|c|c|c||c|c|c|}
  \hline

  & $d_{1, \alpha_1}^{0, 0}$ & $d_{2, \alpha_1}^{0,0}$ & $d_{3, \alpha_1}^{0,0}$ & $d_{4, \alpha_1}^{0,0}$ & $d_{5, \alpha_2}^{0,1}$ & $d_{6, \alpha_2}^{0,1}$ & $d_{7, \alpha_2}^{0,1}$ & $d_{8, \alpha_2}^{0,1}$ & $d_{9, \alpha_2}^{0,1}$ & $d_{10, \alpha_2}^{0,1}$ & $d_{11, \alpha_3}^{0,1}$ & $d_{12, \alpha_3}^{0,1}$ & $d_{13, \alpha_3}^{0,1}$ & $d_{14, \alpha_4}^{1,0}$ & $d_{15, \alpha_4}^{1,0}$ & $d_{16, \alpha_4}^{1,0}$ & $d_{17, \alpha_4}^{1,0}$ & $d_{18, \alpha_4}^{1,0}$ & $d_{19, \alpha_4}^{1,0}$ & $d_{20, \alpha_5}^{1,0}$ & $d_{21, \alpha_5}^{1,0}$ & $d_{22, \alpha_5}^{1,0}$ & $d_{23, \alpha_5}^{1,0}$ & $d_{24, \alpha_5}^{1,0}$ & $d_{25, \alpha_5}^{1,0}$ & $d_{26, \alpha_6}^{1,1}$ & $d_{27, \alpha_6}^{1,1}$ & $d_{28, \alpha_6}^{1,1}$ & $d_{29, \alpha_6}^{1,1}$ & $d^{1,1}_{30, \alpha_6}$ & $d_{31, \alpha_6}^{1,1}$& $d_{32, \alpha_7}^{2,0}$ & $d_{33, \alpha_7}^{2,0}$ & $d_{34, \alpha_7}^{2,0}$   \\  \hline
  $d_{1, \alpha_1}^{0, 0}$ & 1 & 0 & 0 & 0 & 0 & 0 & 0 & 0 & 0 & 0& 0 & 0 & 0 & 1 & 0 & 1 & 0 & 1 & 0 & 1 & 0 & 1 & 0 & 1 & 0 & 0 & 0 & 0 & 0 & 0 & 0 & 1 & 1 & 1  \\  \hline
   $d_{2, \alpha_1}^{0, 0}$& 0 & 1 & 0 & 0 & 0 & 0 & 0 & 0 & 0 & 0& 0 & 0 & 0 & 1 & 0 & 0 & 1 & 0 & 1 & 0 & 1 & 0 & 1 & 1 & 0 & 0 & 0 & 0 & 0 & 0 & 0 & 1 & 1 & 1  \\  \hline
   $d_{3, \alpha_1}^{0, 0}$& 0 & 0 & 1 & 0 & 0 & 0 & 0 & 0 & 0 & 0& 0 & 0 & 0 & 0 & 1 & 0 & 1 & 1 & 0 & 1 & 0 & 0 & 1 & 0 & 1 & 0 & 0 & 0 & 0 & 0 & 0 & 1 & 1 & 1  \\  \hline
  $d_{4, \alpha_1}^{0, 0}$& 0 & 0 & 0 & 1 & 0 & 0 & 0 & 0 & 0 & 0& 0 & 0 & 0 & 0 & 1 & 1 & 0 & 0 & 1 & 0 & 1 & 1 & 0 & 0 & 1 & 0 & 0 & 0 & 0 & 0 & 0 & 1 & 1 & 1  \\  \hline \hline
  $d_{5, \alpha_2}^{0, 1}$& 0 & 0 & 0 & 0 & $x$ & 0 & 0 & 0 & 0 & 0& 0 & 0 & 0 & $x$ & 0 & 0 & 0 & 0 & 0 & 0 & 0 & 0 & 0 & 0 & 0 & 0 & $x$ & 0 & $x$ & 0 & 0 & $x$ & $x$ & 0  \\  \hline
  $d_{6, \alpha_2}^{0, 1}$& 0 & 0 & 0 & 0 & 0 & $x$ & 0 & 0 & 0 & 0& 0 & 0 & 0 & 0 & $x$ & 0 & 0 & 0 & 0 & 0 & 0 & 0 & 0 & 0 & 0 & 0 & $x$ & 0 & $x$ & 0 & 0 & $x$ & $x$ & 0  \\  \hline
  $d_{7, \alpha_2}^{0, 1}$& 0 & 0 & 0 & 0 & 0 & 0 & $x$ & 0 & 0 & 0& 0 & 0 & 0 & 0 & 0 & $x$ & 0 & 0 & 0 & 0 & 0 & 0 & 0 & 0 & 0 & $x$ & 0 & 0 & 0 & 0 & $x$ & $x$ & 0 & $x$  \\  \hline
  $d_{8, \alpha_2}^{0, 1}$& 0 & 0 & 0 & 0 & 0 & 0 & 0 & $x$ & 0 & 0& 0 & 0 & 0 & 0 & 0 & 0 & $x$ & 0 & 0 & 0 & 0 & 0 & 0 & 0 & 0 & $x$ & 0 & 0 & 0 & 0 & $x$ & $x$ & 0 & $x$  \\  \hline
 $d_{9, \alpha_2}^{0, 1}$&  0 & 0 & 0 & 0 & 0 & 0 & 0 & 0 & $x$ & 0& 0 & 0 & 0 & 0 & 0 & 0 & 0 & $x$ & 0 & 0 & 0 & 0 & 0 & 0 & 0 & 0 & 0 & $x$ & 0 & $x$ & 0 & 0 & $x$ & $x$  \\  \hline
 $d_{10, \alpha_2}^{0, 1}$&  0 & 0 & 0 & 0 & 0 & 0 & 0 & 0 & 0 & $x$& 0 & 0 & 0 & 0 & 0 & 0 & 0 & 0 & $x$ & 0 & 0 & 0 & 0 & 0 & 0 & 0 & 0 & $x$ & 0 & $x$ & 0 & 0 & $x$ & $x$  \\  \hline
 $d_{11, \alpha_3}^{0, 1}$&  0 & 0 & 0 & 0 & 0 & 0 & 0 & 0 & 0 & 0& $x$ & 0 & 0 & 0 & 0 & 0 & 0 & 0 & 0 & $x$ & $x$ & 0 & 0 & 0 & 0 & $x$ & $x$ & 0 & 0 & 0 & 0 & $x$ & 0 & 0  \\  \hline
 $d_{12, \alpha_3}^{0, 1}$&  0 & 0 & 0 & 0 & 0 & 0 & 0 & 0 & 0 & 0& 0 & $x$ & 0 & 0 & 0 & 0 & 0 & 0 & 0 & 0 & 0 & $x$ & $x$ & 0 & 0 & 0 & 0 & $x$ & $x$ & 0 & 0 & 0 & $x$ & 0  \\  \hline
  $d_{13, \alpha_3}^{0, 1}$& 0 & 0 & 0 & 0 & 0 & 0 & 0 & 0 & 0 & 0& 0 & 0 & $x$ & 0 & 0 & 0 & 0 & 0 & 0 & 0 & 0 & 0 & 0 & $x$ & $x$ & 0 & 0 & 0 & 0 & $x$ & $x$ & 0 & 0 & $x$  \\  \hline \hline
 $d_{14, \alpha_4}^{1,0}$&  1 & 1 & 0 & 0 & $x$ & 0 & 0 & 0 & 0 & 0& 0 & 0 & 0 & $x^2$ & 0 & 1 & 1 & 1 & 1 & 1 & 1 & 1 & 1 & 0 & 0 & 0 & $x$ & 0 & $x$ & $x$ & 0 & $x^2$ & $x^2$ & 0  \\  \hline
 $d_{15, \alpha_4}^{1, 0}$&  0 & 0 & 1 & 1 & 0 & $x$ & 0 & 0 & 0 & 0& 0 & 0 & 0 & 0 & $x^2$ & 1 & 1 & 1 & 1 & 1 & 1 & 1 & 1 & 0 & 0 & 0 & $x$ & 0 & $x$ & $x$ & 0 & $x^2$ & $x^2$ & 0  \\  \hline
  $d_{16, \alpha_4}^{1, 0}$& 1 & 0 & 0 & 1 & 0 & 0 & $x$ & 0 & 0 & 0& 0 & 0 & 0 & 1 & 1 & $x^2$ & 0 & 1 & 1 & 1 & 1 & 0 & 0 & 1 & 1 & $x$ & 0 & 0 & 0 & 0 & $x$ & $x^2$ & 0 & $x^2$  \\  \hline
  $d_{17, \alpha_4}^{1, 0}$& 0 & 1 & 1 & 0 & 0 & 0 & 0 & $x$ & 0 & 0& 0 & 0 & 0 & 1 & 1 & 0 & $x^2$ & 1 & 1 & 1 & 1 & 0 & 0 & 1 & 1 & $x$ & 0 & 0 & 0 & 0 & $x$ & $x^2$ & 0 & $x^2$  \\  \hline
 $d_{18, \alpha_4}^{1, 0}$&  1 & 0 & 1 & 0 & 0 & 0 & 0 & 0 & $x$ & 0& 0 & 0 & 0 & 1 & 1 & 1 & 1 & $x^2$ & 0 & 0 & 0 & 1 & 1 & 1 & 1 & 0 & 0 & $x$ & 0 & $x$ & 0 & 0 & $x^2$ & $x^2$  \\  \hline
 $d_{19, \alpha_4}^{1, 0}$&  0 & 1 & 0 & 1 & 0 & 0 & 0 & 0 & 0 & $x$& 0 & 0 & 0 & 1 & 1 & 1 & 1 & 0 & $x^2$ & 0 & 0 & 1 & 1 & 1 & 1 & 0 & 0 & $x$ & 0 & $x$ & 0 & 0 & $x^2$ & $x^2$  \\  \hline
$d_{20, \alpha_5}^{1, 0}$&   1 & 0 & 1 & 0 & 0 & 0 & 0 & 0 & 0 & 0& $x$ & 0 & 0 & 1 & 1 & 1 & 1 & 0 & 0 & $x^2$ & $x$ & 1 & 1 & 1 & 1 & $x$ & $x$ & 0 & 0 & 0 & 0 & $x^2$ & 0 & 0  \\  \hline
 $d_{21, \alpha_5}^{1, 0}$&  0 & 1 & 0 & 1 & 0 & 0 & 0 & 0 & 0 & 0& $x$ & 0 & 0 & 1 & 1 & 1 & 1 & 0 & 0 & $x$ & $x^2$ & 1 & 1 & 1 & 1 & $x$ & $x$ & 0 & 0 & 0 & 0 & $x^2$ & 0 & 0  \\  \hline
  $d_{22, \alpha_5}^{1, 0}$& 1 & 0 & 0 & 1 & 0 & 0 & 0 & 0 & 0 & 0& 0 & $x$ & 0 & 1 & 1 & 0 & 0 & 1 & 1 & 1 & 1 & $x^2$ & $x$ & 1 & 1 & 0 & 0 & $x$ & $x$ & 0 & 0 & 0 & $x^2$ & 0  \\  \hline
 $d_{23,\alpha_5}^{1, 0}$&  0 & 1 & 1 & 0 & 0 & 0 & 0 & 0 & 0 & 0& 0 & $x$ & 0 & 1 & 1 & 0 & 0 & 1 & 1 & 1 & 1 & $x$ & $x^2$ & 1 & 1 & 0 & 0 & $x$ & $x$ & 0 & 0 & 0 & $x^2$ & 0  \\  \hline
 $d_{24, \alpha_5}^{1, 0}$&  1 & 1 & 0 & 0 & 0 & 0 & 0 & 0 & 0 & 0& 0 & 0 & $x$ & 0 & 0 & 1 & 1 & 1 & 1 & 1 & 1 & 1 & 1 & $x^2$ & $x$ & 0 & 0 & 0 & 0 & $x$ & $x$ & 0 & 0 & $x^2$ \\  \hline
$d_{25, \alpha_5}^{1, 0}$&   0 & 0 & 1 & 1 & 0 & 0 & 0 & 0 & 0 & 0& 0 & 0 & $x$ & 0 & 0 & 1 & 1 & 1 & 1 & 1 & 1 & 1 & 1 & $x$ & $x^2$ & 0 & 0 & 0 & 0 & $x$ & $x$ & 0 & 0 & $x^2$ \\  \hline \hline
 $d_{26, \alpha_6}^{1, 1}$&  0 & 0 & 0 & 0 & 0 & 0 & $x$ & $x$ & 0 & 0& $x$ & 0 & 0 & 0 & 0 & $x$ & $x$ & 0 & 0 & $x$ & $x$ & 0 & 0 & 0 & 0 & $x^3$ & $x^2$ & 0 & 0 & 0 & 0 & $x^3$ & 0 & 0 \\  \hline
$d_{27, \alpha_6}^{1, 1}$&   0 & 0 & 0 & 0 & $x$ & $x$ & 0 & 0 & 0 & 0& $x$ & 0 & 0 & $x$ & $x$ & 0 & 0 & 0 & 0 & $x$ & $x$ & 0 & 0 & 0 & 0 & $x^2$ & $x^3$ & 0 & 0 & 0 & 0 & $x^3$ & 0 & 0 \\  \hline
 $d_{28, \alpha_6}^{1, 1}$&  0 & 0 & 0 & 0 & 0 & 0 & 0 & 0 & $x$ & $x$& 0 & $x$ & 0 & $x$ & $x$ & 0 & 0 & $x$ & $x$ & 0 & 0 & $x$ & $x$ & 0 & 0 & 0 & 0 & $x^3$ & $x^2$ & 0 & 0 & 0 & $x^3$ & 0 \\  \hline
 $d_{29, \alpha_6}^{1, 1}$&  0 & 0 & 0 & 0 & $x$ & $x$ & 0 & 0 & 0 & 0& 0 & $x$ & 0 & $x$ & $x$ & 0 & 0 & 0 & 0 & 0 & 0 & $x$ & $x$ & 0 & 0 & 0 & 0 & $x^2$ & $x^3$ & 0 & 0 & 0 & $x^3$ & 0 \\  \hline
 $d_{30, \alpha_6}^{1, 1}$&  0 & 0 & 0 & 0 & 0 & 0 & 0 & 0 & $x$ & $x$& 0 & 0 & $x$ & 0 & 0 & 0 & 0 & $x$ & $x$ & 0 & 0 & 0 & 0 & $x$ & $x$ & 0 & 0 & 0 & 0 & $x^3$ & $x^2$ & 0 & 0 & $x^3$ \\  \hline
 $d_{31, \alpha_6}^{1, 1}$&  0 & 0 & 0 & 0 & 0 & 0 & $x$ & $x$ & 0 & 0& 0 & 0 & $x$ & 0 & 0 & $x$ & $x$ & 0 & 0 & 0 & 0 & 0 & 0 & $x$ & $x$ & 0 & 0 & 0 & 0 & $x^2$ & $x^3$ & 0 & 0 & $x^3$ \\  \hline \hline
 $d_{32, \alpha_7}^{2, 0}$&  1 & 1 & 1 & 1 & $x$ & $x$ & $x$ & $x$ & 0 & 0& $x$ & 0 & 0 & $x^2$ & $x^2$ & $x^2$ & $x^2$ & 0 & 0 & $x^2$ & $x^2$ & 0 & 0 & 0 & 0 & $x^3$ & $x^3$ & 0 & 0 & 0 & 0 & $x^4$ & 0 & 0 \\  \hline
 $d_{33, \alpha_7}^{2, 0}$&  1 & 1 & 1 & 1 & $x$ & $x$ & 0 & 0 & $x$ & $x$& 0 & $x$ & 0 & $x^2$ & $x^2$ & 0 & 0 & $x^2$ & $x^2$ & 0 & 0 & $x^2$ & $x^2$ & 0 & 0 & 0 & 0 & $x^3$ & $x^3$ & 0 & 0 & 0 & $x^4$ & 0 \\  \hline
 $d_{34, \alpha_7}^{2, 0}$&  1 & 1 & 1 & 1 & 0 & 0 & $x$ & $x$ & $x$ & $x$& 0 & 0 & $x$ & 0 & 0 & $x^2$ & $x^2$ & $x^2$ & $x^2$ & 0 & 0 & 0 & 0 & $x^2$ & $x^2$ & 0 & 0 & 0 & 0 & $x^3$ & $x^3$ & 0 & 0 & $x^4$\\  \hline
\end{tabular}
}
\end{landscape}
}
\newpage
After applying the column operations and by Theorem \ref{T3.32} the matrix $\overrightarrow{G}_{2.1+0}^3$ reduces as follows:

$\overrightarrow{A}_{0, 0} \sim \widetilde{\overrightarrow{A}}_{0,0} = I_4, \overrightarrow{A}_{1, 1} \sim \widetilde{\overrightarrow{A}}_{1,1} = xI_9, \overrightarrow{A}_{2, 2} \sim \widetilde{\overrightarrow{A}}_{2, 2} = (x^2-x-2)I_{12} + (-2)I'_{12}. $

\NI where $I_n$ denotes $n \times n$ identity matrix and $I'_n$ denotes $n \times n$ off-diagonal matrix.

After applying the row and column operations, the matrix $\overrightarrow{A}_{\rho}$ is reduced as follows:\\

\tiny{$\overrightarrow{A}_{\rho} \sim $\begin{tabular}{|c|c|c|c|c|c|c|c|c|c|}
                   \hline
                    & $d_{26, \alpha_6}^{1, 1}$ & $d_{27, \alpha_6}^{1, 1}$ & $d_{28, \alpha_6}^{1, 1}$ & $d_{29, \alpha_6}^{1, 1}$ & $d_{30, \alpha_6}^{1, 1}$ & $d_{31, \alpha_6}^{1, 1}$ & $d_{32, \alpha_7}^{2, 0}$ & $d_{33, \alpha_7}^{2, 0}$ & $d_{34,\alpha_7}^{2, 0}$ \\\hline
                   $d_{26, \alpha_6}^{1, 1}$ & $x^3 - 3x$ &$x^2-x$& $0$ & $0$ & $0$ & $-2x$ & $-x^2+x$ & $0$ & $0$ \\\hline
                   $d_{27, \alpha_6}^{1, 1}$ & $x^2-x$ & $x^3-3x$ & $0$ & $-2x$ & $0$ & $0$ & $-x^2+x$ & $0$ & $0$ \\\hline
                   $d_{28, \alpha_6}^{1, 1}$ & $0$ & $0$ & $x^3 - 3x$ & $x^2-x$ & $-2x$ & $0$ & $0$ & $-x^2+x$ & $0$ \\\hline
                   $d_{29, \alpha_6}^{1, 1}$ & $0$ & $-2x$ & $x^2-x$ & $x^3 - 3x$ & $0$ & $0$ & $0$ & $-x^2+x$ & $0$ \\\hline
                   $d_{30, \alpha_6}^{1, 1}$ & $0$ & $0$ & $-2x$ & $0$ & $x^3-3x$ & $x^2-x$ & $0$ & $0$ & $-x^2+x$ \\\hline
                   $d_{31, \alpha_6}^{1, 1}$ & $-2x$ & $0$ & $0$ & $0$ & $x^2-x$ & $x^3-3x$ & $0$ & $0$ & $-x^2+x$ \\\hline
                   $d_{32, \alpha_7}^{2, 0}$ & $-x^2+x$ & $-x^2+x$ & $0$ & $0$ & $0$ & $0$ & $x^4-2x^3$ & $-2x^2+2x+8$ & $-2x^2+2x+8$ \\
                   &  &  &  &  &  &  & $-4x^2+5x+8$ &  &  \\\hline
                   $d_{33, \alpha_7}^{2, 0}$ & $0$ & $0$ & $-x^2+x$ & $-x^2+x$ & $0$ & $0$ & $-2x^2+2x+8$ & $x^4-2x^3$ & $-2x^2+2x+8$ \\
                    &  &  &  &  &  &  &  & $-4x^2+5x+8$ &  \\\hline
                   $d_{34, \alpha_7}^{2, 0}$ & $0$  & $0$ & $0$ & $0$ & $-x^2+x$ & $-x^2+x$ & $-2x^2+2x+8$ & $-2x^2+2x+8$ & $x^4-2x^3$ \\
                    &  &  &  &  &  &  &  &  & $-4x^2+5x+8$ \\
                   \hline
                 \end{tabular}
}\\

The entry $x^2-x$ in the above matrix cannot be eliminated while applying column operations since the following diagrams do not belong to $\overrightarrow{A}^{\mathbb{Z}_2}_{3}(x).$

\begin{center}
\includegraphics{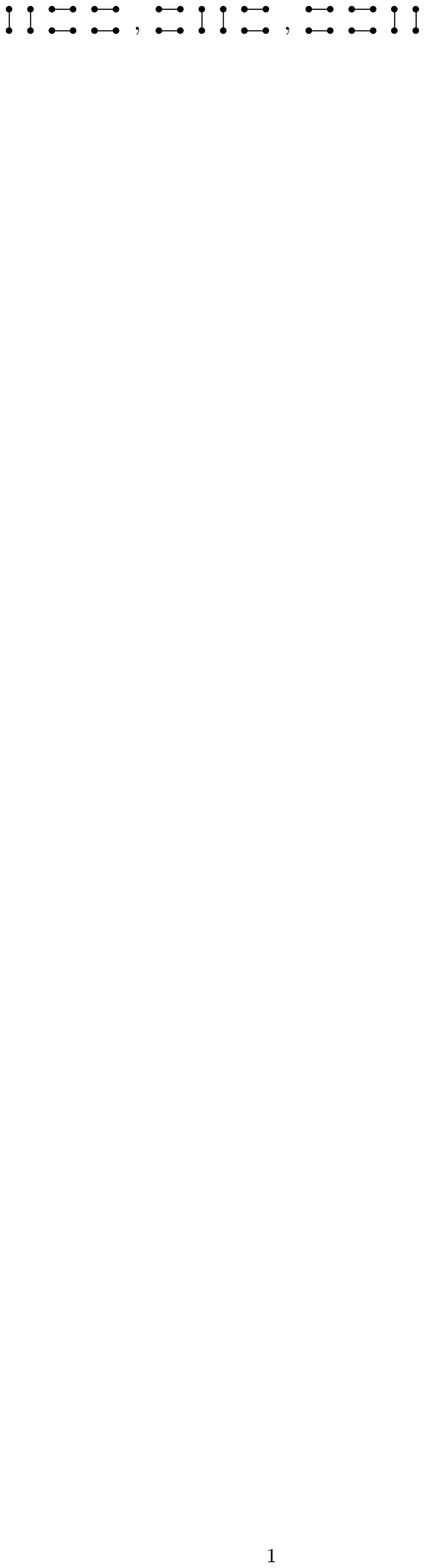}
\end{center}

\begin{center}
{\bf Acknowledgement}
\end{center}
 The authors would like to express their gratitude and sincere thanks to the  referee for all his(her) valuable commands and suggestions which in turn made the paper easy to read.


\begin{thebibliography}{PST}
\bibitem[1]{GL} J. J. Graham and G. I. Lehrer, {\it Cellular Algebras}, Inventiones Mathematicae, 123, 1 - 34(1996).
\bibitem[2]{HA} Arun Ram, Tom Halverson, {\it The partition algebras}, European Journal of electronics, {\bf 26}(2005), 869-921.
\bibitem[3]{J} V. F. R. Jones, The Potts model and the symmetric group, in {\it Subfactors (Kyuzeso, 1993)}, 259--267, World Sci. Publ., River Edge, NJ.
\bibitem[4]{K} S. K\"onig\ and\ C. Xi, When is a cellular algebra quasi-hereditary?, Math. Ann. {\bf 315} (1999), no.~2, 281--293.
\bibitem[5]{K1} N. Karimilla Bi, Cellularity of a larger class of diagram algebras, accepted for publication in Kyunpook Mathematical journal(arXiv:1506.02780).
\bibitem[6]{SP} M. Parvathi, Signed partition algebras, Comm. Algebra {\bf 32} (2004), no.~5, 1865--1880.
\bibitem[7]{PH} P. Martin\ and\ H. Saleur, Algebras in higher-dimensional statistical mechanics---the exceptional partition (mean field) algebras, Lett. Math. Phys. {\bf 30} (1994), no.~3, 179--185.

\bibitem[8]{PM1} P. P. Martin, Representations of graph Temperley-Lieb algebras, Publ. Res. Inst. Math. Sci. {\bf 26} (1990), no.~3, 485--503.

\bibitem[9]{PM2} P. Martin, Temperley-Lieb algebras for nonplanar statistical mechanics---the partition algebra construction, J. Knot Theory Ramifications {\bf 3} (1994), no.~1, 51--82.

\bibitem[10]{PM3} P. Martin, The structure of the partition algebras, J. Algebra {\bf 183} (1996), no.~2, 319--358.
\bibitem[11]{PM4} P. Martin,  The partition algebra and the Potts model transfer matrix spectrum in high dimension, J. Phys. A {\bf 33} (2000) 3669 -- 3695.
\bibitem[12]{PM5} P.P. Martin,  Potts models and related problems in statistical mechanics, Series on Advances in stastical Mechanics, vol.5, World Sientific Publishing Co. Inc., Teaneck NJ, 1991.
\bibitem[13]{PS} M. Parvathi, C. Selvaraj, Signed Brauer's algebra as centralizer algebras, comm. in Algebra {\bf 27(12)}, 5985-5998(1999).
\bibitem[14]{St} Richard P. Stanley, Enumerative combinatorics, Volume I, Cambridge studies in Advanced mathematics 49.
\bibitem[15]{VSS} V. Kodiyalam, R. Srinivasan \ and \ V. S. Sunder, The algebra of $G$-relations, Proc. Indian Acad. Sci. Math. Sci. {\bf 110} (2000), no.~3, 263--292.
\bibitem[16]{W} H. Wenzl, {\it Representations of Hecke algebras of type $A_n$ and subfactors}, Invent. Math 92, (1988), 349 - 383.
\bibitem[17]{X}C. Xi, Partition algebras are cellular, Compositio Math. {\bf 119} (1999), no.~1, 99--109.

\end{thebibliography}
\end{document}